\documentclass[UTF-8,reqno]{amsart}
\usepackage{enumerate}
\usepackage{mhequ}
\usepackage{amssymb,url,color, booktabs,nccmath}
\usepackage{verbatim}
\usepackage{comment}
\usepackage[left=3.2cm,right=3.2cm,top=4cm,bottom=4cm]{geometry}
\usepackage{mathrsfs}
\usepackage{enumitem,dsfont}
\usepackage{color}
\usepackage[colorlinks=true]{hyperref}
\hypersetup{
    linkcolor=blue,          % color of internal links
    citecolor=red,        % color of links to bibliography
    filecolor=blue,      % color of file links
    urlcolor=cyan}

\definecolor{darkergreen}{rgb}{0.0, 0.5, 0.0}

\newtheorem{theorem}{Theorem}[section]
\newtheorem{lemma}[theorem]{Lemma}
\newtheorem{proposition}[theorem]{Proposition}

\newtheorem{corollary}[theorem]{Corollary}

\newtheorem{definition}{Definition}
\theoremstyle{definition}
\newtheorem{remark}{Remark}

\allowdisplaybreaks

\usepackage{esint}
\usepackage{cite}
\usepackage{verbatim}
\usepackage{times}
\usepackage{color}
\usepackage{bm}
\usepackage{fancyhdr}
\numberwithin{equation}{section}
\begin{document}
%\renewcommand{\theequation}{\arabic{section}.\arabic{equation}}
%\begin{titlepage}
\title{High Order Asymptotic Expansion at Infinity for Strong Solutions to Incompressible Navier-Stokes Equations}
%\end{titlepage}
\author{Weiquan Chen}
\address[W. Chen]{Academy of Mathematics and Systems Science,
Chinese Academy of Sciences, Beijing 100190, P. R. China}
%\address[Weiquan Chen]{University of Chinese Academy of Science, Beijing 100190, China}
\email{chenweiquan@amss.ac.cn}

\author{Zhongmin Qian}
\address[Z. Qian]{Mathematical Institute, University of Oxford, Oxford OX2 6GG, UK}
\email{zhongmin.qian@maths.ox.ac.uk}

%\author{Shuai Xi}
%\address[Shuai Xi]{Shandong University of Technology, Shandong, China}
\author{Shuai Xi}
\address[S. Xi]{School of Mathematics and Systems Science, Shandong University of Science and Technology, Qingdao, 266590, P. R. China}
\email{shuaixi@sdust.edu.cn}

\begin{comment}
\author{Zhao Dong}
\address[Zhao Dong]{Academy of Mathematics and Systems Science,
Chinese Academy of Sciences, Beijing 100190, China}
\address[Zhao Dong]{University of Chinese Academy of Science, Beijing 100190, China}
\email{dzhao@amt.ac.cn}	
\end{comment}

%\thanks{Research  supported by National Key R\&D Program of China (No. 2020YFA0712700) and the NSFC (No. 11931004, 12090010, 12090014)}
%and the support by key Lab of Random Complex Structures and Data Science, Youth Innovation Promotion Association (2020003, No. 2008DP173182), Chinese Academy of Science.

%\textrm{Roman Family}
\begin{abstract}
	We discuss an interesting distinction between the incompressible Navier-Stokes equations (the velocity equations) and its vorticity form in whole space. We show that if the initial vorticity has a Gaussian bound then the bound is inherited up to the maximal lifespan of the strong solution. However, it turns out that the velocity equations don not share the same property. In fact, $L^p$-strong solutions to the velocity equations arising from ``well-localized" initial value generally behave at infinity like derivatives (of order $\geq3$) of the fundamental solution of Laplacian. To show this, a clean expansion up to maximal lifespan is derived :
	\begin{align}
		u(x,t)=-\nabla\sum_{|\alpha|=0}^{d-1}\frac{(-1)^{|\alpha|}}{\alpha !}\partial^\alpha\partial_{i,j}^2\Gamma(x)\int_0^t{\rm M}_{\alpha}^{i,j}(s){\rm d}s+O\big(|x|^{-2d-1}\big)\nonumber
	\end{align}
	where  ${\rm M}_{\alpha}^{i,j}(t):=\int_{\mathbb R^d}y^\alpha u^i(y,t)u^j(y,t){\rm d}y$ and $\Gamma$ is the fundamental solution of Laplacian. This improves the first order expansion given by L. Brandolese and F. Vigneron \cite{BV07}. 
	%As a consequence, a necessary and sufficient condition for strong solutions to have spacial decay faster than the critical $O\big(|x|^{-d-1}\big)$ is derive in terms of orthogonality.
\end{abstract}\
\maketitle

% SECTION ONE
\section{Introduction}\label{sec. introd.}
\setcounter{section}{1}

The incompressible Navier-Stokes equations play a central role in the field of Fluid Mechanics and Fluid Dynamics. It arises as a consequence of Newton's Second Law and the law of conservation of mass, governing evolution of the velocity field $u$ of the moving incompressible fluid:
\\
\begin{equation}\label{NS velocity}
	\left\{\begin{aligned}
		&\partial_t u + {\bf div}(u\otimes u) + \nabla p =\nu \Delta u,\\
		&{\bf div}\ u =0.\\
		%&u(0)=u_0 ,
	\end{aligned}\right.
\end{equation}
\\
Here $\nu>0$ is the kinematic viscosity constant. In principle, the equations (\ref{NS velocity}) are tricky in two aspects -- one is the non-linearity ${\rm div}(u\otimes u)$; the other is the appearance of unknown gradient pressure $\nabla p$. In dimension $d=2$ or $3$, one always formally ``gets rid of " the pressure term by taking ``${\bf curl}$" on both sides and study the evolution of the vorticity field $\omega:={\bf curl}\ u$ :
\\
\begin{equation}\label{NS vorticity}
	\begin{aligned}
		&\bullet d=2:\ \partial_t \omega + u\cdot\nabla\omega -\nu \Delta\omega=0 ;\\
		&\bullet d=3:\ \partial_t \omega + u\cdot\nabla\omega -\nu \Delta\omega=\omega\cdot\nabla u.
	\end{aligned}
\end{equation}
\\
In most of the existing literature, people study the equations (\ref{NS velocity}) and (\ref{NS vorticity}) individually as they share different characteristics   
in PDEs structure aspect. 

\par In this paper, we focus on the case where the two systems are posed on whole space $\mathbb R^d$ ($d\geq 2$ for (\ref{NS velocity}) and $d=2,3$ for (\ref{NS vorticity}) ). %with decaying condition given at infinity. 
We investigate a main difference between evolution of velocity $u$ and of vorticity $\omega$ arising respectively from Gaussian-localized initial value $u_0$ and $\omega_0={\bf curl}\ u_0$. In short language, we show that the vorticity equations (\ref{NS vorticity}) inherit Gaussian bounds, namely if $\omega_0$ satisfies a Gaussian bound point-wisely then it will be ``inherited" by the strong solution $\omega$ up to the maximal lifespan. 
\par However, due to the present of the gradient pressure term, this property can generally fail for the velocity equations (\ref{NS velocity}). In fact, it is known that at least for a short time, strong solutions to (\ref{NS velocity}) arising from ``small" and ``well-localized" initial values decay asymptotically like $O\big(|x|^{-d-1}\big)$ in general. On the decaying property of solutions, we have a variety of literature that we would encounter in the next sub-section. In this paper, we arrive at a deeper result that one can actually derive an explicit asymptotic expansion of strong solutions (generated by well-localized data) for $|x|\rightarrow\infty$ with fixed $t>0$. From this formula, one obtains a necessary and sufficient condition in terms of orthogonality for localized strong solutions to decay faster than the critical $O\big(|x|^{-d-1}\big)$ up to its maximal lifespan.

\subsection{Related history and comparison}

\subsubsection{Decaying of velocity}

The history of decaying property of solutions to Navier-Stokes system concerns both spacial and time decay. The direction is initiated by T. Kato \cite{Kat84} who proved existence of $L^p$-strong solutions ($d\leq p\leq\infty$) to (\ref{NS velocity}) in time-weighted spaces: the solutions $u$ are constructed so that if $u_0\in L^d$ then $t^{\frac{1-d/p}{2}}u,~t^{1-\frac{d}{2p}}\nabla u\in C_tL^p$. This implies global time decay in 2D and for small data in 3D. Later, M. E. Schonbek \cite{Sch95} derived decaying estimates for higher derivatives in 2D. The $L^2$-behavior $\|u(t)\|_2=O(t^{-\alpha})$ for Leray-Hopf weak solutions is in particular interest and it turns out that the rate exponent $\alpha$ has a critical value $\alpha=\frac{d+2}{4}$ :\\

\noindent$\bullet~{\bf Below~critical:}$\\
\cite{Sch85,MK86}: $\|u(t)\|_2\leq (1+t)^{-\frac{3}{2}(\frac{1}{r}-\frac{1}{2})}$ if $u_0\in L^1\cap L^r(\mathbb R^3)$ ($1\leq r<2$) has non-zero mean value;\\
\cite{Wie87}: $\|u(t)\|_2\leq (1+t)^\alpha$ for $\alpha<\frac{d+4}{2}$, if $u_0\in L^2$ satisfies $\|e^{t\Delta}u_0\|_2\leq (1+t)^\alpha$.\\
\noindent$\bullet~{\bf Critical:}$\\
\cite{Wie87}: $\|u(t)\|_2\leq (1+t)^{-\frac{d+4}{2}}$ if $u_0\in L^2$ satisfies $\int(1+|x|)|u_0|{\rm d}x<\infty$. This decaying rate is optimal in the sense that $\exists~u$ such that $\displaystyle\liminf_{t\rightarrow\infty}t^\frac{d+4}{2}\|u(t)\|_2>0$.\\

\noindent An interesting orthogonality type necessary and sufficient condition for fast decaying comes to stage: T. Miyakawa and M. E. Schonbek \cite{MS01} proved that $\|u(t)\|_2=o\big(t^{-\frac{d+4}{2}}\big)$ if and only if the initial value satisfies $\int_{\mathbb R^d}x_iu_0^j(x){\rm d}x=0$ and that
\begin{align}\label{orthogonality 1}
 \int_0^\infty\int_{\mathbb R^d}u^i(x,t)u^j(x,t){\rm d}x{\rm d}t=C\delta_{ij}
\end{align} 
for all $i,j=1,...,d$. And later L. Brandolese \cite{Bra04a} utilized the orthogonality characterization and successfully constructed solutions that achieve the super-critical $L^2$-decay. Such orthogonality type condition for fast energy decaying seems to be derived firstly in \cite{DS94}.

\par Point-wise type estimate $|u(x,t)|\lesssim(1+|x|)^{-\alpha}(1+t)^{-\beta/2}$ for ``small" strong solutions are firstly derived in \cite{Tak99} and \cite{AGSS00}. Like above, we also have a critical line $\alpha+\beta=d+1$ here. See \cite{Miy00} for the sub-critical and critical case ($\alpha+\beta\leq d+1$) where the smallness conditions are given on the Stokes flow $e^{t\Delta}u_0$~. For the super-critical case, L. Brandolese \cite{Bra04a} proved that under some suitable localization and smallness condition on the initial value, there exists global strong solution such that the above point-wise type estimate holds for $\alpha+\beta=d+3$. Here, being critical does not mean that the estimate is generally sharp for solutions generated by fast decaying initial values (for example $u_0(x)=O\big(|x|^{-d-1}\big)$), but just that it shares the same estimate as the Stokes flows. We also note here the results concerning decay estimates in weighted $L^p$-norm \cite{KT06,KT07}.

\par From the point-wise estimate, one recognizes $O\big(|x|^{-d-1}\big)$ as critical spacial decaying for strong solutions to (\ref{NS velocity}). Existence of strong solutions that achieve the critical spacial decaying uniformly in time was firstly proved in \cite{Miy00,HZ01} for ``small" and ``well-localized" initial values. In \cite{BM02} it is shown that the solution can decay like $o\big(|x|^{-d-1}\big)$ only if the orthogonality condition
\begin{align}
	\int_{\mathbb R^d}u^{i}(x,t)u^{j}(x,t){\rm d}x=c(t)\delta_{ij}\nonumber
\end{align}
holds during evolution. They also prove that, up to some sufficiently small time, this condition is also sufficient for the fast decaying. In the same paper, L. Brandolese and Y. Meyer have constructed examples of solutions that satisfy the orthogonality condition for all time. Later, by using invariance under rotations, L. Brandolese \cite{Bra04b} gives a very interesting way to characterize subsets of initial values $u_0\in\mathcal S$ that can generate $O\big(|x|^{-n-1}\big)$-decaying ($n\geq d$) strong solutions. 

\subsubsection{Asymptotic expansions}
Studies on large time asymptotic expansion have an earlier history than spatial expansion. To the knowledge of the authors, the first asymptotic expansion result of the velocity equation (\ref{NS velocity}) refers to \cite{Pla98} where small global in time solutions are constructed in Besov spaces and they are shown to behave asymptotically like self-similar solutions as $t\rightarrow\infty$. Expansion in $L^p$-norm up to second order is given for small solutions in \cite{Cap96,FM01}. In two dimensions, Y. Giga and T. Kambe \cite{GK88} seek large time asymptotic expansion via vorticity equation (\ref{NS vorticity}) instead, they show expansion in $L^p$-norm with the Oseen vortex being the first approximation, see also \cite{GW02b,GW02a}.
\par Spacial asymptotic expansion came on the stage later when L. Brandolese and F. Vigneron \cite{BV07} first derived a clean asymptotic formula for strong solutions generated by well-localized data:
\begin{align}
	u(x,t)=e^{\nu t\Delta}u_0(x)-\left|\mathbb S^{d-1}\right|^{-1}\sum_{i,j}\nabla\left(\frac{d~x_ix_j-\delta_{i,j}|x|^2}{|x|^{d+2}}\right)\int_0^t\left\langle u^i(s),u^j(s)\right\rangle_2{\rm d}s+o_t\big(|x|^{-d-1}\big)~.\nonumber
\end{align}
This is exactly the first order term given in our Theorem \ref{Thm. Velocity Asymptotic Expansion}. This formula is interesting in its style of separate variables and it says that localized strong solutions behave asymptotically like a potential field which is linear combination of gradient of second derivatives of the fundamental solution of Laplacian. The method that they use relies on a key decomposition of the non-linearity 
\begin{align}\label{decomposition of non-linearity}
	(u\otimes u)(x,t)=\left[\int_{\mathbb R^d}(u\otimes u)(y,t){\rm d}y\right]G(x)+r(x,t)
\end{align}
which is inspired by the similar decomposition used in the proof of $L^2$-lower bound estimate\cite{Sch91}. Later, by considering the equation of the vorticity tenor, I. Kukavica and E. Reis \cite{KR11} established an expansion of arbitrary order, which is given in terms of Fourier inverse and cannot be written explicitly in $x$. Recently, L. Brandolese \cite{Bra22} also obtained an arbitrary order expansion for 2D case in polar coordinate of the form
\begin{equation}
	\sum_{n=1}^m\frac{A_n(t)}{r^n}\left(
	\begin{aligned}
		&\cos\big(n\theta+\phi_n(t)\big)\\
		&\sin\big(n\theta+\phi_n(t)\big)
	\end{aligned}\right)+o\big(r^{-m}\big)~.\nonumber
\end{equation}
This is due to a similar expansion of the Biot-Savart law and can be possibly extended to 3D case for small enough time. More recently, R. McOwen and P. Topalov \cite{MT24,Top25} derive a similar expansion for arbitrary dimensions by establishing local well-posedness in a class of weighted Sobolev and asymptotic spaces. Their principal part is essentially given by $\displaystyle\sum_{n}\frac{a_n(\theta,t)}{r^n}$ where for each $t$, the function $a_n(\cdot,t)$ is eigenfunction of the Laplace-Beltrami operator on $\mathbb S^{d-1}$ with eigenvalue $n(n-d+2)$. \\

\subsubsection{Comparison and main idea}
\par We improve the first order expansion given in \cite{BV07}: the number of terms of our expansion equals the dimension $d$ and the expansion holds up to the maximal lifespan of the strong solution instead of some small time. 
\par Our method is different from \cite{BV07} and do not rely on the classical decomposition (\ref{decomposition of non-linearity}). The main idea is to make use of a mild solution representation which involves second derivatives of the Poisson potential instead of Riesz transforms, see (\ref{mild solution formulation of u}) and (\ref{decomposition of B(F) I}). The representation is derived by integration by part where ``safe" boundary term is generated by the singularity of the Poisson potential. Next, we split the principal value convolution $\partial_i\partial_j\Gamma\ast\big[u^i(t)u^j(t)\big]$ with respect to several different regions of integration, then in the safe region we expand $\partial_i\partial_j\Gamma$ by Taylor expansion and prove that all of the others only contribute higher order decaying. Finally the problem is reduced to lemma \ref{lem. expansion of G^K} which is essentially expanding the convolution $\nabla G_{\nu\tau}\ast\partial^\alpha\partial_i\partial_j\Gamma$. It is worth noting that lemma \ref{lem. expansion of G^K} is of its own interest in the sense that due to a delicate cancellation in the final computation, only first order term of the expansion survives. %so that we deduce a very clean formula.
\par To make all the above arguments work, we need certain spacial decaying estimate of the strong solution $u$. Especially, we have to treat very carefully on the region where the convolution integral has to be understood in the sense of principal value. So in order to derive high order expansion, we also need decaying estimate of $\nabla u$. Although we have plenty of works concerning estimates of this type, but it seems that all of them are proved for small solutions or in general, up to some sufficiently small time. So to ensure that our expansion can hold until the maximal lifespan, we establish the following type estimate with some non-decreasing function $h$ :  
\begin{align}
	|u(x,t)|+\sqrt{\nu t}~|\nabla u(x,t)|\leq h(t)\langle x\rangle^{-p}~,\ \ (x,t)\in\mathbb R^d\times[0,T_{max})\nonumber
\end{align}
assuming certain localization only on the data $u_0$ itself. See Proposition \ref{prop. polynomial bound for U_n} and \ref{prop. polynomial bound for grad u}. The strategy used to iterated the estimate onto the whole interval $[0,T_{max})$ is the same as how we prove the Gaussian localization of the vorticity equations.

\par As it is the case for velocity, decaying estimates of vorticity equations were only derived previously for 2D case and for small solutions in 3D. In general case of 3D, the estimates are only known to hold up to some sufficiently small time. In another aspect, all existing literature concern only preserving of polynomial localization. For this kind of results, we refer to \cite{GW02b,GW02a} and \cite{KT06,KT07}. To the best knowledge of the authors, Theorem \ref{Thm. vorticity equations transport Gaussian bounds} below is the first result concerning Gaussian decay estimate of vorticity and is also the first spacial decaying result that known to hold up to the maximal lifespan of strong solution (with the possibility of blowing up in the end). The method can also be used to show similar type of estimates for other kind of localization, for example polynomial or exponential.

\
\subsection{Main results}

Throughout the paper, we use the notation for Gaussian function:
\begin{align}
	G_{t}(x):=\frac{1}{(4\pi t)^{d/2}}~e^{-\frac{|x|^2}{4t}}~,\ \ (x,t)\in\mathbb R^d\times(0,+\infty)
\end{align}
which is known as the fundamental solution to heat equation with heat constant one. For the equations (\ref{NS velocity}) and (\ref{NS vorticity}) posed on $\mathbb R^d$ with decaying condition at infinity, we would focus on initial values $u_0,\omega_0\in L^\infty$ satisfying a Gaussian bound (not necessarily simultaneously), i.e.
\begin{align}
	&|u_0(x)|\leq C~ G_{\nu\sigma}(x)~,\ \ a.e.~x\in\mathbb R^d \label{Gaussian bound u_0}\\
	&|\omega_0(x)|\leq \kappa~G_{\nu\sigma}(x)~,\ \ a.e.~x\in\mathbb R^d \label{Gaussian bound omega_0}
\end{align}
for some constants $C,\kappa,\sigma>0$. Note that in general (\ref{Gaussian bound u_0}) and (\ref{Gaussian bound omega_0}) do not imply that the initial values depend on viscosity $\nu>0$ as we do not require the other constants to be independent of $\nu$. We say that an initial value is Gaussian-localized if it satisfies one of such Gaussian bounds.
  
\begin{definition}\label{def. solu. to velocity equations}
Let $d\leq p\leq\infty$. We say that $u\in C\left([0,T); L^p\right)$ is the strong/mild solution to (\ref{NS velocity}) arising from $u_0\in L^p$ if the mild solution formula  
\begin{equation}
	\begin{aligned}\label{mild solution formulation u}
		u(x,t)= G_{\nu t}\ast u_0(x)-\int_0^tG_{\nu (t-s)}\ast\mathcal{P}{\bf div}\big(u(s)\otimes u(s)\big)(x){\rm d}s
	\end{aligned}
\end{equation}
holds in $C\left([0,T); L^p\right)$ and for $a.e.~(x,t)\in\mathbb R^d\times[0,T)$. Here $\mathcal{P}={\bf Id}+\nabla(-\Delta)^{-1}{\bf div}$ denotes the Leray projector.
\end{definition}

\begin{definition}\label{def. solu. to voticity equations}
Let $d=2,3$ and $d/2\leq p\leq\infty$. We say that $\omega\in C\left([0,T); L^p\right)$ is a strong/mild solution to (\ref{NS vorticity}) arising from $\omega_0\in L^p$ if the mild solution formula  
\begin{equation}
	\begin{aligned}\label{mild solution formulation omega}
		\omega(x,t)= G_{\nu t}\ast \omega_0(x)+\int_0^tG_{\nu (t-s)}\ast{\bf div}\big(-u(s)\otimes \omega(s)+\omega(s)\otimes u(s)\big)(x){\rm d}s
	\end{aligned}
\end{equation}
holds in $C\left([0,T); L^p\right)$ and for $a.e.~(x,t)\in\mathbb R^d\times[0,T)$.
\end{definition}

\begin{remark}
	Initial values satisfying decaying bound (\ref{Gaussian bound u_0}) or (\ref{Gaussian bound omega_0}) obviously lay in space $\bigcap_{p=1}^\infty L^p$. Existence and uniqueness of strong solutions (up to some $T>0$) arising from such initial values is widely known and actually belongs to $C\left([0,T); L^p\right)$ for every large enough $1\leq p\leq\infty$ (for velocity we refer to \cite{FJR72,Kat84} and for vorticity \cite{GW02b,GW02a}). As a typical example of bounded initial values satisfying (\ref{Gaussian bound u_0}) or (\ref{Gaussian bound u_0}), one could simply take a $C^0_c(\mathbb R^d)$ function.
\end{remark}

Our first theorem tells that the vorticity equations inherit Gaussian bound from initial value up to maximal lifespan of strong solution.

\begin{theorem}[Vorticity Inherits Gaussian Bounds]\label{Thm. vorticity equations transport Gaussian bounds}
	Let $d=2,3$ and $\omega$ be the strong solution to (\ref{NS vorticity}) generated by $\omega_0\in L^\infty$ satisfying (\ref{Gaussian bound omega_0}), with maximal lifespan $0<T_{max}\leq\infty$. Then given any $0<\delta\leq1$, the following bound holds for all $(x,t)\in\mathbb R^d\times[0,T_{max})$:
	\begin{align}\label{Gaussian bound for vorticity}
		|\omega(x,t)|\lesssim\kappa\left(C_\delta\frac{t~\|u\|_{\infty,t}^2}{\nu}\right)^{C_\delta\frac{t\|u\|_{\infty,t}^2}{\nu}}G_{(1+\delta)\nu(\sigma+t)}(x)~.
	\end{align}
	Here $\|u\|_{\infty,t}$ stands for $\|u\|_{L^\infty((0,t)\times\mathbb R^d)}$ and $C_\delta>0$ depends only on $\delta$. The implicit constant is universal.
\end{theorem}

\begin{remark}
	We cannot take the limit $\delta\rightarrow0$ here as $C_\delta$ would explode in the limit.
\end{remark}

Our second theorem derive an explicit asymptotic expansion at infinity for strong solutions to the velocity equations (\ref{NS velocity}), with clear formula of the coefficients. In particular, we see that the velocity equations (\ref{NS velocity}) however do not inherit Gaussian decay from initial values in general. As one sees, it breaks the initial Gaussian bound as long as $t>0$, by generating rational function terms which are homogeneous with leading degree $-(d+1)$. This marks the ``critical" order of spatial decay for velocity. The rational terms are linear combination of derivatives of the fundamental solution $\Gamma$ of Laplacian (i.e. $-\Delta\Gamma=\delta_0$). The coefficients are the moment integrals
$$\displaystyle {\rm M}_{\alpha}^{i,j}(t):=\int_{\mathbb R^d}y^\alpha u^i(y,t)u^j(y,t){\rm d}y~,\quad\alpha\in\mathbb N_0^{d}.$$

\begin{theorem}[Velocity Asymptotic Expansion]\label{Thm. Velocity Asymptotic Expansion}
	Let $d\geq2$ and $u$ be the strong solution to (\ref{NS velocity}) generated by $u_0\in L^\infty$ satisfying (\ref{Gaussian bound u_0}), with maximal lifespan $0<T_{max}\leq\infty$. Then for any $(x,t)\in{\rm B}_1(0)^c\times(0,T_{max})$ we have :
	\begin{align}%\label{general velocity asymptotic expansion}
		u(x,t)=-\sum_{|\alpha|=0}^{d-1}\frac{(-1)^{|\alpha|}}{\alpha !}\nabla\partial^\alpha{\rm K}_{i,j}(x)\int_0^t{\rm M}_{\alpha}^{i,j}(s){\rm d}s+{\rm R}(x,t)\nonumber
	\end{align}
	where ${\rm K}_{i,j}:=\partial_{i,j}^2\Gamma$ and we have used Einstein summation convention for the indices $i,j$. The remainder satisfies 
	\begin{align}
		\sup_{t\in[0,T]}|{\rm R}(x,t)|=O\left(|x|^{-2d-1}\right)\ when\ |x|\rightarrow\infty\nonumber
	\end{align}
	for any $[0,T]\subset[0,T_{max})$.
\end{theorem}

\begin{remark}
	What we prove in section \ref{sec. velocity equations} is stronger than Theorem \ref{Thm. Velocity Asymptotic Expansion} in the sense that the same result actually applies to a wider range of initial values satisfying certain polynomial decaying. See Proposition \ref{prop. generalized velocity expansion} and remark \ref{remark. derivation of Thm Two}.
\end{remark}

From Theorem \ref{Thm. Velocity Asymptotic Expansion} one can derive the following necessary and sufficient condition for strong solutions arising from Gaussian-localized initial value to have super-critical decaying order at infinity.

\begin{corollary}\label{corollary orthogonality}
	Let $d\geq2$ and $u$ be strong solution to (\ref{NS velocity}) with maximal lifetime $T_{max}>0$ and initial value $u_0\in L^\infty$ satisfying (\ref{Gaussian bound u_0}). Then for any fixed $t\in(0,T_{max})$,
	\begin{align}
		u(x,t)=o_t\left(|x|^{-d-1}\right)\quad as\quad |x|\rightarrow\infty\nonumber
	\end{align}
	if and only if
	\begin{align}\label{orthogonal condition for super-critical decaying}
		%\int_0^t\int_{\mathbb R^d}u^i(y,s)u^j(y,s){\rm d}y{\rm d}s=c(t)\delta_{i,j}\nonumber
		\big\langle u^i,u^j\big\rangle_{2,t}=\frac{\|u\|_{2,t}^2}{d}~\delta_{i,j}~.
	\end{align}
	Here $\langle f,g\rangle_{2,t}:=\langle f,g\rangle_{L^2((0,t)\times\mathbb R^d)}$ denotes the common $L^2$-inner product.
\end{corollary}
\begin{proof}[Proof of corollary \ref{corollary orthogonality}]
	Fixed any $t\in(0,T_{max})$. We define the $d\times d$ symmetric matrix $A=\big(a_{ij}\big)$ by $a_{ij}:=\big\langle u^i,u^j\big\rangle_{2,t}$.
	Then by Theorem \ref{Thm. Velocity Asymptotic Expansion}, $\displaystyle\lim_{|x|\rightarrow\infty}\Big(|x|^{d+1}|u(x,t)|\Big)=0$ is equivalent to that
	\begin{align}\label{super decaying condition 1}
		\lim_{|x|\rightarrow\infty}\Bigg(|x|^{d+1}\nabla\underbrace{\sum_{i,j=1}^da_{ij}\partial_{i,j}^2\Gamma(x)}_{=:F(x)}\Bigg)=0~.
	\end{align}
	A little calculation gives that $\partial_{i,j}^2\Gamma(x)=\frac{|x|^{-d}}{|\mathbb S^{d-1}|}\left(d\frac{x_ix_j}{|x|^2}-\delta_{i,j}\right)$ and so\\
	\begin{align}
		F(x)=\frac{dx^TAx-({\rm tr}A)|x|^2}{|\mathbb S^{d-1}||x|^{d+2}}~.\noindent
	\end{align}\\
	\noindent Then it is not hard to see that $\nabla F(x)$ is homogeneous of degree $-(d+1)$, i.e. $\nabla F(\lambda x)=\lambda^{-d-1}\nabla F(x)$ for any scalar $\lambda$. Hence, in spherical coordinate $r=|x|$, $\omega=\frac{x}{|x|}$, (\ref{super decaying condition 1}) is equivalent to $\nabla F(\omega)=0$ for $\omega\in\mathbb S^{d-1}$. But again by the homogeneity, this is equivalent to that $\nabla F(x)\equiv0$, i.e. $F(x)=const$. As $F$ decays to zero at infinity, this is in turn equivalent to that $F(x)\equiv0$, i.e.
	\begin{align}
		dx^TAx-({\rm tr}A)|x|^2=0~,\quad i.e.\quad x^T\left(A-\frac{{\rm tr}A}{d}{\bf Id}\right)x=0\quad for~\forall~x\in\mathbb R^d.\nonumber
	\end{align}
	This is to say $A=\frac{{\rm tr}A}{d}{\bf Id}$ which is exactly (\ref{orthogonal condition for super-critical decaying}).
\end{proof}
\

\section{Vorticity Equations Inherit Gaussian Bounds}
\

%\subsection{Starting from Gaussian-bound initial values}
\par This subsection is devoted to Theorem \ref{Thm. vorticity equations transport Gaussian bounds}. The key idea is to utilize the boundedness of strong solutions arising from localized initial values, which reduce the problem to a linear iteration and then apply an inductive estimate using the simple semi-group property of heat kernel:
\begin{align}\label{semi-group property of G}
	G_{t+s}(x)=G_t\ast G_s(x)=\int_{\mathbb R^d}G_t(x-y)G_s(y){\rm d}y~.
\end{align}
\par  Obviously $\omega_0$ satisfying (\ref{Gaussian bound omega_0}) lies in $L^q$ for every $1\leq q\leq\infty$. Following the same argument of Kato \cite{Kat84}, it is not hard to show there exists unique local solution $\omega\in C_tL^q$ ($\forall d/2\leq q\leq \infty$), see also \cite{GW02a}. By Sobolev imbedding and boundedness of Riesz transforms, we obviously have that $\|u\|_{L^\infty_tL^\infty}\lesssim \|\nabla u\|_{L^\infty_t L^{d+1}}\lesssim\|\omega\|_{L^\infty_t L^{d+1}}<\infty$. On the other hand, if $\omega\in C\big([0,T];L^{d+1}\big)$ is the mild solution generated by $\omega_0\in \cap_{q}~L^q$, then $u=K_{SB}\ast\omega\in L^\infty\big(0,T;L^\infty\big)$ and so by a similar argument of the proof of Proposition \ref{prop. convergence of Omega_n}, one would see that $\omega\in C\big([0,T];L^q\big)$ for all $d/2\leq q\leq \infty$.
%then by the L-P-S analog for vorticity, $\omega$ is smooth on $(0,T)\times\mathbb R^d$ and so $u=K_{SB}\ast\omega\in L^\infty\big(0,T;L^\infty\big)$ is the mild solution to (\ref{NS velocity}) generated by $u_0=K_{SB}\ast\omega_0\in L^\infty$.
%And in fact, if $u\in C\big([0,T];L^\infty\big)$ is the mild solution to (\ref{NS velocity}) with $u_0\in L^\infty$, then it can be shown that $\omega\in L^1(0,T;L^\infty)$ (see Appendix \ref{Appendix: regularity of strong solutions}). By BKM blow-up criterion \cite{BKM84}, the vorticity remains regular at least on $[0,T]$. 
So it is reasonable to define the maximal lifespan $T_{max}>0$ by
$$T_{max}:=\sup\Big\{T>0~\Big|~ \exists!~solution~\omega\in C\big([0,T];L^{d+1}\big)~to~(\ref{NS vorticity})~with~\omega(0)=\omega_0\Big\}~.$$
%$$T_{max}:=\sup\big\{T>0:\|K_{BS}\ast\omega\|_{L^\infty(0,T;L^\infty)}<\infty\big\}~.$$ 
\par To begin with, let $\omega$ be a strong solution to (\ref{NS vorticity}) with maximal lifespan $0<T_{max}\leq\infty$ and initial value $\omega_0\in L^\infty$ satisfying (\ref{Gaussian bound omega_0}).
Now we define the sequence $\Omega_n:[0,T_{max})\times\mathbb R^d\rightarrow\mathbb R^d$ by :
\begin{align}
	\Omega_0(t)
	&:=G_{\nu t}\ast\omega_0\nonumber\\
	\Omega_{n+1}(t)
	&:=\Omega_0(t)+\int_0^t\partial_\ell G_{\nu(t-s)}\ast\big
	[-u^\ell(s)\Omega_n(s)+\Omega_n^\ell(s)u(s)\big]{\rm d}s~,\quad n\geq0.\label{definition of Omega_n}
\end{align}
Then the following two propositions imply Theorem \ref{Thm. vorticity equations transport Gaussian bounds}.\\

\begin{proposition}\label{prop. convergence of Omega_n}
	Let $d=2,3$ and also $\Omega_n$ be given by (\ref{definition of Omega_n}). Then for any $0<T<T_{max}$ and $d/2\leq q\leq\infty$, we have the strong convergence $\Omega_n\longrightarrow\omega$ in $L^\infty(0,T;L^q)$.
\end{proposition}

\begin{proof}[Proof of Proposition \ref{prop. convergence of Omega_n}]
	Let $0<T<T_{max}$ be arbitrarily given. We show the $L^\infty(0,T;L^q)$-convergence for any fixed $d/2\leq q\leq\infty$. Just write for any $0\leq\tau<t\leq T$ that
	 \begin{align}
	 	\Omega_{n+1}(t)=\Omega_{n+1}(\tau)
	 	&+\left(G_{\nu t}- G_{\nu\tau}\right)\ast\omega_0+\int_{\tau}^t\partial_\ell G_{\nu(t-s)}\ast\big[-u^\ell(s)\Omega_n(s)+\Omega_n^\ell(s)u(s)\big]{\rm d}s\nonumber\\
	 	&+\int_0^{\tau}\left[\partial_\ell G_{\nu(t-s)}-\partial_\ell G_{\nu(\tau-s)}\right]\ast\big[-u^\ell(s)\Omega_n(s)+\Omega_n^\ell(s)u(s)\big]{\rm d}s~;\nonumber\\
	 	\omega(t)=\omega(\tau)
	 	&+\left(G_{\nu t}- G_{\nu\tau}\right)\ast\omega_0 +\int_{\tau}^t\partial_\ell G_{\nu(t-s)}\ast\big[-u^\ell(s)\omega(s)+\omega^\ell(s)u(s)\big]{\rm d}s\nonumber\\
	 	&+\int_0^{\tau}\left[\partial_\ell G_{\nu(t-s)}-\partial_\ell G_{\nu(\tau-s)}\right]\ast\big[-u^\ell(s)\omega(s)+\omega^\ell(s)u(s)\big]{\rm d}s~,\nonumber
	 \end{align}
	 and then take the difference and estimate $L^q$-norm using Young's inequality. Since we have the $L^1$-estimate
	 \begin{align}
	 	\big\|\partial_\ell G_{\nu(t-s)}-\partial_\ell G_{\nu(\tau-s)}\big\|_1
	 	&\leq\int_{\tau-s}^{t-s}\int_{\mathbb R^d}\frac{\nu}{2(\nu\theta)^{3/2}}\left|P\left(\frac{x}{\sqrt{\nu\theta}}\right)\right|G_{\nu\theta}(x){\rm d}x{\rm d}\theta\nonumber\\
	 	&\lesssim_d \nu^{-1/2}\int_{\tau-s}^{t-s}\frac{{\rm d}\theta}{\theta^{3/2}}\nonumber\\
	 	&\lesssim_d \nu^{-1/2}\left(\frac{1}{\sqrt{\tau-s}}-\frac{1}{\sqrt{t-s}}\right)~,\quad s<\tau<t\nonumber
	 \end{align}
	 where $P^\ell(x):=x_\ell\left(\frac{d}{2}+1-\frac{|x|^2}{4}\right)$ , this would gives that
	 \begin{align}
	 	\big\|\Omega_{n+1}(t)-\omega(t)\big\|_q
	 	\leq &\big\|\Omega_{n+1}(\tau)-\omega(\tau)\big\|_q+ C_d \|u\|_{\infty,t}\int_{\tau}^t\frac{\big\|\Omega_{n}(s)-\omega(s)\big\|_q}{\sqrt{\nu(t-s)}}{\rm d}s\nonumber\\
	 	&+C_d' \|u\|_{\infty,\tau}\int_0^{\tau}\left(\frac{1}{\sqrt{\nu(\tau-s)}}-\frac{1}{\sqrt{\nu(t-s)}}\right)\big\|\Omega_{n}(s)-\omega(s)\big\|_q{\rm d}s\nonumber\\
	 	\leq &\big\|\Omega_{n+1}(\tau)-\omega(\tau)\big\|_q+ 2C_d \|u\|_{\infty,t}\sqrt{\frac{t-\tau}{\nu}}~\sup_{s\in[\tau,t]}\big\|\Omega_{n}(s)-\omega(s)\big\|_{q}\nonumber\\
	 	&+2C_d' \|u\|_{\infty,\tau}\left(\sqrt{\frac{t-\tau}{\nu}}-\sqrt{\frac{t}{\nu}}+\sqrt{\frac{\tau}{\nu}}\right) \sup_{s\in[0,\tau]}\big\|\Omega_{n}(s)-\omega(s)\big\|_{q}\nonumber
	 \end{align}
	 where $L^\infty_t L^\alpha$ ($1\leq\alpha\leq\infty$) stands for $L^\infty(0,t;L^\alpha)$. Let's denote $B_d:=max\big\{2C_d,2C_d'\big\}$. Then from this, we deduce for any interval $[\tau,t]\subset[0,T]$ that
	 \begin{align}
	 	\sup_{s\in[\tau,t]}\big\|\Omega_{n+1}(s)-\omega(s)\big\|_q
	 	\leq &\big\|\Omega_{n+1}(\tau)-\omega(\tau)\big\|_q+ B_d \|u\|_{\infty,T}\sqrt{\frac{t-\tau}{\nu}}~\sup_{s\in[\tau,t]}\big\|\Omega_{n}(s)-\omega(s)\big\|_{q}\nonumber\\
	 	&+B_d \|u\|_{\infty,T}\sqrt{\frac{t-\tau}{\nu}}~\sup_{s\in[0,\tau]}\big\|\Omega_{n}(t)-\omega(t)\big\|_{q}	~.\label{convergence estimate 1}
	 \end{align}
	 Now we choose any $T'\in(0,T]$ such that
	 $$\Lambda:=B_d \|u\|_{\infty,T}\sqrt{\frac{T'}{\nu}}<1~.$$
	 Then set $[\tau,t]=[0,T']$ in (\ref{convergence estimate 1}) and remind that $\Omega_{n}(0)=\Omega_{0}(0)=\omega_0=\omega(0)$, one derives:
	 \begin{align}
	 	\sup_{t\in[0,T']}\big\|\Omega_{n+1}(t)-\omega(t)\big\|_{q}
	 	&\leq \Lambda~\sup_{t\in[0,T']}\big\|\Omega_{n}(t)-\omega(t)\big\|_{q}\nonumber\\
	 	&\leq...\leq\Lambda^{n+1}\sup_{t\in[0,T']}\big\|\Omega_{0}(t)-\omega(t)\big\|_{q}~,\quad n\geq0.\label{convergence estimate on [0,T']}
	 \end{align}
	 which implies convergence on the interval $[0,T']$. Now if $T'<T$, we go to a further interval by setting $[\tau,t]=[T',2T']$ in (\ref{convergence estimate 1}), then together with (\ref{convergence estimate on [0,T']}) we see that
	 \begin{align}
	 	\sup_{t\in[T',2T']}\big\|\Omega_{n+1}(t)-\omega(t)\big\|_{q}
	 	&\leq \Lambda~\sup_{t\in[T',2T']}\big\|\Omega_{n}(t)-\omega(t)\big\|_{q}+2\Lambda^{n+1}\sup_{t\in[0,T']}\big\|\Omega_{0}(t)-\omega(t)\big\|_{q}~,\nonumber
	 \end{align}
	 i.e.
	 \begin{align}
	 	&\sup_{t\in[T',2T']}\big\|\Omega_{n+1}(t)-\omega(t)\big\|_{q}-2(n+1)\Lambda^{n+1}\sup_{s\in[0,T']}\big\|\Omega_{0}(s)-\omega(s)\big\|_{q}\nonumber\\
	 	\leq &\Lambda\left(\sup_{t\in[T',2T']}\big\|\Omega_{n}(t)-\omega(t)\big\|_{q}-2n\Lambda^{n}\sup_{t\in[0,T']}\big\|\Omega_{0}(t)-\omega(t)\big\|_{q}\right),\nonumber
	 \end{align}
	 i.e.
	 \begin{align}
	 	\sup_{t\in[T',2T']}\big\|\Omega_{n+1}(t)-\omega(t)\big\|_{q}
	 	\leq \Lambda^{n+1}\sup_{t\in[T',2T']}\big\|\Omega_{0}(t)-\omega(t)\big\|_{q}+2(n+1)\Lambda^{n+1}\sup_{t\in[0,T']}\big\|\Omega_{0}(t)-\omega(t)\big\|_{q}\nonumber
	 \end{align}
	 which implies the bound
	 \begin{align}
	 	\sup_{t\in[0,2T']}\big\|\Omega_{n+1}(t)-\omega(t)\big\|_{q}\leq (2n+3)~\Lambda^{n+1}\sup_{t\in[0,2T']}\big\|\Omega_{0}(t)-\omega(t)\big\|_{q}~.\nonumber
	 \end{align}
	 Iterating the process, then one would get for step $m:=\big\lfloor T/T'\big\rfloor$ that
	 \begin{align}
	 	\sup_{t\in[0,mT']}\big\|\Omega_{n+1}(t)-\omega(t)\big\|_{q}\leq Q_{m-1}(n)~\Lambda^{n+1}\sup_{t\in[0,mT']}\big\|\Omega_{0}(t)-\omega(t)\big\|_{q}\nonumber
	 \end{align}
	 for some polynomial $Q_{m-1}$ of order $m-1$. And for the final step, one may just set $[\tau,t]=[mT',T]$ in (\ref{convergence estimate 1}) and use the fact that $T-mT'<T'$. Exactly the same argument would then lead us to
	 \begin{align}
	 	\sup_{t\in[0,T]}\big\|\Omega_{n+1}(t)-\omega(t)\big\|_{q}\leq Q_{m}(n)~\Lambda^{n+1}\sup_{t\in[0,T]}\big\|\Omega_{0}(t)-\omega(t)\big\|_{q}\nonumber
	 \end{align}
	 which implies convergence on the whole interval $[0,T]$.
\end{proof}\

\begin{proposition}\label{prop. Gaussian bound of Omega_n}
	Let $d=2,3$ and also $\Omega_n$ be given by (\ref{definition of Omega_n}). Then for any $0<\delta\leq1$ and $0<T<T_{max}$ , the following estimate holds for all $n\in\mathbb N_0$ and $(x,t)\in\mathbb R^d\times[0,T]$ :
	\begin{align}
		\big|\Omega_n(x,t)\big|\lesssim_{d}\kappa\left(C_\delta\frac{T~\|u\|_{\infty,T}^2}{\nu}\right)^{C_\delta\frac{T\|u\|_{\infty,T}^2}{\nu}}G_{(1+\delta)\nu(\sigma+t)}(x)~.\nonumber
	\end{align}
	Here $\|u\|_{\infty,T}$ stands for $\|u\|_{L^\infty((0,T)\times\mathbb R^d)}$ and $C_\delta>0$ is a constant that depends only on $\delta$.
\end{proposition}

\begin{comment}
\begin{remark}
	By boundedness of Biot-Savart law, we have that $\|u\|_{L^\infty((0,T)\times\mathbb R^d)}\lesssim \|\omega\|_{L^\infty(0,T;L^{d+1})}<\infty$.
\end{remark}\
\end{comment}

\begin{proof}[Proof of Proposition \ref{prop. Gaussian bound of Omega_n}]	 
	 We are going to apply similar idea that we use to treat the convergence. So for any $0\leq\tau<t\leq T$ we write
	 \begin{align}
	 	\Omega_{n+1}(t)=\Omega_{n+1}(\tau)
	 	&+\left(G_{\nu t}- G_{\nu\tau}\right)\ast\omega_0 +\int_{\tau}^t\partial_\ell G_{\nu(t-s)}\ast\big[-u^\ell(s)\Omega_n(s)+\Omega_n^\ell(s)u(s)\big]{\rm d}s\nonumber\\
	 	&+\int_0^{\tau}\left[\partial_\ell G_{\nu(t-s)}-\partial_\ell G_{\nu(\tau-s)}\right]\ast\big[-u^\ell(s)\Omega_n(s)+\Omega_n^\ell(s)u(s)\big]{\rm d}s,\quad n\geq0\nonumber
	 \end{align}
	 and $\Omega_{0}(t)=\Omega_{0}(\tau)
	 	+\left(G_{\nu t}- G_{\nu\tau}\right)\ast\omega_0$~.
	 Then for $n\geq0$ we could estimate the difference:
	 \begin{align}
	 	\big|\Omega_{n+1}(t)-\Omega_{n}(t)\big|
	 	\leq &\big|\Omega_{n+1}(\tau)-\Omega_{n}(\tau)\big|+2\|u\|_{\infty,t}\int_{\tau}^t\big|\nabla G_{\nu(t-s)}\big|\ast\big|\Omega_{n}(s)-\Omega_{n-1}(s)\big|{\rm d}s\nonumber\\
	 	&+2\|u\|_{\infty,\tau}\int_0^{\tau}\left|\nabla G_{\nu(t-s)}-\nabla G_{\nu(\tau-s)}\right|\ast\big|\Omega_{n}(s)-\Omega_{n-1}(s)\big|{\rm d}s~.\nonumber
	 \end{align}
	 Here we drop the notation of spacial variables $x$ for simplicity. Since for any fixed $\delta>0$, 
	 \begin{align}
	 	\big|\nabla G_{\nu (t-s)}(x)\big|
	 	&\leq\frac{|x|}{2\nu (t-s)}e^{-\frac{\delta}{1+\delta}\cdot\frac{|x|^2}{4\nu (t-s)}}(1+\delta)^{d/2}G_{(1+\delta)\nu (t-s)}(x)\nonumber\\
	 	&\leq \frac{C_\delta}{\sqrt{\nu (t-s)}}G_{(1+\delta)\nu (t-s)}(x)~;\nonumber\\
	 	\big|\nabla G_{\nu(t-s)}(x)-\nabla G_{\nu(\tau-s)}(x)\big|
	 	&\leq\int_{\tau}^{t}\frac{\nu}{2(\nu(\theta-s))^{3/2}}\left|P\left(\frac{x}{\sqrt{\nu(\theta-s)}}\right)\right|G_{\nu(\theta-s)}(x){\rm d}\theta\nonumber\\
	 	&\leq \int_{\tau}^{t}\frac{\widetilde C_\delta~\nu}{(\nu(\theta-s))^{3/2}}G_{(1+\delta)\nu(\theta-s)}(x){\rm d}x{\rm d}\theta\nonumber
	 \end{align}
	 with $\displaystyle C_\delta=\frac{(1+\delta)^{\frac{d+1}{2}}}{\sqrt{2e\delta}}$ and $\displaystyle\widetilde C_\delta=\frac{1+\delta}{2\delta}\sup_{y\geq0}\left[y\left(\frac{d}{2}+1-\frac{(1+\delta)y^2}{4\delta}\right)e^{-y^2/4}\right]$, we deduce that
	 \begin{align}
	 	&\big|\Omega_{n+1}(t)-\Omega_{n}(t)\big|\nonumber\\
	 	\leq &\big|\Omega_{n+1}(\tau)-\Omega_{n}(\tau)\big|+\frac{2C_\delta\|u\|_{\infty,t}}{\sqrt{\nu}}\int_{\tau}^t G_{(1+\delta)\nu (t-s)}\ast\big|\Omega_{n}(s)-\Omega_{n-1}(s)\big|\frac{{\rm d}s}{\sqrt{t-s}}\nonumber\\
	 	&+\frac{2\widetilde C_\delta\|u\|_{\infty,\tau}}{\sqrt{\nu}}\int_0^{\tau}{\rm d}s\int_{\tau}^t G_{(1+\delta)\nu (\theta-s)}\ast\big|\Omega_{n}(s)-\Omega_{n-1}(s)\big|\frac{{\rm d}\theta}{(\theta-s)^{3/2}}~.\label{Gaussian estimate 1}
	 \end{align}
	 For convenience, we may just set $\Omega_{-1}\equiv0$ so that the inequality also holds for $n=0$. Now we define the Gaussian-weighted supremum norm for any vector fields $f$ that would make sense:
	 \begin{align}
	 	\|f\|_{G(t)}:=\sup_{x\in\mathbb R^d}\Big(G_{(1+\delta)\nu(\sigma+t)}(x)^{-1}|f(x,t)|\Big)\nonumber
	 \end{align}
	 and we denote
	 \begin{align}
	 	\Xi_{n}(\tau,t):=\sup_{s\in[\tau,t]}\big\|\Omega_{n}-\Omega_{n-1}\big\|_{G(s)}~,\quad n\geq1\nonumber
	 \end{align}
	 provided that it is finite. Since by (\ref{semi-group property of G}),
	 \begin{align}
	 	\big|\Omega_0(x,t)\big|\leq \kappa~G_{\nu(\sigma+t)}(x)\leq (1+\delta)^{d/2}\kappa~G_{(1+\delta)\nu(\sigma+t)}(x)~,\nonumber
	 \end{align}
	 so we obviously have
	 \begin{align}
	 	\Xi_{0}(0,t)\leq (1+\delta)^{d/2}\kappa<\infty~,\quad \forall t\geq0.\nonumber
	 \end{align}
	 If we set $n=0$ and $\tau=0$ in (\ref{Gaussian estimate 1}), we would also have that
	 \begin{align}
	 	&\big|\Omega_1(x,t)-\Omega_0(x,t)\big|\nonumber\\
	 	\leq &\frac{2C_\delta\|u\|_{\infty,t}}{\sqrt{\nu}}\int_{\tau}^t G_{(1+\delta)\nu (t-s)}\ast\big|\Omega_{0}(s)\big|\frac{{\rm d}s}{\sqrt{t-s}}\nonumber\\
	 	\leq &\frac{2C_\delta\|u\|_{\infty,t}}{\sqrt{\nu}}~\Xi_{0}(0,t)\int_0^tG_{(1+\delta)\nu(t-s)}\ast G_{(1+\delta)\nu(\sigma+s)}(x)\frac{{\rm d}s}{\sqrt{t-s}}\nonumber\\
	 	\leq &4C_\delta\|u\|_{\infty,t}\sqrt{\frac{t}{\nu}}~\Xi_{0}(0,t)~G_{(1+\delta)\nu(\sigma+t)}(x)~.\nonumber
	 \end{align}
	 which implies
	 \begin{align}
	 	\Xi_{1}(0,t)\leq 4C_\delta\|u\|_{\infty,t}\sqrt{\frac{t}{\nu}}~\Xi_{0}(0,t)<\infty,\quad \forall t\geq0.\nonumber
	 \end{align}
	 By induction, it is not hard to see that for every $n\geq1$,
	 \begin{align}\label{Gaussian estimate 2}
	 	\Xi_{n}(0,t)\leq\left(4C_\delta\|u\|_{\infty,t}\sqrt{\frac{t}{\nu}}\right)^n\Xi_{0}(0,t)<\infty,\quad\forall t\geq0.
	 \end{align}
	 This also implies that $\Xi_{n}(\tau,t)\leq\Xi_{n}(0,t)<\infty$ for all $n\geq0$. As the definition of $\Xi_{n}(\tau,t)$ implies that
	 \begin{align}
	 	\big|\Omega_{n}(x,s)-\Omega_{n-1}(x,s)\big|\leq \Xi_{n}(\tau,t)~G_{(1+\delta)\nu(\sigma+s)}(x)~,\quad \forall s\in[\tau,t],\nonumber
	 \end{align}
	 we can safely derive by (\ref{Gaussian estimate 1}) that
	 \begin{align}
	 	\noindent&\big|\Omega_{n+1}(x,t)-\Omega_{n}(x,t)\big|\nonumber\\
	 	\leq &\big|\Omega_{n+1}(x,\tau)-\Omega_{n}(x,\tau)\big|+\frac{2C_\delta\|u\|_{\infty,t}}{\sqrt{\nu}}\Xi_{n}(\tau,t)\int_\tau^tG_{(1+\delta)\nu(t-s)}\ast G_{(1+\delta)\nu(\sigma+s)}(x)\frac{{\rm d}s}{\sqrt{t-s}}\nonumber\\
	 	&+\frac{2\widetilde C_\delta\|u\|_{\infty,\tau}}{\sqrt{\nu}}\Xi_{n}(0,\tau)\int_0^{\tau}{\rm d}s\int_{\tau}^t G_{(1+\delta)\nu (\theta-s)}\ast G_{(1+\delta)\nu(\sigma+s)}(x)\frac{{\rm d}\theta}{(\theta-s)^{3/2}}\nonumber\\
	 	\leq &\big|\Omega_{n+1}(x,\tau)-\Omega_{n}(x,\tau)\big|+4C_\delta~\|u\|_{\infty,t}\sqrt{\frac{t-\tau}{\nu}}~\Xi_{n}(\tau,t)~G_{(1+\delta)\nu(\sigma+t)}(x)\nonumber\\
	 	&+\frac{4\widetilde C_\delta\|u\|_{\infty,\tau}}{\sqrt{\nu}}\Xi_{n}(0,\tau)\int_{\tau}^tG_{(1+\delta)\nu(\sigma+\theta)}(x)\left(\frac{1}{\sqrt{\theta-\tau}}-\frac{1}{\sqrt{\theta}}\right){\rm d}\theta\nonumber\\
	 	\leq &\left(\frac{\sigma+t}{\sigma+\tau}\right)^{d/2}\Xi_{n+1}(0,\tau)~G_{(1+\delta)\nu(\sigma+t)}(x)+4C_\delta~\|u\|_{\infty,t}\sqrt{\frac{t-\tau}{\nu}}~\Xi_{n}(\tau,t)~G_{(1+\delta)\nu(\sigma+t)}(x)\nonumber\\
	 	&+\frac{8\widetilde C_\delta\|u\|_{\infty,\tau}}{\sqrt{\nu}}\big(\sqrt{t-\tau}+\sqrt{\tau}-\sqrt{t}\big)\left(\frac{\sigma+t}{\sigma+\tau}\right)^{d/2}\Xi_{n}(0,\tau)~G_{(1+\delta)\nu(\sigma+t)}(x)~.\nonumber
	 \end{align}
%Note that $t\mapsto \sqrt{t-\tau}+\sqrt{\tau}-\sqrt{t}$ is an increasing function when $\tau>0$. 
	 Note that the right-hand-side is an non-decreasing function of $t$. So if we denote $B_\delta:=\max\big\{4C_\delta, 8\widetilde C_\delta\big\}$, then by the definition of $\Xi_{n}(\tau,t)$ and (\ref{Gaussian estimate 2}), one has that
	 \begin{align}
	 	\Xi_{n+1}(\tau,t)
	 	\leq &\left(\frac{\sigma+t}{\sigma+\tau}\right)^{d/2}\Xi_{n+1}(0,\tau)+B_\delta~\|u\|_{\infty,t}\sqrt{\frac{t-\tau}{\nu}}~\Xi_{n}(\tau,t)\nonumber\\
	 	&+B_\delta~\|u\|_{\infty,t}\sqrt{\frac{t-\tau}{\nu}}\left(\frac{\sigma+t}{\sigma+\tau}\right)^{d/2}\Xi_{n}(0,\tau)~,\quad n\geq0.\label{Gaussian estimate 3}
	 \end{align}
	 For some $0<T'<T$ to be determine later, if we set $[\tau,t]=\big[T',2T'\big]$ in (\ref{Gaussian estimate 3}), then by the fact that $\frac{\sigma+t}{\sigma+\tau}\leq2$ 
	 %and $\sqrt{t-\tau}+\sqrt{\tau}-\sqrt{t}\leq\sqrt{t-\tau}\leq\sqrt{T'}$, 
	 we have
	 \begin{align}
	 	\Xi_{n+1}(T',2T')
	 	\leq &~2^{d/2}~\Xi_{n+1}(0,T')+B_\delta~\|u\|_{\infty,T}\sqrt{\frac{T'}{\nu}}~\Xi_{n}(T',2T')\nonumber\\
	 	&+B_\delta~\|u\|_{\infty,T}\sqrt{\frac{T'}{\nu}}~2^{d/2}~\Xi_{n}(0,T')\nonumber\\
	 	\leq &~B_\delta~\|u\|_{\infty,T}\sqrt{\frac{T'}{\nu}}~\Xi_{n}(T',2T')+2^{\frac{d}{2}+1}\left(B_\delta~\|u\|_{\infty,T}\sqrt{\frac{T'}{\nu}}\right)^{n+1}\Xi_{0}(0,T')~.\nonumber
	 \end{align}
	 Here in the last inequality we have used (\ref{Gaussian estimate 2}). We now set $0<T'<T$ such that
	 \begin{align}
	 	\Theta:=B_\delta~\|u\|_{\infty,T}\sqrt{\frac{T'}{\nu}}<1~,\nonumber
	 \end{align}
	 then iterating the above inequality would give
	 \begin{align}
	 	\Xi_{n}(T',2T')
	 	\leq~\Theta^n~\Xi_{0}(T',2T')+2^{\frac{d}{2}+1}n~\Theta^n~\Xi_{0}(0,T')\nonumber
	 \end{align}
	 which together with $\Xi_{n}(0,T')\leq~\Theta^n~\Xi_{0}(0,T')$ implies that
	 \begin{align}
	 	\Xi_{n}(0,2T')
	 	\leq~\left(2^{\frac{d}{2}+1}n+1\right)\Theta^n~\Xi_{0}(0,2T')~,\quad n\geq0.\nonumber
	 \end{align}
	 To this end, one could iterate the process by setting 
	 \begin{align}
	 	[\tau,t]=\big[mT',(m+1)T'\big]~,\quad m=0,1,...,\big\lfloor T/T'\big\rfloor-1\nonumber
	 \end{align}
	 successively in (\ref{Gaussian estimate 3}) and, in each step $m$, use the inequality $\frac{\sigma+t}{\sigma+\tau}\leq1+\frac{1}{m}$ instead so that
	 \begin{align}
	 	\Xi_{n+1}\big(mT',(m+1)T'\big)\leq \Theta~\Xi_{n}\big(mT',(m+1)T'\big)+\left(1+\frac{1}{m}\right)^{\frac{d}{2}+1}\Big[\Xi_{n+1}(0,mT')+\Theta~\Xi_{n}(0,mT')\Big]~.\nonumber
	 \end{align}
	 By induction on $m$, one may assume that
	 \begin{align}
	 	\Xi_{n}(0,mT')
	 	\leq~{\rm Q}_{m-1}(n)~\Theta^n~\Xi_{0}(0,mT')~,\quad n\geq0.\nonumber
	 \end{align}
	 holds for some polynomial ${\rm Q}_{m-1}$ of order $m-1$ with non-negative coefficients. Insert this in the previous inequality, we find
	 \begin{align}
	 	&\Xi_{n+1}\big(mT',(m+1)T'\big)\nonumber\\
	 	\leq &\Theta~\Xi_{n}\big(mT',(m+1)T'\big)+\left(1+\frac{1}{m}\right)^{\frac{d}{2}+1}\Big[{\rm Q}_{m-1}(n+1)+{\rm Q}_{m-1}(n)\Big]\Theta^{n+1}~\Xi_{0}(0,mT')~,~n\geq0.\nonumber
	 \end{align}
	 Iterating the inequality with respect to $n$ then shows that
	 \begin{align}
	 	\Xi_{n}\big(mT',(m+1)T'\big)\leq \Theta^n~\Xi_{0}\big(mT',(m+1)T'\big)+\left(1+\frac{1}{m}\right)^{\frac{d}{2}+1}{\rm P}_m(n)~\Theta^{n}~\Xi_{0}(0,mT')~,\quad n\geq0.\nonumber
	 \end{align}
	 with polynomial $\displaystyle {\rm P}_m(n):=\sum_{k=0}^{n-1}\Big[{\rm Q}_{m-1}(k+1)+{\rm Q}_{m-1}(k)\Big]$ for $n\geq1$ and ${\rm P}_m(0):=0$. As we obviously have that ${\rm P}_m(n)\geq {\rm Q}_{m-1}(n)$, this implies
	 \begin{align}
	 	\Xi_{n}\big(0,(m+1)T'\big)
	 	\leq~\left[\left(1+\frac{1}{m}\right)^{\frac{d}{2}+1}{\rm P}_m(n)+1\right]~\Theta^n~\Xi_{0}\big(0,(m+1)T'\big)~,\quad n\geq0\nonumber
	 \end{align}
	 and ${\rm Q}_{m}:=\left(1+\frac{1}{m}\right)^{\frac{d}{2}+1}{\rm P}_m(n)+1$ is clearly a polynomial of order $m$ with non-negative coefficients. Hence, we have shown that
	 \begin{align}
	    \Xi_{n}(0,MT')
	 	\leq~{\rm Q}_{M-1}(n)~\Theta^n~\Xi_{0}(0,MT')~,\quad n\geq0.\nonumber
	 \end{align}
	 for $M:=\big\lfloor T/T'\big\rfloor$. Now one can do the last round by setting $[\tau,t]=\big[MT',T\big]$ in (\ref{Gaussian estimate 3}). Note that since $T<(M+1)T'$, we can still use $\frac{\sigma+t}{\sigma+\tau}<\frac{\sigma+(M+1)T'}{\sigma+MT'}\leq1+\frac{1}{M}$ here. So go through the same process again and one would finally arrive at the estimate on the whole interval $[0,T]$, i.e.
	 \begin{align}
	 	\Xi_{n}(0,T)
	 	\leq~{\rm Q}_{M}(n)~\Theta^n~\Xi_{0}(0,T)~,\quad n\geq0.\nonumber
	 \end{align}
	 As $0<\Theta<1$, we have for all $t\in[0,T]$ that
	 \begin{align}
	 	\big|\Omega_{n}(x,t)\big|
	 	&\leq \sum_{k=0}^\infty\big|\Omega_{k}(x,t)-\Omega_{k-1}(x,t)\big|\nonumber\\
	 	&\leq G_{(1+\delta)\nu(\sigma+t)}(x)\sum_{k=0}^\infty\Xi_{k}(0,T)\nonumber\\
	 	&\leq D_\delta~\Xi_{0}(0,T)~G_{(1+\delta)\nu(\sigma+t)}(x)\label{estimate of Omega_n with implicit constant}
	 \end{align}
	 where $\displaystyle D_\delta:=\sum_{n=0}^\infty{\rm Q}_{M}(n)~\Theta^n<\infty$ is a constant generally depend on $\delta$, since the polynomial order $M$ depends on $T'=T'_\delta$ which indeed relies on $\delta$ so that one is allowed to set $\Theta\in(0,1)$. 
	 
	 \par Now we give an explicit estimate of the constant $D_\delta$. Note that by the fact that ${\rm P}_{M}(n)\leq2n~{\rm Q}_{M-1}(n)$ we can write
	 \begin{align}
	 	{\rm Q}_{M}(n)
	 	&=\left(1+\frac{1}{M}\right)^{\frac{d}{2}+1}{\rm P}_M(n)+1\nonumber\\
	 	&\leq \left(1+\frac{1}{M}\right)^{\frac{d}{2}+1}3n~{\rm Q}_{M-1}(n)\nonumber\\
	 	&\leq ...\leq \left(1+\frac{1}{M}\right)^{\frac{d}{2}+1}\left(1+\frac{1}{M-1}\right)^{\frac{d}{2}+1}...\left(1+\frac{1}{1}\right)^{\frac{d}{2}+1}(3n)^M\underbrace{{\rm Q}_{0}(n)}_{=1}\nonumber\\
	 	&=\left(1+M\right)^{\frac{d}{2}+1}(3n)^M~.\label{estimate of polynomial Q}
	 \end{align}
	 For the order $M$, one may choose $T'=T'_\delta$ such that $\Theta=B_\delta~\|u\|_{\infty,T}\sqrt{\frac{T'}{\nu}}=2^{-3}$, i.e.
	 $$T':=\frac{\nu}{2^6B_\delta^2~\|u\|_{\infty,T}^2}$$
	 and so
	 \begin{align}\label{estimate of polynomial order M}
	 	M=\left\lfloor\frac{T}{T'}\right\rfloor\leq C_\delta\frac{T~\|u\|_{\infty,T}^2}{\nu}
	 \end{align}
	 with $C_\delta=2^6B_\delta^2$. Then putting (\ref{estimate of polynomial Q},\ref{estimate of polynomial order M}) together we have that
	 \begin{align}
	 	D_\delta:=\sum_{n=0}^\infty{\rm Q}_{M}(n)~\Theta^n
	 	&\leq \left(1+M\right)^{\frac{d}{2}+1}\sum_{n=1}^\infty(3n)^M2^{-3n}\nonumber\\
	 	&\leq 2\left(1+M\right)^{\frac{d}{2}+1}\int_0^\infty\lambda^M2^{-\lambda}{\rm d}\lambda\nonumber\\
	 	&=2\left(1+M\right)^{\frac{d}{2}+1}\left(\frac{1}{\ln2}\right)^{M+1}\Gamma(M+1)\nonumber\\
	 	&\lesssim \left(1+M\right)^{\frac{d}{2}+1}\left(\frac{1}{\ln2}\right)^{M+1}\left(\frac{M}{e}\right)^{M+\frac{1}{2}}\lesssim_d M^M \label{estimate of constant D_delta}
	 \end{align}
	 where in the last inequality we have used the fact that $e\ln2>1$ and so the quantity $\frac{\left(1+M\right)^{\frac{d}{2}+1}M^{\frac{1}{2}}}{(e\ln2)^{M}}$ is bounded uniformly in $M>0$. Now the desired estimate of $\Omega_n$ follows by combining (\ref{estimate of Omega_n with implicit constant},\ref{estimate of polynomial order M},\ref{estimate of constant D_delta}) and the fact that $\Xi_{0}(0,T)\leq (1+\delta)^{d/2}\kappa$.
\end{proof}\

\section{The Velocity Equations}\label{sec. velocity equations}

\par In this section, we derive an asymptotic expansion of strong solution $u$ to the velocity equations (\ref{NS velocity}) arising from polynomial-concentrated initial value, that is, we treat a wider range of initial values:
\begin{align}\label{polynomial bound u_0}
	|u_0(x)|\lesssim \big(1+|x|\big)^{-p}~,\quad for~some~0\leq p\leq d+1~.
\end{align}
In particular, initial values satisfying the Gaussian bound (\ref{Gaussian bound u_0}) are included in this class. 

\par As now $u_0\in L^\infty$, the existence of unique local strong solution $u\in C_tL^\infty$ is clear, see \cite{FJR72,Kat84}. We define the maximal lifespan $0<T_{max}\leq\infty$ of the solution by 
$$T_{max}:=\sup\Big\{T>0~\Big|~ \exists!~solution~u\in C\big([0,T];L^\infty\big)~to~(\ref{NS velocity})~with~u(0)=u_0\Big\}~.$$
Since we also have $u_0\in L^q$ for every $d/p< q\leq\infty$, by $Weak$-$Strong$-$Uniqueness$ it is not hard to see that $u\in C_{loc}\big([0,T_{max});L^q\big)$ for every $\max\{d/p,d\}< q\leq\infty$. The derivation of the expansion is given by three sub-sections: \ref{subsec. mild solution formula} derives a mild solution formula in a slightly different form; in \ref{subsec. Decaying estimate} and \ref{subsec. decay estimate of grad u}, we see that the solution and its gradient generally satisfy a polynomial decaying bound of order $-d-1$; finally in \ref{subsec. Asymptotic Expansion} we utilize the estimate to deduce the desired asymptotic expansion. 

\subsection{Two slightly different mild solution formulas}\label{subsec. mild solution formula}

We shall denote the convolution operator
\begin{align}
	{\bf B}_t(F)(x):=-\int_0^t\Big[ G_{\nu (t-s)}\ast\big({\bf Id}+\nabla(-\Delta)^{-1}{\bf div}\big){\bf div} F(s)\Big](x){\rm d}s\nonumber
\end{align}
for any tensor fields $F:\mathbb R^d\times[0,T_{max})\rightarrow (\mathbb R^d)^{\otimes2}$ that would make sense. The formula is just formal at present and would be given in rigorous form in the following sub-sections. The mild solution formula of (\ref{NS velocity}) reads:
\begin{align}\label{mild solution formulation of u}
	u(x,t)=G_{\nu t}\ast u_0(x)+{\bf B}_t(u\otimes u)(x)~.
\end{align}

\subsubsection{The common formulation via Fourier multiplier}
We define convolution by the kernel $\mathcal K_{\ell,k}^j(x,t)$ via Fourier multiplier:
\begin{align}
	\mathcal K^j(t)\ast f:=\left(e^{-4\pi^2t|\xi|^2}\left(\delta_{j,\ell}-\frac{\xi_j\xi_\ell}{|\xi|^2}\right)(-2\pi i\xi_k)~\widehat{f^{\ell,k}}(\xi)\right)^\vee~,\quad f\in\mathcal S(\mathbb R^d).\nonumber
\end{align}
Then at least for $F\in L^\infty\big(0,T;\mathcal S(\mathbb R^d)\big)$, we can write
\begin{align}\label{FJR formula of map B}
	{\bf B}_t(F)=-\int_0^t\mathcal K_{\ell,k}^j\big(\nu(t-s)\big)\ast F^{\ell,k}(s){\rm d}s~.
\end{align}
By Fundamental Theorem of Calculus, one can write in Fourier transform that
\begin{align}
	\widehat{\mathcal K_{\ell,k}^j}(\xi,1)=\delta_{j,\ell}(-2\pi i\xi_k)e^{-4\pi^2|\xi|^2}+(-2\pi i\xi_k)(-4\pi^2\xi_j\xi_\ell)\int_1^{+\infty}e^{-4\pi^2 \tau|\xi|^2}{\rm d}\tau~.\nonumber
\end{align}
As by scaling $\left((-2\pi i\xi_k)(-4\pi^2\xi_j\xi_\ell)e^{-4\pi^2 \tau|\xi|^2}\right)^\vee=\tau^{-\frac{d+3}{2}}\partial^3_{k,j,\ell}G_1\left(\frac{x}{\sqrt{\tau}}\right)$, this implies the following formula which is firstly obtained by C.W.Oseen \cite{Ose27} in 3-dimension (see also \cite{FJR72}):
\begin{align}
	\mathcal K_{\ell,k}^j(x,t)
	&=\delta_{j,\ell}\partial_k G_t(x)+\int_t^{+\infty}\partial^3_{k,j,\ell}G_1\left(\frac{x}{\sqrt{\tau}}\right)\tau^{-\frac{d+3}{2}}{\rm d}\tau\nonumber\\
	&=\delta_{j,\ell}\partial_k G_t(x)+\int_0^{1/t}\partial^3_{k,j,\ell}G_1\left(\sqrt{s}x\right)s^{\frac{d-1}{2}}{\rm d}s~.\nonumber
\end{align}
It is famously known that the point-wise estimate
\begin{align}\label{point-wise bound of kernel mathcal K}
	\left|\mathcal K(x,t)\right|\lesssim_d\left(\sqrt{t}+|x|\right)^{-d-1}
\end{align}
holds as a consequence of the formula. In the following subsections, we would also need estimates for time difference. We come back to the above formula. For $0<\tau<t\leq+\infty$, we could use the identity $\partial_s\mathcal K_{\ell,k}^j(x,s)=\delta_{j,\ell}\partial_k\Delta G_s(x)-\partial^3_{k,j,\ell}G_1\left(\frac{x}{\sqrt{s}}\right)s^{-\frac{d+3}{2}}$ and Fundamental Theorem of Calculus to write
\begin{align}
	\mathcal K_{\ell,k}^j(x,t)-\mathcal K_{\ell,k}^j(x,\tau)=\int_\tau^t\Big[(\delta_{j,\ell}\partial_k\Delta-\partial^3_{k,j,\ell})G_1\Big]\left(\frac{x}{\sqrt{s}}\right)s^{-\frac{d+3}{2}}{\rm d}s~.\nonumber
\end{align}
By taking $L^1$-norm and using the elementary inequality $\big|(\delta_{j,\ell}\partial_k\Delta-\partial^3_{k,j,\ell})G_1(x)\big|\lesssim_d e^{-|x|^2/8}$ separately, one gets that
\begin{align}
	\big\|\mathcal K(t)-\mathcal K(\tau)\big\|_{1}
	&\lesssim_d\int_\tau^ts^{-3/2}{\rm d}s\lesssim_d \tau^{-1/2}-t^{-1/2}~;\label{L^1 estimate of mathcal K(t)-mathcal K(tau)}\\
	\left|\mathcal K_{\ell,k}^j(x,t)-\mathcal K_{\ell,k}^j(x,\tau)\right|
	&\lesssim_d\int_\tau^tG_{2s}(x)s^{-3/2}{\rm d}s~.\label{point-wise estimate of mathcal K(t)-mathcal K(tau)}
\end{align} 
In particular, sending $t\rightarrow+\infty$ would give
\begin{align}
	\big\|\mathcal K(\tau)\big\|_{1}
	&\lesssim_d \tau^{-1/2}~;\label{L^1-bound of kernel mathcal K}
%	\left|\mathcal K_{\ell,k}^j(x,\tau)\right|
%	&\lesssim_d\int_\tau^{+\infty}G_{2s}(x)s^{-3/2}{\rm d}s~.\label{point-wise bound of kernel mathcal K}
\end{align}
Now by Young's inequality, one can easily extend the map $F\mapsto{\bf B}(F)$ given by (\ref{FJR formula of map B}) onto $L^p(0,T;L^q)$ (~$\forall 1\leq p,q\leq\infty$~)~: 
\begin{align}\label{continuity of B_t(F)}
	\big\|{\bf B}(F)\big\|_{L^p(0,T;L^q)}\lesssim_d \sqrt{\frac{T}{\nu}}~\|F\|_{L^p(0,T;L^q)}~.
\end{align}

\

\subsubsection{Formulation via Poisson potential}
We will need a slightly different formulation of ${\bf B}_t(F)$ in subsection \ref{subsec. Asymptotic Expansion}. By integration by parts, 
\begin{align}
	{\bf B}_t(F)^j=-\int_0^t\partial_\ell G_{\nu (t-s)}\ast F^{\ell,j}(s){\rm d}s-\int_0^t\partial_j G_{\nu (t-s)}\ast\left[\Gamma\ast\partial_k\partial_\ell F^{\ell,k}(s)\right]{\rm d}s~,\nonumber
\end{align}
where we have used Einstein summation convention and $\Gamma$ denotes the fundamental solution of Laplacian ``$-\Delta$". Now the following integration-by-parts lemma implies the desired formulation:
\begin{align}\label{decomposition of B(F) I}
	{\bf B}_t(F)^j=-\underbrace{\int_0^t\partial_\ell G_{\nu (t-s)}\ast\left[ F^{\ell,j}(s)-\frac{\delta_{\ell,j}}{d}{\bf tr}(F(s))\right]{\rm d}s}_{=:{\bf C}_t(F)^j}
	-\underbrace{\int_0^t\partial_j G_{\nu (t-s)}\ast\left[\partial_{\ell,k}^2\Gamma\ast F^{\ell,k}(s)\right]{\rm d}s}_{=:{\bf D}_t(F)^j}
\end{align}
for a class of tensor fields $F$. For convenience we would denote by $C^0_{loc,0}\big(\mathbb R^d\times[0,T);(\mathbb R^d)^{\otimes2}\big)$ the entity of tensor fields $F\in C^0_{loc}\big(\mathbb R^d\times[0,T);(\mathbb R^d)^{\otimes2}\big)$ that decay to zero at infinity.

\begin{lemma}\label{lem. integration by parts}
	Let $F\in C^0_{loc,0}\big(\mathbb R^d;(\mathbb R^d)^{\otimes2}\big)$ satisfies $\nabla F\in L^\infty_{loc}$ and
	\begin{align}
		\int_{\mathbb S^{d-1}}|\nabla F(x-R\omega)|{\rm d}\omega=\left\{
		\begin{aligned}
			&o\left(1/R\right)\quad\quad\quad d\geq3\\
			&o\left(\frac{1}{R\ln R}\right)\quad\ d=2
		\end{aligned}\right.
		\quad\quad as~R\rightarrow+\infty\nonumber
	\end{align}
	for $\forall x\in\mathbb R^d$, then the equality holds point-wisely:
	\begin{align}
		\Gamma\ast\partial_k\partial_\ell F^{\ell,k}=\partial_k\partial_\ell\Gamma\ast F^{\ell,k}-\frac{1}{d}{\bf tr}(F)~.\nonumber
	\end{align}
	Here ${\bf tr}(F)$ denotes the trace of $F$.
\end{lemma}

\noindent Note that the limiting boundary term ``$-\frac{1}{d}{\bf tr}(F)$" is a consequence of the integrable singularity of $\nabla\Gamma$ at $x=0$. The proof of Lemma \ref{lem. integration by parts} is quite standard, we leave the proof in Appendix for the sake of flow.

\
\subsection{Decaying estimate of velocity at infinity}\label{subsec. Decaying estimate}

\par In this sub-section, we show that the strong solution $u$ satisfies the following polynomial decaying bound (\ref{polynomial bound for velocity}). Remind that we are treating here initial values satisfying the polynomial bound (\ref{polynomial bound u_0}) with $0\leq p\leq d+1$.

\begin{proposition}\label{prop. polynomial bound for velocity}
	Let $d\geq2$ and $u_0\in L^\infty$ satisfying (\ref{polynomial bound u_0}) with $0\leq p\leq d+1$. Let also $u$ be the strong solution to (\ref{NS velocity}) generated by $u_0$ , with maximal lifespan $0<T_{max}\leq\infty$. Then the following bound holds for all $(x,t)\in\mathbb R^d\times[0,T_{max})$ :
	\begin{align}\label{polynomial bound for velocity}
		\big(1+|x|\big)^{p}\big|u(x,t)\big|\lesssim_{d,p} \big(1+\sqrt{\nu t}\big)^{p}\Big[t~\mathfrak D_\nu(t)^2\Big]^{t~\mathfrak D_\nu(t)^2}
	\end{align}
	where the increasing function $\mathfrak D_\nu(t)$ is given by
	\begin{align}\label{function D(t)}
		\mathfrak D_\nu(t):=\frac{C_{d,p}\|u\|_{\infty,t}\big(1+\sqrt{\nu t}\big)^{p+1}}{\sqrt{\nu~}}
	\end{align}
	and $\|u\|_{\infty,t}$ stands for $\|u\|_{L^\infty((0,t)\times\mathbb R^d)}$ ; $C_{d,p}>0$ depends only on $d,p$.
\end{proposition}

The strategy of proving Proposition \ref{prop. polynomial bound for velocity} is 
similar to that we use to prove Gaussian bound of vorticity, but there are significant difference when we do the decaying estimates. To begin with, we define the sequence ${\rm  U}_n:[0,T_{max})\times\mathbb R^d\rightarrow\mathbb R^d$ by :
\begin{align}
	{\rm  U}_0(t)
	&:=G_{\nu t}\ast u_0\nonumber\\
	{\rm  U}_{n+1}(t)
	&:={\rm  U}_0(t)+{\bf B}_t(u\otimes {\rm  U}_n)~,\quad n\geq0.\label{definition of U_n}
\end{align}
Then the following two propositions imply Proposition \ref{prop. polynomial bound for velocity}.\\

\begin{proposition}\label{prop. convergence of U_n}
	Let $d\geq2$, $u$ be given as in Proposition \ref{prop. polynomial bound for velocity} and ${\rm  U}_n$ be given by (\ref{definition of U_n}). Then for any $~\max\{d/p,d\}< q\leq\infty~$ and $0<T<T_{max}$~, we have the strong convergence ${\rm  U}_n\longrightarrow u$ in $L^\infty(0,T;L^q)$.
\end{proposition}

\begin{proof}[Proof of Proposition \ref{prop. polynomial bound for velocity}]
	The proof is very similar to Proposition \ref{prop. convergence of Omega_n}, although the convolution is given by a different kernel: Let $d\leq q<\infty$ be given. As it is not hard to see that ${\rm  U}_n\in L^\infty(0,T;L^q)$ for each $n\geq0$, we can write for any $0\leq\tau<t\leq T$ that
	 \begin{align}
	 	{\rm  U}_{n+1}(t)={\rm  U}_{n+1}(\tau)
	 	&+\left(G_{\nu t}- G_{\nu\tau}\right)\ast u_0+\int_{\tau}^t\mathcal K_{\ell,k}\big(\nu(t-s)\big)\ast\big[-u^\ell(s){\rm  U}_n^k(s)\big]{\rm d}s\nonumber\\
	 	&+\int_0^{\tau}\left[\mathcal K_{\ell,k}\big(\nu(t-s)\big)-\mathcal K_{\ell,k}\big(\nu(\tau-s)\big)\right]\ast\big[-u^\ell(s){\rm  U}_n^k(s)\big]{\rm d}s~;\label{time-difference of U_n+1}\\
	 	u(t)=u(\tau)
	 	&+\left(G_{\nu t}- G_{\nu\tau}\right)\ast u_0 +\int_{\tau}^t\mathcal K_{\ell,k}\big(\nu(t-s)\big)\ast\big[-u^\ell(s)u^k(s)\big]{\rm d}s\nonumber\\
	 	&+\int_0^{\tau}\left[\mathcal K_{\ell,k}\big(\nu(t-s)\big)-\mathcal K_{\ell,k}\big(\nu(\tau-s)\big)\right]\ast\big[-u^\ell(s)u^k(s)\big]{\rm d}s~.\label{time-difference of u}
	 \end{align}
	 Then take the difference and estimate $L^q$-norm using Young's inequality and the $L^1$-bounds (\ref{L^1 estimate of mathcal K(t)-mathcal K(tau)},\ref{L^1-bound of kernel mathcal K}) of the kernel $\mathcal K$, we get
	  \begin{align}
	 	\big\|{\rm  U}_{n+1}(t)-u(t)\big\|_q
	 	\leq &\big\|{\rm  U}_{n+1}(\tau)-u(\tau)\big\|_q+ C_d \|u\|_{\infty,t}\int_{\tau}^t\frac{\big\|{\rm  U}_{n}(s)-u(s)\big\|_q}{\sqrt{\nu(t-s)}}{\rm d}s\nonumber\\
	 	&+C_d' \|u\|_{\infty,\tau}\int_0^{\tau}\left(\frac{1}{\sqrt{\nu(\tau-s)}}-\frac{1}{\sqrt{\nu(t-s)}}\right)\big\|{\rm  U}_{n}(s)-u(s)\big\|_q{\rm d}s\nonumber\\
	 	\leq &\big\|{\rm  U}_{n+1}(\tau)-u(\tau)\big\|_q+ 2C_d \|u\|_{\infty,t}\sqrt{\frac{t-\tau}{\nu}}~\sup_{s\in[\tau,t]}\big\|{\rm  U}_{n}(s)-u(s)\big\|_{q}\nonumber\\
	 	&+2C_d' \|u\|_{\infty,\tau}\left(\sqrt{\frac{t-\tau}{\nu}}-\sqrt{\frac{t}{\nu}}+\sqrt{\frac{\tau}{\nu}}\right) \sup_{s\in[0,\tau]}\big\|{\rm  U}_{n}(s)-u(s)\big\|_{q}\nonumber
	 \end{align}
	 where $\|u\|_{\infty,t}$ stands for $\|u\|_{L^\infty((0,t)\times\mathbb R^d)}$ . Denote $B_d:=max\big\{2C_d,2C_d'\big\}$, then we deduce for any interval $[\tau,t]\subset[0,T]$ that
	 \begin{align}
	 	\sup_{s\in[\tau,t]}\big\|{\rm  U}_{n+1}(s)-u(s)\big\|_q
	 	\leq &\big\|{\rm  U}_{n+1}(\tau)-\omega(\tau)\big\|_q+ B_d \|u\|_{\infty,T}\sqrt{\frac{t-\tau}{\nu}}~\sup_{s\in[\tau,t]}\big\|{\rm  U}_{n}(s)-u(s)\big\|_{q}\nonumber\\
	 	&+B_d \|u\|_{\infty,T}\sqrt{\frac{t-\tau}{\nu}}~\sup_{s\in[0,\tau]}\big\|{\rm  U}_{n}(t)-u(t)\big\|_{q}~.\nonumber
	 \end{align}
	 This is the counterpart of (\ref{convergence estimate 1}). Now one only need to repeat the inductive process as given in the proof of Proposition \ref{prop. convergence of Omega_n}.
\end{proof}

\
\begin{proposition}\label{prop. polynomial bound for U_n}
	Let $d\geq2$, $u$ and ${\rm  U}_n$ be given as in Proposition \ref{prop. convergence of U_n}. Then for any $0<T<T_{max}$, the following estimate holds for all $n\in\mathbb N_0$ and $(x,t)\in\mathbb R^d\times[0,T]$ :
	\begin{align}\label{polynomial bound for U_n}
		\big(1+|x|\big)^{p}\big|{\rm  U}_{n}(x,t)\big|\lesssim_{d,p} \big(1+\sqrt{\nu T}\big)^{p}\Big[T~\mathfrak D_\nu(T)^2\Big]^{T~\mathfrak D_\nu(T)^2}~.
	\end{align}
	Here the increasing function $\mathfrak D_\nu(t)$ is given by (\ref{function D(t)}).
\end{proposition}

\begin{proof}[Proof of Proposition \ref{prop. polynomial bound for U_n}]
	The proof is again very similar to Proposition \ref{prop. Gaussian bound of Omega_n}. We only show the different part here. Firstly we deduce from (\ref{time-difference of U_n+1}) that
	\begin{align}
		{\rm  U}_{n+1}(t)-{\rm  U}_n(t)=
		&{\rm  U}_{n+1}(\tau)-{\rm  U}_n(\tau)+\int_{\tau}^t\mathcal K_{\ell,k}\big(\nu(t-s)\big)\ast\Big\{-u^\ell(s)\big[{\rm  U}_n^k(s)-{\rm  U}_{n-1}^k(s)\big]\Big\}{\rm d}s\nonumber\\
		&+\int_0^{\tau}\left[\mathcal K_{\ell,k}^j\big(\nu(t-s)\big)-\mathcal K_{\ell,k}^j\big(\nu(\tau-s)\big)\right]\ast\Big\{-u^\ell(s)\big[{\rm  U}_n^k(s)-{\rm  U}_{n-1}^k(s)\big]\Big\}{\rm d}~,\nonumber
	\end{align}
	and then apply the point-wise estimates (\ref{point-wise bound of kernel mathcal K},\ref{point-wise estimate of mathcal K(t)-mathcal K(tau)}) and Fubini's Theorem to write:
	\begin{align}
		\big|{\rm  U}_{n+1}(x,t)-{\rm  U}_n(x,t)\big|\leq 
		&\big|{\rm  U}_{n+1}(x,\tau)-{\rm  U}_n(x,\tau)\big|+\|u\|_{\infty,t}\int_{\tau}^t\int_{\mathbb R^d}\frac{\big|{\rm  U}_n(y,s)-{\rm  U}_{n-1}(y,s)\big|}{\big(\sqrt{\nu(t-s)}+|x-y|\big)^{d+1}}{\rm d}s\nonumber\\
		&+\|u\|_{\infty,\tau}\int_0^\tau\int_{\nu(\tau-s)}^{\nu(t-s)}\theta^{-3/2} G_{2\theta}\ast\big|{\rm  U}_n(s)-{\rm  U}_{n-1}(s)\big|(x){\rm d}\theta{\rm d}s~,\quad 0\leq\tau<t.\label{estimate of U_n-u 1}
	\end{align}
	We have dropped the notation of spacial variables $x$ here for simplicity. (\ref{estimate of U_n-u 1}) is the counterpart of (\ref{Gaussian estimate 1}). For convenience, we may set ${\rm  U}_{-1}\equiv0$ so that the inequality also holds for $n=0$. Now we define the poly-weighted supremum norm for any vector fields $f$ that would make sense:
	 \begin{align}
	 	\|f\|_{P(t)}:=\sup_{x\in\mathbb R^d}\Big(\big(1+|x|\big)^{p}~|f(x,t)|\Big)\nonumber
	 \end{align}
	 and we denote
	 \begin{align}
	 	\Phi_{n}(\tau,t):=\sup_{s\in[\tau,t]}\big\|{\rm  U}_{n}-{\rm  U}_{n-1}\big\|_{P(s)}~,\quad n\geq1\nonumber
	 \end{align}
	 provided that it is finite. Now the counterpart of the semi-group property (\ref{semi-group property of G}) is the following convolution inequalities :
	 \begin{align}
	 	\int_{\mathbb R^d}G_{at}(y)\frac{{\rm d}y}{\big(1+|x-y|\big)^p}
	 	&\lesssim_{d,p,a}\big(1+\sqrt{t}\big)^p\big(1+|x|\big)^{-p}~,\quad 0\leq p<\infty~;\label{counterpart of the semi-group property 1}\\
	 	\int_{\mathbb R^d}\big(\sqrt{a}+|y|\big)^{-d-1}\frac{{\rm d}y}{\big(1+|x-y|\big)^p}
	 	&\lesssim_{d,p}\left(1+\frac{1}{\sqrt{a}}\right)\big(1+|x|\big)^{-p}~,\quad 0\leq p\leq d+1~.\label{counterpart of the semi-group property 2}
	 \end{align}
	 For the sake of flow, we delay the proof of (\ref{counterpart of the semi-group property 1},\ref{counterpart of the semi-group property 2}) to the end of this subsection. Then by definition of ${\rm  U}_0$ and the assumption (\ref{polynomial bound u_0}),
	 \begin{align}
	 	\big|{\rm  U}_0(x,t)\big|\lesssim_{d,p}\big(1+\sqrt{\nu t}\big)^p\big(1+|x|\big)^{-p}\nonumber
	 \end{align}
	 which implies
	 \begin{align}\label{estimate of Phi_0(0,t)}
	 	\Phi_{0}(0,t)\lesssim_{d,p}\big(1+\sqrt{\nu t}\big)^p<\infty~,\quad\forall t>0.
	 \end{align}
	 Now if we set $n=0$ and $\tau=0$ in (\ref{estimate of U_n-u 1}), we would get that ( ${\rm  U}_n(0)=u_0$ )
	 \begin{align}
	 	\big|{\rm  U}_1(x,t)-{\rm  U}_0(x,t)\big|
	 	&\leq\|u\|_{\infty,t}\int_0^t\int_{\mathbb R^d}\frac{\big|{\rm  U}_0(x-y,s)\big|}{\big(\sqrt{\nu(t-s)}+|y|\big)^{d+1}}{\rm d}s\nonumber\\
	 	&\leq\|u\|_{\infty,t}\Phi_{0}(0,t)\int_0^t{\rm d}s\int_{\mathbb R^d}\big(\sqrt{\nu(t-s)}+|y|\big)^{-d-1}\frac{{\rm d}y}{\big(1+|x-y|\big)^p}\nonumber\\
	 	&\lesssim_{d,p}\|u\|_{\infty,t}\Phi_{0}(0,t)\big(1+|x|\big)^{-p}\int_0^t\bigg(1+\frac{1}{\sqrt{\nu(t-s)}}\bigg){\rm d}s\nonumber\\
	 	&\lesssim_{d,p}\|u\|_{\infty,t}\Phi_{0}(0,t)\big(1+|x|\big)^{-p}\big(1+\sqrt{\nu t}\big)\sqrt{\frac{t}{\nu}}\nonumber
	 \end{align}
	 i.e.
	 \begin{align}
	 	\Phi_{1}(0,t)\leq C_{d,p}\|u\|_{\infty,t}\big(1+\sqrt{\nu t}\big)\sqrt{\frac{t}{\nu}}~\Phi_{0}(0,t)<\infty~,\quad \forall t>0.\nonumber
	 \end{align}
	 By induction, it is not hard to see that for every $n\geq1$,
	 \begin{align}\label{estimate of Phi_n(0,t)}
	 	\Phi_{n}(0,t)\leq\left(C_{d,p}\|u\|_{\infty,t}\big(1+\sqrt{\nu t}\big)\sqrt{\frac{t}{\nu}}\right)^n\Phi_{0}(0,t)<\infty~,\quad \forall t>0.
	 \end{align}
	 This also implies that $\Phi_{n}(\tau,t)\leq\Phi_{n}(0,t)<\infty$ for all intervals $[\tau,t]\subset[0,T]$ and $n\geq0$. Now, for any such interval $[\tau,t]$, we can safely deduce from (\ref{estimate of U_n-u 1}) and the inequalities (\ref{counterpart of the semi-group property 1},\ref{counterpart of the semi-group property 2}) that
	 \begin{align}
	 	&\big|{\rm  U}_{n+1}(x,t)-{\rm  U}_n(x,t)\big|\nonumber\\
		\leq &\big|{\rm  U}_{n+1}(x,\tau)-{\rm  U}_n(x,\tau)\big|+\|u\|_{\infty,t}\Phi_{n}(\tau,t)\int_\tau^t{\rm d}s\int_{\mathbb R^d}\big(\sqrt{\nu(t-s)}+|x-y|\big)^{-d-1}\frac{{\rm d}y}{\big(1+|y|\big)^p}\nonumber\\
		&+\|u\|_{\infty,\tau}\Phi_{n}(0,\tau)\int_0^\tau{\rm d}s\int_{\nu(\tau-s)}^{\nu(t-s)}\theta^{-3/2}{\rm d}\theta\int_{\mathbb R^d}G_{2\theta}(y)\frac{{\rm d}y}{\big(1+|x-y|\big)^p}\nonumber\\
		\leq & \Phi_{n+1}(0,\tau)\big(1+|x|\big)^{-p}+\|u\|_{\infty,t}\Phi_{n}(\tau,t)\big(1+|x|\big)^{-p}\int_\tau^t\bigg(1+\frac{1}{\sqrt{\nu(t-s)}}\bigg){\rm d}s\nonumber\\
		&+\|u\|_{\infty,\tau}\Phi_{n}(0,\tau)\big(1+|x|\big)^{-p}\int_0^\tau{\rm d}s\int_{\nu(\tau-s)}^{\nu(t-s)}\big(1+\sqrt{\theta}\big)^p\theta^{-3/2}{\rm d}\theta\nonumber\\
		\leq & \Phi_{n+1}(0,\tau)\big(1+|x|\big)^{-p}+2\|u\|_{\infty,t}\Phi_{n}(\tau,t)\big(1+|x|\big)^{-p}\Big[1+\sqrt{\nu (t-\tau)}\Big]\sqrt{\frac{t-\tau}{\nu}}\nonumber\\
		&+2\|u\|_{\infty,\tau}\Phi_{n}(0,\tau)\big(1+|x|\big)^{-p}\big(1+\sqrt{\nu t}\big)^p\int_0^\tau\bigg[\frac{1}{\sqrt{\nu (\tau-s)}}-\frac{1}{\sqrt{\nu (t-s)}}\bigg]{\rm d}s\nonumber\\
		\leq & \Phi_{n+1}(0,\tau)\big(1+|x|\big)^{-p}+2\|u\|_{\infty,t}\Big[1+\sqrt{\nu (t-\tau)}\Big]\sqrt{\frac{t-\tau}{\nu}}\Phi_{n}(\tau,t)\big(1+|x|\big)^{-p}\nonumber\\
		&+4\|u\|_{\infty,\tau}\big(1+\sqrt{\nu t}\big)^p\left(\sqrt{\frac{t-\tau}{\nu}}+\sqrt{\frac{\tau}{\nu}}-\sqrt{\frac{t}{\nu}}\right)\Phi_{n}(0,\tau)\big(1+|x|\big)^{-p}~.\nonumber
	 \end{align}
	 As the right-hand-side is an non-decreasing function of $t$, this indicates that
	 \begin{align}
	 	\Phi_{n+1}(\tau,t)\leq 
	 	& ~\Phi_{n+1}(0,\tau)+2\|u\|_{\infty,t}\Big[1+\sqrt{\nu (t-\tau)}\Big]\sqrt{\frac{t-\tau}{\nu}}\Phi_{n}(\tau,t)\nonumber\\
		&+4\|u\|_{\infty,\tau}\big(1+\sqrt{\nu t}\big)^p\sqrt{\frac{t-\tau}{\nu}}\Phi_{n}(0,\tau)~,\quad n\geq0.\label{polynomial decay estimate 1}
	 \end{align}
	 Let $0<T'<T$ be determine later. If we set $[\tau,t]=\big[T',2T'\big]$ in (\ref{polynomial decay estimate 1}), then
	 \begin{align}
	 	\Phi_{n+1}(T',2T') 
	 	\leq & ~\Phi_{n+1}(0,T')+2\|u\|_{\infty,T}\big(1+\sqrt{\nu T}\big)\sqrt{\frac{T'}{\nu}}\Phi_{n}(T',2T')\nonumber\\
	 	&+4\|u\|_{\infty,T}\big(1+\sqrt{\nu T}\big)^p\sqrt{\frac{T'}{\nu}}\Phi_{n}(0,T')\nonumber\\
	 	\leq &~\Phi_{n+1}(0,T')+B_{d,p}\|u\|_{\infty,T}\big(1+\sqrt{\nu T}\big)^{p+1}\sqrt{\frac{T'}{\nu}}\Phi_{n}(T',2T')\nonumber\\
	 	&+B_{d,p}\|u\|_{\infty,T}\big(1+\sqrt{\nu T}\big)^{p+1}\sqrt{\frac{T'}{\nu}}\Phi_{n}(0,T')~,\quad n\geq0.\label{estimate of Phi_n(T,2T) 1}
	 \end{align}
	 Here we denote $B_{d,p}=\max\{C_{d,p},4\}$. Now we choose $0<T'<T$ such that
	 \begin{align}
	 	\widetilde\Theta:=B_{d,p}\|u\|_{\infty,T}\big(1+\sqrt{\nu T}\big)^{p+1}\sqrt{\frac{T'}{\nu}}<1~.\nonumber
	 \end{align}
	 Then (\ref{estimate of Phi_n(0,t)}) would give $\Phi_{n}(0,T')\leq \widetilde\Theta^n~\Phi_{0}(0,T')$ for $n\geq0$. Inserting this into (\ref{estimate of Phi_n(T,2T) 1}), we arrive at
	 \begin{align}
	 	\Phi_{n+1}(T',2T')\leq \widetilde\Theta~\Phi_{n}(T',2T')+2~\widetilde\Theta^{n+1}~\Phi_{0}(0,T')~,\quad n\geq0\nonumber
	 \end{align}
	 which implies that
	 \begin{align}
	 	\Phi_{n}(T',2T')\leq \widetilde\Theta^{n}~\Phi_{0}(T',2T')+2n~\widetilde\Theta^{n}~\Phi_{0}(0,T')~.\nonumber
	 \end{align}
	 Since $\Phi_{n}(0,T')\leq \widetilde\Theta^n~\Phi_{0}(0,T')$, we deduce that
	 \begin{align}
	 	\Phi_{n}(0,2T')\leq \left(2n+1\right)\widetilde\Theta^{n}~\Phi_{0}(0,2T')~,\quad n\geq0.\nonumber
	 \end{align}
	 To this end, one could iterate the process by setting
	 \begin{align}
	 	[\tau,t]=\big[mT',(m+1)T'\big]~,\quad m=0,1,...,M-1\nonumber
	 \end{align}
	 successively in (\ref{polynomial decay estimate 1}) ( $M=\big\lfloor T/T'\big\rfloor$ ), so that
	 \begin{align}
	 	\Phi_{n+1}\big(mT',(m+1)T'\big) 
	 	\leq & ~\Phi_{n+1}(0,mT')+B_{d,p}\|u\|_{\infty,T}\big(1+\sqrt{\nu T}\big)^{p+1}\sqrt{\frac{T'}{\nu}}\Phi_{n}\big(mT',(m+1)T'\big)\nonumber\\
	 	&+B_{d,p}\|u\|_{\infty,T}\big(1+\sqrt{\nu T}\big)^{p+1}\sqrt{\frac{T'}{\nu}}\Phi_{n}(0,mT')\nonumber\\
	 	= & ~\widetilde\Theta~\Phi_{n}\big(mT',(m+1)T'\big)+\Big[\Phi_{n+1}(0,mT')+\widetilde\Theta~\Phi_{n}(0,mT')\Big].\nonumber
	 \end{align}
	 Now, by induction on m, one may assume that
	 \begin{align}
	 	\Phi_{n}(0,mT')\leq \widetilde{\rm Q}_{m-1}(n)~\widetilde\Theta^{n}~\Phi_{0}(0,mT')~,\quad n\geq0\nonumber
	 \end{align}
	 holds for some polynomial $\widetilde{\rm Q}_{m-1}$ of order $m-1$ with non-negative coefficients. Insert this into the previous inequality, we deduce that
	 \begin{align}
	 	\Phi_{n+1}\big(mT',(m+1)T'\big)\leq ~\widetilde\Theta~\Phi_{n}\big(mT',(m+1)T'\big)+\left[\widetilde{\rm Q}_{m-1}(n+1)+\widetilde{\rm Q}_{m-1}(n)\right]\widetilde\Theta^{n+1}~\Phi_{0}(0,mT')~.\nonumber
	 \end{align}
	 Iterating the inequality with respect to $n$ gives:
	 \begin{align}
	 	\Phi_{n}\big(mT',(m+1)T'\big)\leq~\widetilde\Theta^{n}~\Phi_{0}\big(mT',(m+1)T'\big)+\widetilde{\rm P}_{m}(n)~\widetilde\Theta^{n}~\Phi_{0}(0,mT')~,\quad n\geq0\nonumber
	 \end{align}
	 with polynomial $\displaystyle\widetilde{\rm P}_{m}(n):=\sum_{k=0}^{n-1}\left[\widetilde{\rm Q}_{m-1}(k+1)+\widetilde{\rm Q}_{m-1}(k)\right]$ for $n\geq 1$ and $\widetilde{\rm P}_{m}(0):=0$. Then together with our induction assumption, we finally have
	 \begin{align}
	 	\Phi_{n}\big(0,(m+1)T'\big)\leq~\left[\widetilde{\rm P}_{m}(n)+1\right]~\widetilde\Theta^{n}~\Phi_{0}\big(0,(m+1)T'\big)~,\quad n\geq0\nonumber
	 \end{align}
	 where $\widetilde{\rm Q}_{m}(n):=\widetilde{\rm P}_{m}(n)+1$ is indeed a polynomial of order $m$ with non-negative coefficients. So the finite induction goes to level $m=M-1$:
	 \begin{align}
	 	\Phi_{n}\big(0,MT'\big)\leq\widetilde{\rm Q}_{M-1}(n)~\widetilde\Theta^{n}~\Phi_{0}\big(0,MT'\big)~,\quad n\geq0.\nonumber
	 \end{align}

	 \par Now one can do the last round by setting $[\tau,t]=\big[MT',T\big]$ in (\ref{polynomial decay estimate 1}). Since $T<(M+1)T'$, we can still use $t-\tau\leq T'$ here. So go through the same process again and one would finally arrive at the estimate on the whole interval $[0,T]$, i.e.
	 \begin{align}
	 	\Phi_{n}(0,T)
	 	\leq\widetilde{\rm Q}_{M}(n)~\widetilde\Theta^n~\Phi_{0}(0,T)~,\quad n\geq0.\nonumber
	 \end{align}
	 As $0<\Theta<1$, we have for all $t\in[0,T]$ that
	 \begin{align}
	 	\big|{\rm U}_{n}(x,t)\big|
	 	&\leq \sum_{k=0}^\infty\big|{\rm U}_{k}(x,t)-{\rm U}_{k-1}(x,t)\big|\nonumber\\
	 	&\leq \big(1+|x|\big)^{-p}\sum_{k=0}^\infty\Phi_{k}(0,T)\nonumber\\
	 	&\leq D_{T'}~\Phi_{0}(0,T)~\big(1+|x|\big)^{-p}\label{estimate of U_n with implicit constant}
	 \end{align}
	 where $\displaystyle D_{T'}:=\sum_{n=0}^\infty\widetilde{\rm Q}_{M}(n)~\Theta^n<\infty$ is a constant depend on $T'$. Now we give an explicit upper-bound of the constant $D_{T'}$. Firstly, we obviously have $\widetilde{\rm Q}_{M}(n)\leq 3n\widetilde{\rm Q}_{M-1}(n)\leq...\leq(3n)^M$ (note that $\widetilde{\rm Q}_{0}(n)\equiv1$). Next, we could set $T'\in(0,T]$ such that $\widetilde\Theta=2^{-3}$, i.e.
	 \begin{align}
	 	T':=\frac{\nu}{2^6B_{d,p}^2\|u\|_{\infty,T}^2\big(1+\sqrt{\nu T}\big)^{2p+2}}\nonumber
	 \end{align}
	 so that if we slightly abuse the notation by denoting $C_{d,p}=2^6B_{d,p}^2$, we would have
	 \begin{align}
	 	M\leq T/T'=\frac{C_{d,p}T\big(1+\sqrt{\nu T}\big)^{2p+2}\|u\|_{\infty,T}^2}{\nu}~.\nonumber
	 \end{align}
	 Then by the same calculation given in the end of proof of Proposition \ref{prop. Gaussian bound of Omega_n},
	 \begin{align}
	 	D_{T'}:=\sum_{n=0}^\infty\widetilde{\rm Q}_{M}(n)~\Theta^n
	 	\leq \sum_{n=1}^\infty(3n)^M2^{-3n}
	 	\lesssim \left(\frac{1}{\ln2}\right)^{M+1}\left(\frac{M}{e}\right)^{M+\frac{1}{2}}
	 	\lesssim M^M~.\nonumber
	 \end{align}
	 Finally, this together with (\ref{estimate of Phi_0(0,t)},\ref{estimate of U_n with implicit constant}) imply the desired decaying estimate.
\end{proof}

\
\begin{proof}[Proof of (\ref{counterpart of the semi-group property 1},\ref{counterpart of the semi-group property 2})]
	Inequality (\ref{counterpart of the semi-group property 1}) is consequence of the elementary one: $$1+|x|\leq 1+|x-y|+|y|\leq\big(1+|x-y|\big)\big(1+|y|\big)~,$$ i.e. $\displaystyle\frac{1}{1+|x-y|}\leq \frac{1+|y|}{1+|x|}$~. So for any $a\in[0,\infty)$ and $p\in\mathbb R$, one sees that
	 \begin{align}
	 	\int_{\mathbb R^d}G_{at}(y)\frac{{\rm d}y}{\big(1+|x-y|\big)^p}
	 	&\leq\big(1+|x|\big)^{-p}\int_{\mathbb R^d}G_{at}(y)\big(1+|y|\big)^{p}{\rm d}y\nonumber\\
	 	&\lesssim_d\big(1+|x|\big)^{-p}\int_{\mathbb R^d}G_a(y)\big(1+\sqrt{t}~|y|\big)^{p}{\rm d}y\nonumber\\
	 	&\lesssim_{d,p,a}\big(1+\sqrt{t}\big)^p\big(1+|x|\big)^{-p}~.\nonumber
	 \end{align}
	 From now on, we should restrict the parameter $p\in[0,d+1]$ for (\ref{counterpart of the semi-group property 2}). We first note that the integral is obviously uniformly bounded for all $x\in\mathbb R^d$:
	 \begin{align}
	 	\int_{\mathbb R^d}\big(\sqrt{a}+|y|\big)^{-d-1}\frac{{\rm d}y}{\big(1+|x-y|\big)^p}\leq\int_{\mathbb R^d}\frac{{\rm d}y}{\big(\sqrt{a}+|y|\big)^{d+1}}\lesssim_{d}\frac{1}{\sqrt{a}}<\infty~.\nonumber
	 \end{align} 
	 So it suffices to prove the inequality (\ref{counterpart of the semi-group property 2}) for say $|x|\geq1$.
	 The strategy relies on splitting the convolution with respect to the regions:
	 \begin{align}
	 	\mathsf R_1(x)&:=\big\{y:~|y|\leq|x|/2\big\}~,\nonumber\\
	 	\mathsf R_2(x)&:=\big\{y:~|x-y|\leq|x|/2\big\}~,\nonumber\\
	 	\mathsf R_3(x)&:=\big\{y:~|x-y|>|x|/2~and~|x|/2<|y|\leq 2|x|\big\}~.\nonumber\\
	 	\mathsf R_4(x)&:=\big\{y:~|x-y|>|x|/2~and~|y|>2|x|\big\}~.\nonumber
	 \end{align}
	 So $\displaystyle\int_{\mathbb R^d}\big(\sqrt{a}+|y|\big)^{-d-1}\frac{{\rm d}y}{\big(1+|x-y|\big)^p}=\sum_{k=1}^4I_k(x)$ with
	 \begin{align}
	 	I_k(x):=\int_{\mathsf R_k(x)}\big(\sqrt{a}+|y|\big)^{-d-1}\frac{{\rm d}y}{\big(1+|x-y|\big)^p}~.\nonumber
	 \end{align}
	 For $y\in \mathsf R_1(x)$, we use the fact $|x-y|\geq |x|-|y|\geq |x|/2$ to write
	 \begin{align}
	 	I_1(x)\leq \frac{2^p}{\big(1+|x|\big)^p}\int_{\mathbb R^d}\frac{{\rm d}y}{\big(\sqrt{a}+|y|\big)^{d+1}}\lesssim_{d,p}\frac{1}{\sqrt{a}}\big(1+|x|\big)^{-p}~.\nonumber
	 \end{align}
	 For $y\in \mathsf R_2(x)$, we use the fact $|y|\geq |x|-|x-y|\geq |x|/2$ instead:
	 \begin{align}
	 	I_2(x)\leq\frac{2^{d+1}}{|x|^{d+1}}\int_{|x-y|\leq|x|/2}\frac{{\rm d}y}{\big(1+|x-y|\big)^p}\leq\frac{2^{d+1}}{|x|^{d+1}}\int_0^{|x|/2}(1+r)^{d-p-1}{\rm d}r~.\nonumber
	 \end{align}
	 The integral on the far right-hand-side is bounded by constant multiple of
	 \begin{equation}
	 	\left\{\begin{aligned}
	 		&\big(1+|x|/2\big)^{d-p}\quad p<d\\
	 		&\ln\left(1+|x|/2\right)\quad\ p=d\\
	 		&~const\quad\quad\quad\quad \ \ p>d
	 	\end{aligned}\right.\nonumber
	 \end{equation}
	 Since we restrict $0\leq p\leq d+1$, in any occasion we should have that
	 \begin{align}
	 	I_2(x)\lesssim_{d,p}\big(1+|x|\big)^{-p}~.\nonumber
	 \end{align}
	 For $y\in \mathsf R_3(x)$, we simply use $|y|>|x|/2$ and the fact that $\mathsf R_3(x)\subset\big\{y:~|x-y|\leq3|x|\big\}$ to write
	 \begin{align}
	 	I_3(x)\leq\frac{2^{d+1}}{|x|^{d+1}}\int_{|x-y|\leq3|x|}\frac{{\rm d}y}{\big(1+|x-y|\big)^p}~.\nonumber
	 \end{align}
	 And by the same calculation for $I_2(x)$, we would eventually arrive at
	 \begin{align}
	 	I_3(x)\lesssim_{d,p}\big(1+|x|\big)^{-p}~.\nonumber
	 \end{align}
	 For $y\in \mathsf R_4(x)$, we use the fact $|x-y|\geq |y|-|x|> |y|/2$ to write
	 \begin{align}
	 	I_4(x)\leq2^p\int_{|y|>2|x|}|y|^{-p-d-1}{\rm d}y=2^p\int_{2|x|}^{+\infty}r^{-p-2}{\rm d}r\lesssim_p|x|^{-p-1}\lesssim_p|x|^{-p}~,\quad |x|\geq1~.\nonumber
	 \end{align}
	 Putting altogether, we get (\ref{counterpart of the semi-group property 2}).
\end{proof}

\
\subsection{Decaying estimate of gradient of velocity}\label{subsec. decay estimate of grad u}

Since now the solution $u\in C_{loc}\big([0,T_{max});L^\infty\big)$,
%for every $\max\{d/p,d\}< q\leq\infty$, by BKM blow-up criterion \cite{BKM84} (which also applies to incompressible Navier-Stokes), there is no doubt that it has gradient $\nabla u\in L^1_{loc}\big([0,T_{max});L^\infty\big)$. (In fact, 
it is not hard to shown that $\nabla u\in L^1_{loc}\big([0,T_{max});L^\infty\big)$ (see Appendix \ref{Appendix: regularity of strong solutions} for a proof). This is consistent with the famous BKM blow-up criterion \cite{BKM84}. Also, the mild solution formula clearly holds:
\begin{align}
	\partial_j u(t)= G_{\nu t}\ast\partial_j u_0+ {\bf B}_t(\partial_j u\otimes u)+{\bf B}_t(u\otimes \partial_j u)~,\quad j=1,...,d\nonumber
\end{align}
in $L^1_{loc}\big([0,T_{max});L^\infty\big)$ and for $a.e.~(x,t)\in\mathbb R^d\times[0,T_{max})$.
One can apply the same procedure given in the previous sub-section to the sequence $\big\{{\rm  V}_n=({\rm  V}_n^{ij})_{d\times d}:n\in\mathbb N_0\big\}$ of tensor fields :
\begin{align}
	{\rm  V}_0^{~\cdot j}(t)
	&:=G_{\nu t}\ast\partial_j u_0\nonumber\\
	{\rm  V}_{n+1}^{~\cdot j}(t)
	&:={\rm  V}_0^{~\cdot j}(t)+{\bf B}_t\big({\rm  V}_n^{~\cdot j}\otimes u\big)+{\bf B}_t\big(u\otimes {\rm  V}_n^{~\cdot j}\big)~,\quad n\geq0\nonumber
\end{align}
and prove the same type of bound for $\nabla u$. 

\begin{comment}
We note here that although $\nabla u$ and the ${\rm  V}_n$'s do not belong to any of the $L^\infty_{loc}\big([0,T_{max});L^q\big)$ spaces, but by the same argument in the proof of proposition \ref{prop. convergence of U_n} one can still derive the uniform convergence for any $0<T<T_{max}$ :
\begin{align}
	\sup_{t\in[0,T]}\big\|{\rm  V}_{n+1}(t)-\nabla u(t)\big\|_{\infty}\leq Q_{m}(n)~\Lambda^{n+1}\sup_{t\in[0,T]}\big\|{\rm  V}_{0}(t)-\nabla u(t)\big\|_{\infty}\longrightarrow 0~,\quad as~n\rightarrow\infty\nonumber
\end{align}
since
\begin{align}
	\big\|{\rm  V}_{0}(t)-\nabla u(t)\big\|_{\infty}
	&\leq \sum_j\Big(\big\|{\bf B}_t(\partial_j u\otimes u)\big\|_\infty+\big\|{\bf B}_t(u\otimes \partial_j u)\big\|_\infty\Big)\nonumber\\
	&\lesssim_d\int_0^t\Big(\big\|(\nabla u\otimes u)(s)\big\|_\infty+\big\|(u\otimes\nabla u)(s)\big\|_\infty\Big)\frac{{\rm d}s}{\sqrt{t-s}}\nonumber\\
	&\lesssim_d\|u\|_{\infty,t}\int_0^t\|\nabla u(s)\|_\infty\frac{{\rm d}s}{\sqrt{t-s}}\nonumber
\end{align}
\end{comment}

\begin{proposition}\label{prop. polynomial bound for grad u}
	Under the same condition of Proposition \ref{prop. polynomial bound for velocity}, the following bound holds for any $0\leq p\leq d+1$ and all $(x,t)\in\mathbb R^d\times[0,T_{max})$:
	\begin{align}\label{polynomial bound for grad u}
		\big(1+|x|\big)^{p}\big|\nabla u(x,t)\big|\lesssim_{d,p} \frac{\big(1+\sqrt{\nu t}\big)^{p}}{\sqrt{\nu t}}\left[1+\frac{\mathfrak D_\nu(t)}{\sqrt{\nu~}}\right]\Big[t~\mathfrak D_\nu(t)^2\Big]^{t~\mathfrak D_\nu(t)^2}
	\end{align}
	where the increasing function $\mathfrak D_\nu(t)$ is given by (\ref{function D(t)}).
\end{proposition}
\noindent Note that we do not need any further assumption concerning regularity of $u_0$ , since we have
\begin{align}
	\big|{\rm V}_0(x,t)\big|
	&\leq\int_{\mathbb R^d}\big|\nabla G_{\nu t}(y)\big||u_0(x-y)|{\rm d}y\nonumber\\
	&\lesssim_{d}\frac{1}{\sqrt{\nu t}}\int_{\mathbb R^d} G_{2\nu t}(y)\frac{{\rm d}y}{\big(1+|x-y|\big)^p}\lesssim_{d,p}\frac{\big(1+\sqrt{\nu t}\big)^p}{\sqrt{\nu t}}\big(1+|x|\big)^{-p}~.\nonumber
\end{align}
The key reason that this can work without additional assumption on the initial value is the observation that the $t$-factor on the far right-hand-side of the inequality is an non-decreasing function of $t$ multiple of the integrable singularity $t\mapsto1/\sqrt{\nu t}$. However, one would have to raise additional assumption on the initial value if he/she expects to bound higher derivatives of the solution in the same spirit. The proof of Proposition \ref{prop. polynomial bound for grad u} falls into the same line of sub-section \ref{subsec. Decaying estimate}, which is combined by the following convergence proposition and proving the same bound (\ref{polynomial bound for grad u}) for ${\rm V}_n$. The latter is an almost verbatim repeat of the proof of Proposition \ref{prop. polynomial bound for U_n} , we shall leave it to the reader. We only prove Proposition \ref{prop. convergence of V_n} below.

\begin{proposition}\label{prop. convergence of V_n}
Let $d\geq2$ and also ${\rm  V}_n$ be given by (\ref{definition of U_n}). Then for any $0<T<T_{max}$~, we have the strong convergence ${\rm  V}_n\longrightarrow\nabla u$ in $L^1(0,T;L^\infty)$.	
\end{proposition}

\begin{proof}[Proof of Proposition \ref{prop. convergence of V_n}]
	The strategy is the same as before, but we need some different treatment in detailed estimates. We begin by defining the quantities for each $n\in\mathbb N_0$ :
	\begin{align}
		\Pi_n(t_1,t_2):=\int_{t_1}^{t_2}\big\|{\rm  V}_n(s)-\nabla u(s)\big\|_\infty{\rm d}s~,\quad 0\leq t_1\leq t_2\leq T~.\nonumber
	\end{align}
	Firstly, by Young's inequality and (\ref{L^1-bound of kernel mathcal K}) we have for any $t\in[0,T]$ that
	\begin{align}
		\int_{0}^{t}\big\|{\rm  V}_{n+1}(s)-\nabla u(s)\big\|_\infty{\rm d}s
		&\leq C_d \|u\|_{\infty,t}\int_{0}^{t}{\rm d}s\int_{0}^{s}\frac{\big\|{\rm  V}_{n}(s')-\nabla u(s')\big\|_\infty}{\sqrt{\nu(s-s')}}{\rm d}s'\nonumber\\
		&\leq C_d \|u\|_{\infty,t}\int_{0}^{t}\big\|{\rm  V}_{n}(s')-\nabla u(s')\big\|_\infty{\rm d}s'\int_{s'}^{t}\frac{{\rm d}s}{\sqrt{\nu(s-s')}}\nonumber\\
		&\leq 2C_d \|u\|_{\infty,t}\sqrt{\frac{t}{\nu}}\int_{0}^{t}\big\|{\rm  V}_{n}(s')-\nabla u(s')\big\|_\infty{\rm d}s'.\nonumber
	\end{align}
	Hence we should have that
	\begin{align}\label{convergence estimate of V_n 1}
		\Pi_n(0,t)\leq \left(2C_d\|u\|_{\infty,t}\sqrt{\frac{t}{\nu}}\right)^n\Pi_0(0,t)<\infty~,\quad n\in\mathbb N_0
	\end{align}
	as the same calculation would give
	\begin{align}
		\Pi_0(0,t):=\int_{0}^{t}\big\|{\rm  V}_{0}(s)-\nabla u(s)\big\|_\infty{\rm d}s\lesssim_d \|u\|_{\infty,t}\sqrt{\frac{t}{\nu}}\int_{0}^{t}\|\nabla u(s)\|_q{\rm d}s<\infty~.\nonumber
	\end{align}
	Now, for any $0\leq\tau'<t'\leq T$, we can write
	\begin{align}
	 	{\rm  V}_{n+1}^{i,j}(t')=~
	 	&{\rm  V}_{n+1}^{i,j}(\tau')+\left(G_{\nu t'}- G_{\nu\tau'}\right)\ast\partial_j u_0^i\nonumber\\
	 	&+\int_{\tau'}^{t'}\mathcal K_{\ell,k}^i\big(\nu(t'-s')\big)\ast\Big[-{\rm  V}_{n}^{\ell,j}(s')u^k(s')-u^\ell(s'){\rm  V}_{n}^{k,j}(s')\Big]{\rm d}s'\nonumber\\
	 	&+\int_0^{\tau'}\Big[\mathcal K_{\ell,k}^i\big(\nu(t'-s')\big)-\mathcal K_{\ell,k}^i\big(\nu(\tau'-s')\big)\Big]\ast\Big[-{\rm  V}_{n}^{\ell,j}(s')u^k(s')-u^\ell(s'){\rm  V}_{n}^{k,j}(s')\Big]{\rm d}s'~;%\label{time-difference of V_n+1}
	 	\nonumber\\
	 	\partial_j u^i(t')=~
	 	&\partial_j u^i(\tau')+\left(G_{\nu t'}-G_{\nu\tau'}\right)\ast \partial_j u_0^i\nonumber\\
	 	&+\int_{\tau'}^{t'}\mathcal K_{\ell,k}^i\big(\nu(t'-s')\big)\ast\big[-\partial_j u^\ell(s')u^k(s')-u^\ell(s')\partial_j u^k(s')\big]{\rm d}s'\nonumber\\
	 	&+\int_0^{\tau'}\left[\mathcal K_{\ell,k}^i\big(\nu(t'-s')\big)-\mathcal K_{\ell,k}^i\big(\nu(\tau'-s')\big)\right]\ast\big[-\partial_j u^\ell(s')u^k(s')-u^\ell(s')\partial_j u^k(s')\big]{\rm d}s'~.%\label{time-difference of grad u}
	 	\nonumber
	 \end{align}
	 Take the difference, then by Young's inequality and (\ref{L^1 estimate of mathcal K(t)-mathcal K(tau)}, \ref{L^1-bound of kernel mathcal K}) one derive that
	 \begin{align}
	 	\big\|{\rm  V}_{n+1}(t')-\nabla u(t')\big\|_\infty\leq ~
	 	&\big\|{\rm  V}_{n+1}(\tau')-\nabla u(\tau')\big\|_\infty+C_d\|u\|_{\infty,T}\int_{\tau'}^{t'}\frac{\big\|{\rm  V}_{n}(s')-\nabla u(s')\big\|_\infty}{\sqrt{\nu(t'-s')}}{\rm d}s'\nonumber\\
	 	&+C_d\|u\|_{\infty,T}\int_{0}^{\tau'}\bigg[\frac{1}{\sqrt{\nu(\tau'-s')}}-\frac{1}{\sqrt{\nu(t'-s')}}\bigg]\big\|{\rm  V}_{n}(s')-\nabla u(s')\big\|_\infty{\rm d}s'~.\nonumber
	 \end{align}
	 Next, we replace $t'=t+\tau'$ in the inequality for any $0\leq t\leq T-\tau'$ and then integrate the $\tau'$ variable on the interval $[0,\tau]$ with $0\leq\tau\leq T-t$. This would give that
	 \begin{align}
	 	&\int_{t}^{t+\tau}\big\|{\rm  V}_{n+1}(s)-\nabla u(s)\big\|_\infty{\rm d}s\nonumber\\
	 	\leq
	 	& \int_{0}^{\tau}\big\|{\rm  V}_{n+1}(\tau')-\nabla u(\tau')\big\|_\infty{\rm d}\tau'+C_d\|u\|_{\infty,T}\int_{0}^{\tau}{\rm d}\tau'\int_{0}^{t}\frac{\big\|{\rm  V}_{n}(s+\tau')-\nabla u(s+\tau')\big\|_\infty}{\sqrt{\nu(t-s)}}{\rm d}s\nonumber\\
	 	&+C_d\|u\|_{\infty,T}\int_{0}^{\tau}{\rm d}\tau'\int_{0}^{\tau'}\bigg[\frac{1}{\sqrt{\nu(\tau'-s)}}-\frac{1}{\sqrt{\nu(t+\tau'-s)}}\bigg]\big\|{\rm  V}_{n}(s)-\nabla u(s)\big\|_\infty{\rm d}s\nonumber\\
	 	\leq
	 	& \int_{0}^{\tau}\big\|{\rm  V}_{n+1}(s)-\nabla u(s)\big\|_\infty{\rm d}s+C_d\|u\|_{\infty,T}\int_{0}^{t}\frac{{\rm d}s}{\sqrt{\nu(t-s)}}\int_{s}^{s+\tau}\big\|{\rm  V}_{n}(s')-\nabla u(s')\big\|_\infty{\rm d}s'\nonumber\\
	 	&+C_d\|u\|_{\infty,T}\int_{0}^{\tau}\big\|{\rm  V}_{n}(s)-\nabla u(s)\big\|_\infty{\rm d}s\underbrace{\int_{s}^{\tau}\bigg[\frac{1}{\sqrt{\nu(\tau'-s)}}-\frac{1}{\sqrt{\nu(t+\tau'-s)}}\bigg]{\rm d}\tau'}_{=\frac{2}{\sqrt{\nu}}\big(\sqrt{\tau-s}+\sqrt{t}-\sqrt{t+\tau-s}\big)}\nonumber\\
	 	\leq
	 	& \int_{0}^{\tau}\big\|{\rm  V}_{n+1}(s)-\nabla u(s)\big\|_\infty{\rm d}s+2C_d\|u\|_{\infty,T}\sqrt{\frac{t}{\nu}}\int_{0}^{t+\tau}\big\|{\rm  V}_{n}(s)-\nabla u(s)\big\|_\infty{\rm d}s\nonumber\\
	 	&+2C_d\|u\|_{\infty,T}\sqrt{\frac{t}{\nu}}\int_{0}^{\tau}\big\|{\rm  V}_{n}(s)-\nabla u(s)\big\|_\infty{\rm d}s~.\nonumber
	 \end{align}
	 Hence we arrive at
	 \begin{align}
	 	\Pi_{n+1}(t,t+\tau)\leq ~
	 	&\Pi_{n+1}(0,\tau)+2C_d\|u\|_{\infty,T}\sqrt{\frac{t}{\nu}}~\Pi_{n}(t,t+\tau)\nonumber\\
	 	&+2C_d\|u\|_{\infty,T}\sqrt{\frac{t}{\nu}}~\Big[\Pi_{n}(0,\tau)+\Pi_{n}(0,t)\Big]~,\quad n\in\mathbb N_0~.\label{convergence estimate of V_n 2}
	 \end{align}
	 From now on, we choose a small enough $0<T'\leq T$ such that
	 \begin{align}
	 	\Lambda:=2C_d\|u\|_{\infty,T}\sqrt{\frac{T'}{\nu}}<1~.\nonumber
	 \end{align}
	 Then by (\ref{convergence estimate of V_n 1}) we obviously have
	 \begin{align}\label{convergence estimate of V_n 3}
	 	\Pi_{n}(0,T')\leq \Lambda^n~\Pi_{0}(0,T')~,\quad n\in\mathbb N_0~.
	 \end{align}
	 If $T'<T$, take $t=\tau=T'$ in (\ref{convergence estimate of V_n 2}), then inserting (\ref{convergence estimate of V_n 3}) would give
	 \begin{align}
	 	\Pi_{n+1}(T',2T')
	 	&\leq \Pi_{n+1}(0,T')+\Lambda\cdot\Pi_{n}(T',2T')+2\Lambda\cdot\Pi_{n}(0,T')\nonumber\\
	 	&\leq \Lambda\cdot\Pi_{n}(T',2T')+3\Lambda^{n+1}\Pi_{0}(0,T')~.\nonumber
	 \end{align}
	 This implies
	 \begin{align}
	 	\Pi_{n}(T',2T')\leq \Lambda^{n}~\Pi_{0}(T',2T')+3n\Lambda^{n}~\Pi_{0}(0,T')~,\nonumber
	 \end{align}
	 Combining with (\ref{convergence estimate of V_n 3}), one sees that
	 \begin{align}
	 	\Pi_{n}(0,2T')\leq \underbrace{(3n+1)}_{=:Q_1(n)}\Lambda^{n}~\Pi_{0}(0,2T')~,\quad n\in\mathbb N_0~.
	 \end{align}
	 To this end, we may assume by induction that
	 \begin{align}\label{convergence induction of V_n}
	 	\Pi_{n}\big(0,mT'\big)\leq Q_{m-1}(n)~\Lambda^{n}~\Pi_{0}\big(0,mT'\big)~,\quad n\in\mathbb N_0
	 \end{align}
	 holds for some positive integer $m\leq \left\lfloor T/T'\right\rfloor-1$ and some polynomial $Q_{m-1}$ of order $m-1$ with non-negative coefficients such that $Q_{m-1}(0)=1$. Then we can take $t=T'$ and $\tau=mT'$ in (\ref{convergence estimate of V_n 2}) and insert (\ref{convergence estimate of V_n 3}, \ref{convergence induction of V_n}) into the inequality, so that
	 \begin{align}
	 	\Pi_{n+1}\big(T',(m+1)T'\big)
	 	&\leq \Pi_{n+1}\big(0,mT'\big)+\Lambda\cdot\Pi_{n}\big(T',(m+1)T'\big)+\Lambda\cdot\Big[\Pi_n\big(0,mT'\big)+\Pi_{n}(0,T')\Big]\nonumber\\
	 	&\leq \Lambda\cdot\Pi_{n}\big(T',(m+1)T'\big)+\Big[Q_{m-1}(n+1)+Q_{m-1}(n)+1\Big]\Lambda^{n}~\Pi_{0}\big(0,mT'\big)~.\nonumber
	 \end{align}
	 If we define $\displaystyle P_{m}(n):=\sum_{k=0}^{n-1}\Big[Q_{m-1}(k+1)+Q_{m-1}(k)+1\Big]$ for $n\geq1$ and $P_{m}(0):=0$, then the above inequality would give that
	 \begin{align}
	 	\Pi_{n}\big(T',(m+1)T'\big)\leq \Lambda^n~\Pi_{0}\big(T',(m+1)T'\big)+P_{m}(n)~\Lambda^{n}~\Pi_{0}\big(0,mT'\big)\nonumber
	 \end{align}
	 which combining with (\ref{convergence estimate of V_n 3}) implies that
	 \begin{align}
	 	\Pi_{n}\big(0,(m+1)T'\big)\leq \underbrace{\Big[P_{m}(n)+1\Big]}_{=:Q_{m}(n)}\Lambda^n~\Pi_{0}\big(0,(m+1)T'\big)~,\quad n\in\mathbb N_0~.
	 \end{align}
	 Hence we have just proved:
	 \begin{align}
	 	\Pi_{n}\big(0,MT'\big)\leq Q_{M-1}(n)~\Lambda^{n}~\Pi_{0}\big(0,MT'\big)~,\quad n\in\mathbb N_0
	 \end{align}
	 for integer $M:=\left\lfloor T/T'\right\rfloor$. Finally, repeat one more time the above procedure by setting $t=T-MT'$ and $\tau=MT'$ in (\ref{convergence estimate of V_n 2}), one would arrive at
	 \begin{align}
	 	\Pi_{n}(t,T)\leq \Lambda^n~\Pi_{0}(t,T)+P_{M}(n)~\Lambda^{n}~\Pi_{0}\big(0,MT'\big)~.\nonumber
	 \end{align}
	 On the other hand, (\ref{convergence estimate of V_n 1}) implies $\Pi_{n}(0,t)\leq \Lambda^n~\Pi_{0}(0,t)$ with $t=T-MT'<T'$. Combining the two inequality then gives:
	 \begin{align}
	 	\Pi_{n}(0,T)\leq \Big[P_{M}(n)+1\Big]\Lambda^n~\Pi_{0}(0,T)~,\quad n\in\mathbb N_0\nonumber
	 \end{align}
	 which concludes the desired convergence as $0<\Lambda<1$.
\end{proof}

\
\subsection{Asymptotic expansion of velocity at infinity}\label{subsec. Asymptotic Expansion}
In this sub-section we shall return on concerning Gaussian concentrated initial value $u_0$, namely that satisfies (\ref{Gaussian bound u_0}). For convenience, we denote $\langle x\rangle:=1+|x|$. Previously, we have shown that the strong solution $u$ satisfies the following type estimate:
\begin{align}\label{decaying assumption on u}
	|u(x,t)|+\sqrt{\nu t}~|\nabla u(x,t)|\leq h(t)\langle x\rangle^{-p}~,\ \ (x,t)\in\mathbb R^d\times[0,T_{max})
\end{align}
with some non-decreasing function $h:[0,T_{max})\rightarrow \overline{\mathbb R_+}$ possibly relying on the exponent $0\leq p\leq d+1$. In this sub-section we will prove a generalized version of Theorem \ref{Thm. Velocity Asymptotic Expansion} which applies to a class of strong solutions. As one would see, the expression of function $h$ does not affact the final asymptotic expansion formula.\
 
\begin{proposition}\label{prop. generalized velocity expansion}
	For any strong solution $u$ to (\ref{NS velocity}) satisfying (\ref{decaying assumption on u}) for $\frac{d}{2}+1<p<\infty$  with any possible initial value $u_0\in L^\infty$, the following equality holds for any integer $1\leq N< 2p-d-1$ and $(x,t)\in\{x:|x|\geq1\}\times[0,T_{max})$:
	\begin{align}%\label{general velocity asymptotic expansion}
		u(x,t)=G_{\nu t}\ast u_0(x)-\sum_{|\alpha|=0}^N\frac{(-1)^{|\alpha|}}{\alpha !}\nabla\partial^\alpha{\rm K}_{i,j}(x)\int_0^t{\rm M}_{\alpha}^{i,j}(s){\rm d}s+{\rm R}_1(x,t)+{\rm R}_2(x,t)\nonumber
	\end{align}
	where ${\rm K}_{i,j}:=\partial_{i,j}^2\Gamma$ and the moment coefficient 
$$\displaystyle {\rm M}_{\alpha}^{i,j}(t):=\int_{\mathbb R^d}y^\alpha u^i(y,t)u^j(y,t){\rm d}y~.$$ The remainders satisfy
	\begin{align}
		&\sup_{t\in[0,T']}\big|{\rm R}_1(x,t)\big|=O\left(|x|^{-d-N-1}\right);\quad\sup_{t\in[0,T']}\big|{\rm R}_2(x,t)\big|=O\left(|x|^{-2p}\right)\quad when\ |x|\rightarrow\infty\nonumber
	\end{align}
	for any $[0,T']\subset[0,T_{max})$. 
\end{proposition}

\begin{remark}\label{remark. derivation of Thm Two}
	If we restrict ourself on $u_0\in L^\infty$ satisfying Gaussian bound (\ref{Gaussian bound u_0}) and set $p=d+1$ and $N=d$ in Proposition \ref{prop. generalized velocity expansion}, then one derive Theorem \ref{Thm. Velocity Asymptotic Expansion} immediately as we obviously have
	\begin{align}
		\sup_{t\in[0,T']}\big|G_{\nu t}\ast u_0(x)\big|\lesssim\sup_{t\in[0,T']} G_{\nu(\sigma+t)}(x)\lesssim(4\pi\sigma)^{-d/2}e^{-\frac{|x|^2}{4\nu(\sigma+T')}}~.\nonumber
	\end{align}
	The same conclusion also holds for $u_0\in L^\infty$ satisfying $\displaystyle\sup_{x\in\mathbb R^d}\Big(\langle x\rangle^{2d+2}\big|u_0(x)\big|\Big)<\infty$.
\end{remark}\

\par For the rest of this section, for notational simplicity, we would simplify the notation by writing $T=T_{max}$. Up to time $T$, the solution $u$ clearly satisfies the mild solution formula point-wisely:
\begin{align}%\label{mild solution formulation of u}
	u(x,t)=G_{\nu t}\ast u_0(x)+{\bf B}_t(u\otimes u)(x)~.\nonumber
\end{align}
By (\ref{decaying assumption on u}) and lemma \ref{lem. integration by parts}, one can write for $F=u\otimes u$ that
\begin{align}\label{decomposition of B(F)}
	{\bf B}_t(F)^j=-\underbrace{\int_0^t\partial_\ell G_{\nu (t-s)}\ast\left[ F^{\ell,j}(s)-\frac{\delta_{\ell,j}}{d}{\bf tr}(F(s))\right]{\rm d}s}_{=:{\bf C}_t(F)^j}
	-\underbrace{\int_0^t\partial_j G_{\nu (t-s)}\ast\left[{\rm K}_{\ell,k}\ast F^{\ell,k}(s)\right]{\rm d}s}_{=:{\bf D}_t(F)^j}
\end{align}
where ${\rm K}_{\ell,k}:=\partial_{\ell,k}^2\Gamma$~. From here, we are going to treat ${\bf C}_t(F)$ and ${\bf D}_t(F)$ separately. As one would see, the term ${\bf C}_t(F)$ only contributes $O_t(|x|^{-2p})$-decaying at infinity while the singular convolution term ${\bf D}_t(F)$ gives rise to the potential leading order.

\begin{proposition}\label{prop. estimate of C(F)}
	Let $F=u\otimes u$ where $u$ is a strong solution to (\ref{NS velocity}) satisfying (\ref{decaying assumption on u}) with initial value $u_0\in L^\infty$. Then the estimate
	\begin{align}\label{estimate of C(F)}
		\left|{\bf C}_t(F)(x)\right|\lesssim_{d,p} \left(\langle x\rangle^{-2p}+e^{-\frac{|x|^2}{17\nu t}}\right)\int_0^t\frac{h(s)^2}{\sqrt{\nu(t-s)}~}{\rm d}s
	\end{align}
	holds for $\forall (x,t)\in\mathbb R^d\times[0,T)$.
\end{proposition}

\begin{proof}[Proof of Proposition \ref{prop. estimate of C(F)}]
	For notational simplicity, we denote $\displaystyle g:=F-\frac{{\bf tr}(F)}{d}{\bf Id}=u\otimes u-\frac{|u|^2}{d}{\bf Id}$ so that
	\begin{align}
		{\bf C}_t(F)^j=\int_0^t\partial_\ell G_{\nu (t-s)}\ast g^{\ell,j}(s){\rm d}s.\nonumber
	\end{align}
	Now let $x\in\mathbb R^d$ be arbitrarily fixed. We would split the convolution 
	$$\displaystyle\left[\partial_\ell G_{\nu (t-s)}\ast g^{\ell,j}(s)\right](x)=\int_{\mathbb R^d}\partial_\ell G_{\nu (t-s)}(x-y) g^{\ell,j}(y,s){\rm d}y$$ 
	into integrals on region $J_1=J_1(x):=\big\{y:|x-y|\leq |x|/2\big\}$ and $J_1^c$. The principle is that in the far region $J_1$ the heat kernel does not improve decaying so the integral preserves polynomial decay; in the close region $J_2$ however, one can expect exponential decay.
	
	\par We begin with the fact that $\forall y\in J_1(x)$ we have $|y|\geq |x|-|x-y|\geq |x|/2$ and so the main contribution to decaying comes from $g$, i.e. by (\ref{decaying assumption on u}),
	\begin{align}
		|g(y,s)|\lesssim_d h(s)^2~\langle x\rangle^{-2p}~,\quad\quad \forall y\in J_1(x)~.\nonumber
	\end{align}
	Hence we derive that
	\begin{align}
		\left|\int_{J_1}\partial_\ell G_{\nu (t-s)}(x-y) g^{\ell,j}(y,s){\rm d}y\right|
		&\lesssim_d h(s)^2~|x|^{-2p}\int_{\mathbb R^d}\left|\nabla G_{\nu (t-s)}(y)\right|{\rm d}y\nonumber\\
		&\lesssim_d \frac{h(s)^2}{\sqrt{\nu(t-s)}~}~\langle x\rangle^{-2p}~.\label{J_1 estimate of C(F)}
	\end{align}
	For $y\in J_1(x)^c$, we simply use that $\|g(s)\|_\infty\lesssim_d h(s)^2$ and so
	\begin{align}
		\left|\int_{J_1^c}\partial_\ell G_{\nu (t-s)}(x-y) g^{\ell,j}(y,s){\rm d}y\right|
		&\lesssim_d h(s)^2\int_{|z|\geq |x|/2}\left|\nabla G_{\nu (t-s)}(z)\right|{\rm d}z\nonumber\\
		&\lesssim_d \frac{h(s)^2}{\sqrt{\nu(t-s)}~}\int_{|z|\geq |x|/2}\frac{|z|}{\sqrt{\nu(t-s)}}G_{\nu (t-s)}(z){\rm d}z\nonumber\\
		&\lesssim_d \frac{h(s)^2}{\sqrt{\nu(t-s)}~}~e^{-\frac{|x|^2}{17\nu t}}~.\label{J_1^c estimate of C(F)}
	\end{align}
	Combining (\ref{J_1 estimate of C(F)}) and (\ref{J_1^c estimate of C(F)}), we see that (\ref{estimate of C(F)}) holds.
\end{proof}\

\

\noindent The term 
$${\bf D}_t(F)=\int_0^t\nabla G_{\nu (t-s)}\ast\left[{\rm K}_{\ell,k}\ast F^{\ell,k}(s)\right]{\rm d}s$$
contains a convolution by the singular kernel ${\rm K}_{\ell,k}$ of Calder\'on-Zygmund type. We will first prove the following asymptotic expansion for the convolution $\left[{\rm K}_{\ell,k}\ast F^{\ell,k}(s)\right](y)$ (Proposition \ref{prop. asymptotic expansion of K-convolution}) and then derive the final expansion for ${\bf D}_t(F)(x)$. We would use the notations:
\begin{align}
	{\rm M}_{\alpha}^{i,j}(t)
	&:=\int_{\mathbb R^d}y^\alpha u^i(y,t)u^j(y,t){\rm d}y~,\quad \alpha\in \mathbb N_0^{d};\\
	{\rm A}_{N}(t)
	&:=\int_{\mathbb R^d}\langle y\rangle^N |u(y,t)|^2{\rm d}y~,\quad 0\leq N<2p-d.
\end{align}
Clearly $\left|{\rm M}_{\alpha}^{i,j}(t)\right|\leq {\rm A}_{|\alpha|}(t)<\infty$  for all $0\leq |\alpha|<2p-d$ by (\ref{decaying assumption on u}).

\
\begin{proposition}\label{prop. asymptotic expansion of K-convolution}
	Let $F=u\otimes u$ where $u$ is a strong solution to (\ref{NS velocity}) satisfying (\ref{decaying assumption on u}) with initial value $u_0\in L^\infty$ and $p>\frac{d}{2}+1$. Then for $|y|\leq 1/2$, we have the following estimate:
	\begin{align}\label{local estimate of K-convolution}
		\Big|\left[{\rm K}_{\ell,k}\ast F^{\ell,k}(s)\right](y)\Big|\lesssim_{d,p} h(s)\Big(h(s)+\|\nabla u(s)\|_\infty\Big)~;
	\end{align}
	and for $|y|> 1/2$, we have the expansion ( $\forall 1\leq N<2p-d-1$ ) :
	\begin{align}\label{asymptotic expansion of K-convolution}
		\left[{\rm K}_{\ell,k}\ast F^{\ell,k}(s)\right](y)=\sum_{|\alpha|=0}^N(-1)^{|\alpha|}\frac{~{\rm M}_{\alpha}^{\ell,k}(s)}{\alpha!}\partial^\alpha{\rm K}_{\ell,k}(y)+\mathcal R_N(y,s)+\mathcal I(y,s)
	\end{align}
	with remainders estimates:
	\begin{align}
		\big|\mathcal R_N(y,s)\big|
		&\lesssim_{d,p,N} {\rm A}_{N+1}(s)~\langle y\rangle^{-(d+N+1)}+h(s)^2~\langle y\rangle^{-2p}\ln\big(2\langle y\rangle\big)~,\label{remainder estimate 1}\\
		\big|\mathcal I(y,s)\big|
		&\lesssim_{d,p} \frac{h(s)^2}{\sqrt{\nu s}}~\langle y\rangle^{-2p}~.\label{remainder estimate 2}
	\end{align}
\end{proposition}

\begin{proof}[Proof of Proposition \ref{prop. asymptotic expansion of K-convolution}]\

	\par\noindent {$\bullet~{\bf The~Uniform~Bound~(\ref{local estimate of K-convolution}):}$} \
	\par We split the convolution 
	$$\displaystyle \left[{\rm K}_{\ell,k}\ast F^{\ell,k}(s)\right](y)=\int_{\mathbb R^d}{\rm K}_{\ell,k}(y-z) F^{\ell,k}(z,s){\rm d}z$$
	into the singular part $\big\{z:|y-z|\leq 1\big\}$ and the non-singular part $\big\{z:|y-z|> 1\big\}$. For the singular part, the integral is understood in the sense of principal value. We use the oddness of the kernel to write
	\begin{align}
		\int_{|z|=r}{\rm K}_{\ell,k}(z){\rm d}z=0~,\forall~r>0\nonumber
	\end{align}
	and so
	\begin{align}
		\int_{|y-z|\leq 1}{\rm K}_{\ell,k}(y-z) F^{\ell,k}(z,s){\rm d}z=\int_{|y-z|\leq 1}{\rm K}_{\ell,k}(y-z)\Big[ F^{\ell,k}(z,s)-F^{\ell,k}(y,s)\Big]{\rm d}z~.\nonumber
	\end{align}
	By that $\big|{\rm K}_{\ell,k}(y-z)\big|\lesssim |y-z|^{-d}$ and
	\begin{align}
		\big|F^{\ell,k}(z,s)-F^{\ell,k}(y,s)\big|
		&\leq |y-z|\cdot\|\nabla F(s)\|_\infty\nonumber\\
		&\lesssim |y-z|\cdot h(s)~\|\nabla u(s)\|_\infty~,\nonumber
	\end{align}
	we have
	\begin{align}
		\left|\int_{|y-z|\leq 1}{\rm K}_{\ell,k}(y-z) F^{\ell,k}(z,s){\rm d}z\right|\lesssim_d h(s)~\|\nabla u(s)\|_\infty~.\label{local estimate: singular}
	\end{align}
	For the non-singular part, we simply use (\ref{decaying assumption on u}) and $\big|{\rm K}_{\ell,k}(y-z)\big|\lesssim1$ to write
	\begin{align}
		\left|\int_{|y-z|>1}{\rm K}_{\ell,k}(y-z) F^{\ell,k}(z,s){\rm d}z\right|
		&\lesssim_d\int_{\mathbb R^d}| F(z,s)|{\rm d}z\nonumber\\
		&\lesssim_d h(s)^2\int_{\mathbb R^d}\frac{{\rm d}z}{\langle z\rangle^{2p}}\label{local estimate: non-singular}
	\end{align}
	where we need the constraint $p>d/2$ so that the integral on the far right-hand-side is finite. Combining (\ref{local estimate: singular}) and (\ref{local estimate: non-singular}) we derive (\ref{local estimate of K-convolution}), i.e.
	\begin{align}
		\Big|\left[{\rm K}_{\ell,k}\ast F^{\ell,k}(s)\right](y)\Big|\lesssim_{d,p} h(s)\Big(h(s)+\|\nabla u(s)\|_\infty\Big)~.\nonumber
	\end{align}
	\\
	{$\bullet~{\bf The~Expansion~(\ref{asymptotic expansion of K-convolution})~ and~Remainders~Estimates~(\ref{remainder estimate 1},\ref{remainder estimate 2}):}$}\
	\par The strategy is to split the convolution into the close region $J_1=J_1(y):=\big\{z:|z|<|y|/2\big\}$ where we can safely apply Taylor's expansion; the far-singular region $J_2:=J_1^c~\cap\big\{z:|z-y|\leq 1/4\big\}$ that contributes the remainder estimate (\ref{remainder estimate 2}); the far-non-singular region $J_3:=J_1^c~\cap\big\{z: 1/4<|z-y|\leq |y|/2\big\}$ and the tail region $J_4:=J_1^c~\cap\big\{z: |z-y|> |y|/2\big\}$ which only contribute $o\big(|y|^{-(2p-1)}\big)$-decay. Now we treat them one by one.\\
	
	\noindent{(I) The Close Region $J_1(y):=\big\{z:|z|<|y|/2\big\}$}
	\par For $z\in J_1(y)$, we clearly have $|y-z|> |y|/2$ and so we could write the Taylor expansion
	\begin{align}
		{\rm K}_{\ell,k}(y-z)=\sum_{|\alpha|=0}^N(-1)^{|\alpha|}\frac{\partial^\alpha{\rm K}_{\ell,k}(y)}{\alpha!}~z^\alpha+R^{\rm K}_N(y,z)\nonumber
	\end{align}
	with remainder
	\begin{align}
		\big| R^{\rm K}_N(y,z)\big|
		&\leq |z|^{N+1}\sum_{|\alpha|=N+1}\frac{1}{\alpha!}\sup_{0\leq\theta\leq1}\big|\partial^\alpha{\rm K}_{\ell,k}(y-\theta z)\big|\nonumber\\
		&\lesssim_{d,N} |z|^{N+1}\sup_{|w-y|\leq|z|}\big|\nabla^{N+1}{\rm K}_{\ell,k}(w)\big|\nonumber\\
		&\lesssim_{d,N} |z|^{N+1}\sup_{|w|\geq|y|/2}\big|\nabla^{N+1}{\rm K}_{\ell,k}(w)\big|\nonumber\\
		&\lesssim_{d,N} |z|^{N+1}\cdot|y|^{-(d+N+1)}\nonumber
	\end{align}
	where in the last two inequality we have used that $\big\{w:|w-y|\leq|z|\big\}\subset\big\{w:|w-y|\leq|y|/2\big\}\subset\big\{w:|w|\geq|y|/2\big\}$ for $z\in J_1(y)$ and that $\big|\nabla^{N+1}{\rm K}_{\ell,k}(w)\big|\lesssim |w|^{-(d+N+1)}$. Then if we denote
	\begin{align}
		\mathcal R^1_N(y,s)
		&:=\int_{J_1(y)}R^{\rm K}_N(y,z)F^{\ell,k}(z,s){\rm d}z\nonumber\\
		\mathcal R^2_N(y,s)
		&:=\sum_{|\alpha|=0}^N(-1)^{|\alpha|}\frac{\partial^\alpha{\rm K}_{\ell,k}(y)}{\alpha!}\int_{J_1(y)^c}z^\alpha F^{\ell,k}(z,s){\rm d}z\nonumber
	\end{align}
	we can write
	\begin{align}\label{expansion for J_1-K}
		\int_{J_1}{\rm K}_{\ell,k}(y-z) F^{\ell,k}(z,s){\rm d}z=\sum_{|\alpha|=0}^N(-1)^{|\alpha|}\frac{\partial^\alpha{\rm K}_{\ell,k}(y)}{\alpha!}\underbrace{\int_{\mathbb R^d}z^\alpha F^{\ell,k}(z,s){\rm d}z}_{={\rm M}_\alpha^{\ell,k}(s)}+\mathcal R^1_N(y,s)+\mathcal R^2_N(y,s)~.
	\end{align}
	For $\mathcal R^1_N$ we have clearly that
	\begin{align}
		\big|\mathcal R^1_N(y,s)\big|
		&\lesssim_{d,N}|y|^{-(d+N+1)}\int_{J_1(y)}|z|^{N+1}|F(z,s)|{\rm d}z\nonumber\\
		&\lesssim_{d,N}{\rm A}_{N+1}(s)\cdot|y|^{-(d+N+1)}\label{remainder estimate 1 for J_1-K}
	\end{align}
	Here we need the constraint $N<2p-d-1$ so that ${\rm A}_{N+1}(s)<\infty$.
	For $\mathcal R^2_N$ we give the tail estimate by (\ref{decaying assumption on u}):
	\begin{align}
		\left|\int_{J_1(y)^c}z^\alpha F^{\ell,k}(z,s){\rm d}z\right|
		&\leq h(s)^2\int_{|z|>|y|/2}|z|^{|\alpha|-2p}{\rm d}z\nonumber\\
		&\lesssim_d h(s)^2\cdot|y|^{d+|\alpha|-2p}~.\nonumber
	\end{align}
	Here we have used $d+|\alpha|-2p\leq d+N-2p<-1$ so that the integral in the first inequality converges. And in order that we can take at least $N=1$, we need $\displaystyle p>\frac{d}{2}+1$. Now, together with $\big|\partial^\alpha {\rm K}_{\ell,k}(y)\big|\lesssim |y|^{-(d+|\alpha|)}$, we see that
	\begin{align}
		\big|\mathcal R^1_N(y,s)\big|
		&\leq \sum_{|\alpha|=0}^N\frac{\big|\partial^\alpha{\rm K}_{\ell,k}(y)\big|}{\alpha!}\left|\int_{J_1(y)^c}z^\alpha F^{\ell,k}(z,s){\rm d}z\right|\nonumber\\
		&\lesssim_{d,N} h(s)^2\cdot|y|^{-2p}~.\label{remainder estimate 2 for J_1-K}
	\end{align}
	\\
	\noindent{(II) The Far-Singular Region $J_2:=J_1^c~\cap\big\{z:|z-y|\leq 1/4\big\}$}
	\par Since we have restricted ourself for $|y|>1/2$, the ball $\big\{z:|z-y|\leq 1/4\big\}\subset\big\{z:|z|\geq|y|/2\big\}=J_1(y)^c$ and hence $J_2(y)=\big\{z:|z-y|\leq 1/4\big\}$. Again we use the cancellation
	\begin{align}
		\int_{|z|=r}{\rm K}_{\ell,k}(z){\rm d}z=0~,\forall~r>0\nonumber
	\end{align}
	to write
	\begin{align}
		\mathcal I(y,s)
		:=&\int_{|y-z|\leq 1/4}{\rm K}_{\ell,k}(y-z) F^{\ell,k}(z,s){\rm d}z\label{definition of mathcal I}\\
		=&\int_{|y-z|\leq 1/4}{\rm K}_{\ell,k}(y-z)\Big[ F^{\ell,k}(z,s)-F^{\ell,k}(y,s)\Big]{\rm d}z~.\nonumber
	\end{align}
	Now, by (\ref{decaying assumption on u}),
	\begin{align}
		\big|F^{\ell,k}(z,s)-F^{\ell,k}(y,s)\big|
		&\leq |y-z|\cdot\sup_{0\leq\theta\leq1}\big|\nabla F\big(y+\theta(z-y),s\big)\big|\nonumber\\
		&\lesssim |y-z|\cdot \sup_{0\leq\theta\leq1}\big|u\big(y+\theta(z-y),s\big)\big|\cdot\sup_{0\leq\theta\leq1}\big|\nabla u\big(y+\theta(z-y),s\big)\big|\nonumber\\
		&\lesssim |y-z|\cdot\frac{h(s)^2}{\sqrt{\nu s}}\Big(|y|-|z-y|\Big)^{-2p}~,\nonumber
	\end{align}
	Together with $\big|{\rm K}_{\ell,k}(y-z)\big|\lesssim |y-z|^{-d}$, we derive
	\begin{align}
		\big|\mathcal I(y,s)\big|
		&\lesssim_d \frac{h(s)^2}{\sqrt{\nu s}}\int_{|y-z|\leq 1/4}|y-z|^{-d+1}\big(|y|-|z-y|\big)^{-2p}{\rm d}z \nonumber\\
		&\lesssim_d \frac{h(s)^2}{\sqrt{\nu s}}\int_{|y-z|\leq 1/4}|y-z|^{-d+1}\big(|y|-1/4\big)^{-2p}{\rm d}z \nonumber\\
		&\lesssim_d \frac{h(s)^2}{\sqrt{\nu s}}~|y|^{-2p}\label{remainder estimate mathcal I}
	\end{align}
	where in the last inequality we have used that $|y|-1/4>|y|/2$, i.e. $|y|>1/2$.\\
	
	\noindent{(III) The Far-Non-singular Region $J_3:=J_1^c~\cap\big\{z:1/4<|z-y|\leq |y|/2\big\}$}
	\par For $z\in \big\{z:1/4<|z-y|\leq |y|/2\big\}$ we obviously have $|z|\geq |y|-|z-y|\geq|y|/2$, i.e. $z\in J_1(y)^c$. So actually $J_3=\big\{z:1/4<|z-y|\leq |y|/2\big\}$. As the kernel is integrable there, we simply use (\ref{decaying assumption on u}) and that $\big|{\rm K}_{\ell,k}(y-z)\big|\lesssim |y-z|^{-d}$ to write
	\begin{align}
		\left|\int_{J_3(y)}{\rm K}_{\ell,k}(y-z) F^{\ell,k}(z,s){\rm d}z\right|
		&\lesssim_d h(s)^2 \int_{1/4<|y-z|\leq |y|/2}\langle z\rangle^{-2p}|y-z|^{-d}{\rm d}z\nonumber\\
		&\lesssim_d h(s)^2 \langle y\rangle^{-2p}\int_{1/4}^{|y|/2}\frac{{\rm d}r}{r}\nonumber\\
		&\lesssim_d h(s)^2 \langle y\rangle^{-2p}\ln\big(2|y|\big)~.\label{remainder estimate 2 for J_3}
	\end{align}\\
	\noindent{(IV) The Tail Region $J_4:=J_1^c~\cap\big\{z:|z-y|>|y|/2\big\}$}
	\par For $z\in J_4(y)$ we clearly have $\big|{\rm K}_{\ell,k}(y-z)\big|\lesssim |y-z|^{-d}\lesssim |y|^{-d}$. So together by (\ref{decaying assumption on u}),
	\begin{align}
		\left|\int_{J_4(y)}{\rm K}_{\ell,k}(y-z) F^{\ell,k}(z,s){\rm d}z\right|
		&\lesssim_d h(s)\cdot |y|^{-d}\int_{J_1(y)^c} |z|^{-2p}{\rm d}z\nonumber\\
		&\lesssim_d h(s)\cdot |y|^{-d}\int_{|y|/2}^{+\infty}r^{-2p+d-1}{\rm d}r\nonumber\\
		&\lesssim_d h(s)\cdot |y|^{-2p}~.\label{remainder estimate 2 for J_4}
	\end{align}\
	
	\par Finally, putting (\ref{expansion for J_1-K}-\ref{remainder estimate 2 for J_4}) together, if we denote
	\begin{align}
		\mathcal R_N(y,s):=\mathcal R_N^1(y,s)+\mathcal R_N^2(y,s)+\int_{J_3(y)}{\rm K}_{\ell,k}(y-z) F^{\ell,k}(z,s){\rm d}z+\int_{J_4(y)}{\rm K}_{\ell,k}(y-z) F^{\ell,k}(z,s){\rm d}z,\nonumber
	\end{align}
	we would have the expansion
	\begin{align}
		\int_{\mathbb R^d}{\rm K}_{\ell,k}(y-z) F^{\ell,k}(z,s){\rm d}z=\sum_{|\alpha|=0}^N(-1)^{|\alpha|}\frac{{\rm M}_\alpha^{\ell,k}(s)}{\alpha!}\partial^\alpha{\rm K}_{\ell,k}(y)+\mathcal R_N(y,s)+\mathcal I(y,s)\nonumber
	\end{align}
	for any $1\leq N<2p-d-1$, and for $|y|> 1/2$ the remainders estimates are given by (\ref{remainder estimate mathcal I}) and
	\begin{align}
		\big|\mathcal R_N(y,s)\big|\lesssim_{d,N}{\rm A}_{N+1}(s)~\langle y\rangle^{-(d+N+1)}+h(s)^2~\langle y\rangle^{-2p}\ln\big(2\langle y\rangle\big)~.\nonumber
	\end{align}
	This ends the proof.
\end{proof}\

\noindent Now with Proposition \ref{prop. asymptotic expansion of K-convolution} in hand, we can write by (\ref{asymptotic expansion of K-convolution}):
\begin{align}
	{\bf D}_t(F)
	=&\int_0^t\nabla G_{\nu (t-s)}\ast\left[{\rm K}_{\ell,k}\ast F^{\ell,k}(s)\right]{\rm d}s\nonumber\\
	=&\sum_{|\alpha|=0}^N\frac{(-1)^{|\alpha|}}{\alpha!}\int_0^t{\rm M}_{\alpha}^{\ell,k}(s)\left[\int_{|y|>1/2}\nabla G_{\nu (t-s)}(x-y)\partial^\alpha{\rm K}_{\ell,k}(y)\right]{\rm d}s\nonumber\\
	&+\underbrace{\int_0^t\int_{|y|>1/2}\nabla G_{\nu (t-s)}(x-y) \mathcal R_N(y,s){\rm d}y{\rm d}s}_{=:{\rm R}_1(x,t)}\nonumber\\
	&+\underbrace{\int_0^t\int_{|y|>1/2}\nabla G_{\nu (t-s)}(x-y) \mathcal I(y,s){\rm d}y{\rm d}s}_{=:{\rm I}(x,t)}\nonumber\\
	&+\underbrace{\int_0^t\int_{|y|\leq1/2}\nabla G_{\nu (t-s)}(x-y) \left[{\rm K}_{\ell,k}\ast F^{\ell,k}(s)\right](y){\rm d}y{\rm d}s}_{=:{\rm S}(x,t)}~.\label{D(F) decomposition}
\end{align}
The following lemma gives decaying estimates of the last three terms. (From here, we use the monotonicity of function $h$.)

\begin{lemma}\label{lem. remainders decaying}
	The following estimates hold for $|x|\geq1$ and $0<t<T$:
	\begin{align}
		\big|{\rm R}_1(x,t)\big|
		&\lesssim_{d,N}\left(\langle x\rangle^{-d-N-1}+e^{-\frac{|x|^2}{17\nu t}}\right)\left(\frac{1}{\sqrt{\nu t}}\int_0^t{\rm A}_{N+1}(s){\rm d}s+h(t)^2\sqrt{\frac{t}{\nu}}\right)~,\label{remainder decaying 1}\\
		\big|{\rm I}(x,t)\big|
		&\lesssim_{d}\bigg(\langle x\rangle^{-2p}+e^{-\frac{|x|^2}{17\nu t}}\bigg)\frac{h(t)^2}{\nu}~,\label{remainder decaying 2}\\
		\big|{\rm S}(x,t)\big|
		&\lesssim_{d,p} e^{-\frac{|x|^2}{17\nu t}}\big(1+\sqrt{\nu t}\big)\frac{h(t)^2}{\nu}~.\label{remainder decaying 3}
	\end{align}
\end{lemma}

\begin{proof}[Proof of Lemma \ref{lem. remainders decaying}]
	(\ref{remainder decaying 1},\ref{remainder decaying 2}) can be proved in very similar way to Proposition \ref{prop. estimate of C(F)}. For example, one splits the convolution 
	$$\displaystyle{\rm R}_1(x,t)=\int_{\mathbb R^d}\nabla G_{\nu (t-s)}(x-y) \mathcal R_N(y,s){\bf 1}_{(1/2,\infty)}(y){\rm d}y$$
	is split into integrals on the close region $J_1(x)=\big\{y:|y-x|\leq|x|/2\big\}$ and the far region $J_1(x)^c$. Note that by (\ref{remainder estimate 1}) and the fact that $N<2p-d-1$, i.e. $2p>d+N+1$, one could write
	\begin{align}
		\big|\mathcal R_N(y,s){\bf 1}_{(1/2,\infty)}(y)\big|\lesssim_{d,N}\left[{\rm A}_{N+1}(s)+h(s)^2\right]~\langle y\rangle^{-(d+N+1)}~.\nonumber
	\end{align}
	So for the close region where $\mathcal R_N(y,s)$ contributes to the decaying,
	\begin{align}
		&\left|\int_{J_1(x)}\nabla G_{\nu (t-s)}(x-y) \mathcal R_N(y,s){\bf 1}_{(1/2,\infty)}(y){\rm d}y\right|\nonumber\\
		\lesssim_{d,N} &\langle x\rangle^{-(d+N+1)}\int_0^t\frac{{\rm A}_{N+1}(s)+h(s)^2}{\sqrt{\nu(t-s)}~}{\rm d}s\nonumber\\
		\lesssim_{d,N} &\langle x\rangle^{-(d+N+1)}\left(\frac{1}{\sqrt{\nu t}}\int_0^t{\rm A}_{N+1}(s){\rm d}s+h(t)^2\sqrt{\frac{t}{\nu}}\right).\nonumber
	\end{align}
	For the far region $J_1(x)^c$ where the Gaussian kernel would contribute exponential decaying, the same treatment in (\ref{J_1^c estimate of C(F)}) would give that
	\begin{align}
		&\left|\int_{J_1(x)^c}\nabla G_{\nu (t-s)}(x-y) \mathcal I(y,s){\bf 1}_{(1/2,\infty)}(y){\rm d}y\right|\nonumber\\
		\lesssim_{d} &e^{-\frac{|x|^2}{17\nu t}}\int_0^t\frac{{\rm A}_{N+1}(s)+h(s)^2}{\sqrt{\nu(t-s)}~}{\rm d}s\nonumber\\
		\lesssim_d &e^{-\frac{|x|^2}{17\nu t}}\left(\frac{1}{\sqrt{\nu t}}\int_0^t{\rm A}_{N+1}(s){\rm d}s+h(t)^2\sqrt{\frac{t}{\nu}}\right).\nonumber
	\end{align}
	Thus we see that (\ref{remainder decaying 1}) holds.	 Argue in the same way by using the estimate instead:
	\begin{align}
		\big|\mathcal I(y,s){\bf 1}_{(1/2,\infty)}(y)\big|
		&\lesssim_d \frac{h(s)^2}{\sqrt{\nu s}}\langle y\rangle^{-2p}~,\nonumber
	\end{align}
	then one shows that
	\begin{align}
		\big|{\rm I}(x,t)\big|
		&\lesssim_{d}\bigg(\langle x\rangle^{-2p}+e^{-\frac{|x|^2}{17\nu t}}\bigg)\int_0^t\frac{h(s)^2{\rm d}s}{\nu\sqrt{s(t-s)}~}\nonumber\\
		&\lesssim_{d}\bigg(\langle x\rangle^{-2p}+e^{-\frac{|x|^2}{17\nu t}}\bigg)\frac{h(t)^2}{\nu}~.\nonumber
	\end{align}

	\par The term ${\rm S}(x,t)$ would inherit Gaussian decay from the kernel $\nabla G_{\nu (t-s)}(x-y)$ once we restrict that $|x|\geq1$ since the convolution integral region is $|y|\leq1/2$. That is, by the fact that $|x-y|\geq |x|-1/2\geq |x|/2$ for $|y|\leq1/2$ and $|x|\geq1$, we can write:
	\begin{align}
		G_{\nu (t-s)}(x-y)
		&\lesssim_d e^{-\frac{16}{17}\frac{|x-y|^2}{4\nu(t-s)}}G_{17\nu (t-s)}(x-y)\lesssim_d e^{-\frac{|x|^2}{17\nu t}}G_{17\nu (t-s)}(x-y)\nonumber	
	\end{align}
	so that ( recalling the uniform bound (\ref{local estimate of K-convolution}) )
	\begin{align}
		\big|{\rm S}(x,t)\big|
		&\lesssim_{d,p} \int_0^t\frac{h(s)\big[h(s)+\|\nabla u(s)\|_\infty\big]}{\sqrt{\nu(t-s)}}\left(\int_{|y|\leq1/2}\frac{|x-y|}{\sqrt{\nu(t
		-s)}} G_{\nu (t-s)}(x-y){\rm d}y\right){\rm d}s\nonumber\\
		&\lesssim_{d,p} e^{-\frac{|x|^2}{17\nu t}}\int_0^t\left(1+\frac{1}{\sqrt{\nu s}}\right)\frac{h(s)^2{\rm d}s}{\sqrt{\nu(t-s)}~}\nonumber\\
		&\lesssim_{d,p} e^{-\frac{|x|^2}{17\nu t}}\big(1+\sqrt{\nu t}\big)\frac{h(t)^2}{\nu}\nonumber
	\end{align}
	where in the second inequality we have used that $\sqrt{\nu s}\|\nabla u(s)\|_\infty\leq h(s)$ which is a consequence of (\ref{decaying assumption on u}).
\end{proof}\

\noindent Now, by Lemma \ref{lem. remainders decaying}, we see that
\begin{align}
	{\bf D}_t(F)(x)
	=&\sum_{|\alpha|=0}^N\frac{(-1)^{|\alpha|}}{\alpha!}\int_0^t{\rm M}_{\alpha}^{\ell,k}(s)\left[\int_{|y|>1/2}\nabla G_{\nu (t-s)}(x-y)\partial^\alpha{\rm K}_{\ell,k}(y){\rm d}y\right]{\rm d}s\nonumber\\
	&+{\rm R}_1(x,t)+{\rm I}(x,t)+{\rm S}(x,t)\nonumber
\end{align}
such that for any $[0,T']\subset[0,T)$ we have
\begin{align}
	&\sup_{t\in[0,T']}\big|{\rm R}_1(x,t)\big|=O\big(|x|^{-d-N-1}\big),\quad when~|x|\rightarrow\infty;\label{remainder decaying order 1}\\
	&\sup_{t\in[0,T']}\big|{\rm I}(x,t)\big|+\sup_{t\in[0,T']}\big|{\rm S}(x,t)\big|=O\big(|x|^{-2p}\big),\quad when~|x|\rightarrow\infty.\label{remainder decaying order of I and S}
\end{align}
To this end, everything left is just to expand the convolution
\begin{align}%\label{definition of G^K}
	G^{\rm K}(x,\tau)=\big(G^{\rm K}\big)^\alpha_{\ell,k}(x,\tau) :=\int_{|y|>1/2}\nabla G_{\nu\tau}(x-y)\partial^\alpha{\rm K}_{\ell,k}(y){\rm d}y~.\nonumber
\end{align}
We shall prove the following expansion lemma where a delicate cancellation happens so that we have a very clean expansion.

\begin{lemma}\label{lem. expansion of G^K}
	The following equality holds:
	\begin{align}\label{expansion of G^K}
		\big(G^{\rm K}\big)^\alpha_{\ell,k}(x,\tau)=\nabla\partial^\alpha{\rm K}_{\ell,k}(x)+{\rm  r}^\alpha_{\ell,k}(x,\tau)
	\end{align}
	with the remainder estimate for $|x|\geq1$ and $0\leq\tau\leq t$ :
	\begin{align}\label{remainder estimate for G^K}
		\big|{\rm  r}(x,\tau)\big|\lesssim_{d,N,N'}\frac{(\nu t)^{N'/2}}{|x|^{d+N'+1}}+\frac{e^{\frac{d\sqrt{\nu t}}{|x|}}}{|x|}~e^{-\frac{|x|^2}{17\nu t}},\quad \forall N'\in\mathbb N.
	\end{align}
\end{lemma}\

\noindent With this lemma in hand, if we denote
\begin{align}
	{\rm R}_2(x,t):={\rm I}(x,t)+{\rm S}(x,t)+\sum_{|\alpha|=0}^N\frac{(-1)^{|\alpha|}}{\alpha!}\int_0^t{\rm M}_{\alpha}^{\ell,k}(s)~{\rm  r}^\alpha_{\ell,k}(x,t-s){\rm d}s~,\label{definition of Remainder R_2}
\end{align}
we would have that
\begin{align}\label{final expansion for D(F)}
	{\bf D}_t(F)(x)=\sum_{|\alpha|=0}^N\frac{(-1)^{|\alpha|}}{\alpha !}\nabla\partial^\alpha{\rm K}_{i,j}(x)\int_0^t{\rm M}_{\alpha}^{i,j}(s){\rm d}s+{\rm R}_1(x,t)+{\rm R}_2(x,t)
\end{align}
with remainders satisfying (\ref{remainder decaying order 1}) and
\begin{align}
	\sup_{t\in[0,T']}\big|{\rm R}_2(x,t)\big|=O\big(|x|^{-2p}\big),\quad when~|x|\rightarrow\infty.\label{remainder decaying order 2}
\end{align}
(\ref{remainder decaying order 2}) is a consequence of (\ref{remainder decaying order of I and S}), (\ref{definition of Remainder R_2}) and (\ref{remainder estimate for G^K}). Finally, putting together (\ref{mild solution formulation of u}), (\ref{decomposition of B(F)}), (\ref{estimate of C(F)}), (\ref{final expansion for D(F)}), (\ref{remainder decaying order 1}) and (\ref{remainder decaying order 2}), we see that Proposition \ref{prop. generalized velocity expansion} holds. So now it just remains to prove Lemma \ref{lem. expansion of G^K}.

\begin{proof}[Proof of Lemma \ref{lem. expansion of G^K}]
	The key principle is similar to the proof of Proposition \ref{prop. estimate of C(F)}. We should split as
	\begin{align}
		G^{\rm K}(x,\tau)
		&=\int_{|x-y|>1/2}\nabla G_{\nu\tau}(y)\partial^\alpha{\rm K}_{\ell,k}(x-y){\rm d}y\nonumber\\
		&=\underbrace{\int_{\Lambda_1(x)}\nabla G_{\nu\tau}(y)\partial^\alpha{\rm K}_{\ell,k}(x-y){\rm d}y}_{=:\mathcal M^\alpha_{\ell,k}(x,\tau)}+\underbrace{\int_{\Lambda_2(x)}\nabla G_{\nu\tau}(y)\partial^\alpha{\rm K}_{\ell,k}(x-y){\rm d}y}_{=:\mathcal T^\alpha_{\ell,k}(x,\tau)}
	\end{align}
	where $\Lambda_1(x):=\big\{y:|y|<|x|/2~and~|x-y|>1/2\big\}$ and $\Lambda_2(x):=\big\{y:|y|\geq|x|/2~and~|x-y|>1/2\big\}$. Then in the tail region $\Lambda_2(x)$ one could expect that the Gaussian kernel dominates decaying. So by that $\big|\partial^\alpha{\rm K}_{\ell,k}(x-y)\big|\lesssim1$ for $|x-y|>1/2$, we derive:
	\begin{align}
		\left|\mathcal T^\alpha_{\ell,k}(x,\tau)\right|
		&\lesssim_{d,\alpha}\int_{|y|\geq|x|/2}\frac{|y|}{\nu\tau}G_{\nu\tau}(y){\rm d}y\nonumber\\
		&\lesssim_{d,\alpha}e^{-\frac{|x|^2}{17\nu\tau}}\int_{|y|\geq|x|/2}\frac{|y|}{\nu\tau}G_{17\nu\tau}(y){\rm d}y\nonumber\\
		&\lesssim_{d,\alpha}|x|^{-1}e^{-\frac{|x|^2}{17\nu t}}~,\quad 0\leq\tau\leq t.\label{estimate of mathcal T}
	\end{align}
	For the main part $\mathcal M^\alpha_{\ell,k}(x,\tau)$, we may apply Taylor's expansion
	\begin{align}
		\partial^\alpha{\rm K}_{\ell,k}(x-y)=\sum_{|\beta|=0}^{N'}(-1)^{|\beta|}\frac{\partial^{\alpha+\beta}{\rm K}_{\ell,k}(x)}{\beta!}~y^\beta+\widetilde R^{\alpha,\ell,k}_{N'}(x,y)\nonumber
	\end{align}
	with remainder
	\begin{align}
		\left| \widetilde R^{\alpha,\ell,k}_{N'}(x,y)\right|
		&\leq |y|^{N'+1}\sum_{|\beta|=N'+1}\frac{1}{\beta!}\sup_{0\leq\theta\leq1}\big|\partial^{\alpha+\beta}{\rm K}_{\ell,k}(x-\theta y)\big|\nonumber\\
		&\lesssim_{d,\alpha,N'} |y|^{N'+1}\cdot|x|^{-(d+|\alpha|+N'+1)}\nonumber
	\end{align}
	where in the last inequality we have used $\big|\partial^{\alpha+\beta}{\rm K}_{\ell,k}(w)\big|\lesssim |w|^{-(d+|\alpha|+|\beta|+1)}$ and that $|x-\theta y|\geq|x|-|y|>|x|/2$ for $y\in\Lambda_1(x)$. Then we see that
	\begin{align}
		\mathcal M^\alpha_{\ell,k}(x,\tau)
		=&\sum_{|\beta|=0}^{N'}(-1)^{|\beta|}\frac{\partial^{\alpha+\beta}{\rm K}_{\ell,k}(x)}{\beta!}\int_{\mathbb R^d}y^\beta\nabla G_{\nu\tau}(y){\rm d}y+\underbrace{\int_{\Lambda_1(x)}\nabla G_{\nu\tau}(y)\widetilde R^{\alpha,\ell,k}_{N'}(x,y){\rm d}y}_{=:\mathcal J^{\alpha,\ell,k}_1(x,\tau)}\nonumber\\
		&+\underbrace{\sum_{|\beta|=0}^{N'}(-1)^{|\beta|}\frac{\partial^{\alpha+\beta}{\rm K}_{\ell,k}(x)}{\beta!}\int_{\Lambda_1(x)^c}y^\beta\nabla G_{\nu\tau}(y){\rm d}y}_{=:\mathcal J^{\alpha,\ell,k}_2(x,\tau)}~.\nonumber
	\end{align}
	For $\mathcal J_1$ we use the above remainder estimate to write
	\begin{align}
		\left|\mathcal J^{\alpha,\ell,k}_1(x,\tau)\right|
		&\lesssim_{d,\alpha,N'} |x|^{-(d+|\alpha|+N'+1)}\int_{\mathbb R^d}\frac{|y|^{N'+2}}{\nu\tau}G_{\nu\tau}(y){\rm d}y\nonumber\\
		&\lesssim_{d,\alpha,N'} |x|^{-(d+|\alpha|+N'+1)}(\nu t)^{N'/2}~,\quad 0\leq\tau\leq t.\label{estimate of mathcal J_1}
	\end{align}
	For $\mathcal J_2$ we need the constraint $|x|\geq1$ which implies $|x-y|\geq|x|-|y|>|x|/2\geq 1/2$ for $y\in\Lambda_1(x)=\big\{y:|y|<|x|/2~and~|x-y|>1/2\big\}$ and so $\Lambda_1(x)=\big\{y:|y|<|x|/2\big\}$. Then we only need to do the tail estimate:
	\begin{align}
		\left|\int_{\Lambda_1(x)^c}y^\beta\nabla G_{\nu\tau}(y){\rm d}y\right|
		&\leq \int_{|y|\geq|x|/2}\frac{|y|^{|\beta|+1}}{\nu\tau}G_{\nu\tau}(y){\rm d}y\nonumber\\
		&\lesssim |x|^{-1}e^{-\frac{|x|^2}{17\nu\tau}}\int_{|y|\geq|x|/2}\frac{|y|^{|\beta|+2}}{\nu\tau}G_{17\nu\tau}(y){\rm d}y\nonumber\\
		&\lesssim_{d,|\beta|} |x|^{-1}(\nu t)^{|\beta|/2}e^{-\frac{|x|^2}{17\nu t}}~,\quad 0\leq\tau\leq t\nonumber
	\end{align}
	which implies:
	\begin{align}
		\left|\mathcal J^{\alpha,\ell,k}_2(x,\tau)\right|
		&\lesssim_{d,N'}|x|^{-1}e^{-\frac{|x|^2}{17\nu t}}\sum_{|\beta|=0}^{N'}\frac{1}{\beta!}(\nu t)^{|\beta|/2}|x|^{-(d+|\alpha|+|\beta|)}\nonumber\\
		&\lesssim_{d,N'}|x|^{-(d+|\alpha|+1)}e^{-\frac{|x|^2}{17\nu t}}\sum_{n=0}^{N'}\left(\frac{\sqrt{\nu t}}{|x|}\right)^n\underbrace{\sum_{|\beta|=n}\frac{1}{\beta!}}_{=d^n/n!}\nonumber\\
		&\lesssim_{d,N'}|x|^{-(d+|\alpha|+1)}e^{\frac{d\sqrt{\nu t}}{|x|}}~e^{-\frac{|x|^2}{17\nu t}}~.\label{estimate of mathcal J_2}
	\end{align}
	Here in the first inequality we have used the fact that for any multi-index $\alpha$, the quantity $|x|^{d+|\alpha|}\big|\partial^\alpha {\rm K}(x)\big|$ can be bounded from above by a constant only depends on $d$ and $N=|\alpha|$ but not necessarily on $\alpha$. This is simply because
	\begin{align}
		|x|^{d+|\alpha|}\big|\partial^\alpha {\rm K}(x)\big|=\left|\partial^\alpha {\rm K}\left(\frac{x}{|x|}\right)\right|\leq\sup_{\omega\in\mathbb S^{d-1}}\big|\partial^\alpha {\rm K}(\omega)\big|\leq\max_{|\alpha|=N}\sup_{\omega\in\mathbb S^{d-1}}\big|\partial^\alpha {\rm K}(\omega)\big|~.\nonumber
	\end{align}
	Now if we denote
	\begin{align}
		{\rm  r}^\alpha_{\ell,k}:=\mathcal T^\alpha_{\ell,k}+\mathcal J^{\alpha,\ell,k}_1+\mathcal J^{\alpha,\ell,k}_2~,\nonumber
	\end{align}
	(\ref{estimate of mathcal T}), (\ref{estimate of mathcal J_1}) and (\ref{estimate of mathcal J_2}) would give that
	\begin{align}
		\left|{\rm  r}^\alpha_{\ell,k}(x,\tau)\right|\lesssim_{d,\alpha,N'}\frac{(\nu t)^{N'/2}}{|x|^{d+N'+1}}+\frac{e^{\frac{d\sqrt{\nu t}}{|x|}}}{|x|}~e^{-\frac{|x|^2}{17\nu t}}~,\quad (~|x|\geq1,~0\leq\tau\leq t~),\nonumber
	\end{align}
	and
	\begin{align}\label{formal expansion of G^K}
		\big(G^{\rm K}\big)^\alpha_{\ell,k}(x,\tau)
		=\sum_{|\beta|=0}^{N'}(-1)^{|\beta|}\frac{\partial^{\alpha+\beta}{\rm K}_{\ell,k}(x)}{\beta!}\int_{\mathbb R^d}y^\beta\nabla G_{\nu\tau}(y){\rm d}y+{\rm  r}^\alpha_{\ell,k}(x,\tau)~.
	\end{align}
	Now it remains to compute the moments $\displaystyle {\bf M}_\beta(\tau):=\int_{\mathbb R^d}y^\beta\nabla G_{\nu\tau}(y){\rm d}y$. By integration by parts, one has that
	\begin{align}
		{\bf M}_\beta(\tau)^{j}:=\int_{\mathbb R^d}y^\beta\partial_j G_{\nu\tau}(y){\rm d}y=-\beta_j\int_{\mathbb R^d}y^{\beta-e_j} G_{\nu\tau}(y){\rm d}y\nonumber
	\end{align}
	where $e_j\in\mathbb N_0^d$ denotes the multi-index whose components are zero except that the $j$-th component equals one. Then it is clear that the right-hand-side varnishes unless $\beta-e_j=2\gamma$ for some $\gamma\in\mathbb N_0^d$. In this case, one has that
	\begin{align}
		{\bf M}_{2\gamma+e_j}(\tau)^{j}
		&=-(2\gamma_j+1)\int_{\mathbb R^d}y^{2\gamma} G_{\nu\tau}(y){\rm d}y\nonumber\\
		&=-(2\gamma_j+1)(\nu\tau)^{|\gamma|}\int_{\mathbb R^d}y^{2\gamma} G_{1}(y){\rm d}y\nonumber\\
		&=-(2\gamma_j+1)(\nu\tau)^{|\gamma|}~\frac{(2\gamma)!}{\gamma!}~.\nonumber
	\end{align}
	And for $\beta=2\gamma+e_j$, we obviously have $\beta!=(2\gamma)!(2\gamma_j+1)$ so that
	\begin{align}
		\frac{1}{\beta!}{\bf M}_\beta(\tau)^{j}=-\frac{(\nu\tau)^{|\gamma|}}{\gamma!}~.\nonumber
	\end{align}
	Then for any $N'\geq1$, we have
	\begin{align}
		\left[\big(G^{\rm K}\big)^\alpha_{\ell,k}(x,\tau)\right]^j
		&=\sum_{|\beta|=0}^{N'}(-1)^{|\beta|}\frac{1}{\beta!}{\bf M}_\beta(\tau)^{j}~\partial^{\alpha+\beta}{\rm K}_{\ell,k}(x)+\left[{\rm  r}^\alpha_{\ell,k}(x,\tau)\right]^j	\nonumber\\
		&=\sum_{n=0}^{\frac{N'-1}{2}}\sum_{|\gamma|=n}(-1)^{2|\gamma|+1}\left[-\frac{(\nu\tau)^{|\gamma|}}{\gamma!}\right]\partial^{\alpha+2\gamma+e_j}{\rm K}_{\ell,k}(x)+\left[{\rm  r}^\alpha_{\ell,k}(x,\tau)\right]^j\nonumber\\
		&=\sum_{n=0}^{\frac{N'-1}{2}}\frac{(\nu\tau)^{|\gamma|}}{n!}\underbrace{\sum_{|\gamma|=n}\frac{n!}{\gamma!}\partial^{2\gamma}}_{=\Delta^n}\partial_j\partial^\alpha{\rm K}_{\ell,k}(x)+\left[{\rm  r}^\alpha_{\ell,k}(x,\tau)\right]^j\nonumber\\
		&=\partial_j\partial^\alpha{\rm K}_{\ell,k}(x)+\left[{\rm  r}^\alpha_{\ell,k}(x,\tau)\right]^j~,\quad (x\neq0)\nonumber
	\end{align}
	where in the last equality we have used the fact that $\Delta {\rm K}_{\ell,k}(x)=0$ for $x\neq0$. This ends the proof.
\end{proof}\

\appendix
\renewcommand{\appendixname}{Appendix~\Alph{section}}
\renewcommand{\theequation}{A.\arabic{equation}}

\section{An integration by parts lemma}\label{Appendix: integration by parts lemma}

In this appendix we prove Lemma \ref{lem. integration by parts}. For convenience, we recall the lemma here. Let $\Gamma$ denote the fundamental solution of Laplacian ``$-\Delta$".

\begin{lemma}%\label{lem. integration by parts}
	Let $F\in C^0_{loc,0}\big(\mathbb R^d;(\mathbb R^d)^{\otimes2}\big)$ satisfies $\nabla F\in L^\infty_{loc}$ and
	\begin{align}
		\int_{\mathbb S^{d-1}}|\nabla F(x-R\omega)|{\rm d}\omega=\left\{
		\begin{aligned}
			&o\left(1/R\right)\quad\quad\quad d\geq3\\
			\\
			&o\left(\frac{1}{R\ln R}\right)\quad\ d=2
		\end{aligned}\right.
		\quad\quad as~R\rightarrow+\infty
	\end{align}
	for $\forall x\in\mathbb R^d$, then the equality holds point-wisely:
	\begin{align}
		\Gamma\ast\partial_k\partial_\ell F^{\ell,k}=\partial_k\partial_\ell\Gamma\ast F^{\ell,k}-\frac{1}{d}{\bf tr}(F)~.\nonumber
	\end{align}
	Here ${\bf tr}(F)$ denotes the trace of $F$.
\end{lemma}

\begin{proof}
	Recall that
	\begin{align}\label{decaying assumption on grad F}
		\Gamma(x)=\left\{
		\begin{aligned}
			&\frac{|x|^{2-d}}{(d-2)|\mathbb S^{d-1}|}\quad\quad d\geq3\\
			\\
			&-\frac{1}{2\pi}\ln|x|\quad\quad\quad d=2
		\end{aligned}\right.
	\end{align}
	We only show the case $d\geq3$ as case $d=2$ is very similar. Now fix $x\in\mathbb R^d$ arbitrarily and write
	\begin{align}
		\Gamma\ast\partial_k\partial_\ell F^{\ell,k}(x)=\int_{\mathbb R^d}\Gamma(y)\partial_k\partial_\ell F^{\ell,k}(x-y){\rm d}y=\lim_{R\rightarrow+\infty}\int_{R^{-1}\leq|y|\leq R}\Gamma(y)\partial_k\partial_\ell F^{\ell,k}(x-y){\rm d}y.\nonumber
	\end{align}
	As everything is regular on annulus $\{R^{-1}\leq|y|\leq R\}$, one can safely apply integration by parts on this region. The boundary terms are:
	\begin{align}
		\underbrace{-R^{-(d-1)}\int_{\mathbb S^d}\Gamma(R^{-1}\omega)\omega_k\partial_\ell F^{\ell,k}(x-R^{-1}\omega){\rm d}\omega}_{=:s_R(x)}
		+\underbrace{R^{(d-1)}\int_{\mathbb S^d}\Gamma(R\omega)\omega_k\partial_\ell F^{\ell,k}(x-R\omega){\rm d}\omega}_{=:S_R(x)}\nonumber
	\end{align}
	where we have
	\begin{align}
		|s_R(x)|\lesssim R^{-1}\int_{\mathbb S^{d-1}}|\nabla F(x-R^{-1}\omega)|{\rm d}\omega\longrightarrow0~,\quad R\rightarrow+\infty\nonumber
	\end{align}
	by Dominated Convergence as $\nabla F\in L^\infty_{loc}$ by assumption; and 
	\begin{align}
		|S_R(x)|\lesssim R^{-1}\int_{\mathbb S^{d-1}}|\nabla F(x-R\omega)|{\rm d}\omega\longrightarrow0~,\quad R\rightarrow+\infty\nonumber
	\end{align}
	by condition (\ref{decaying assumption on grad F}). Hence, the boundary terms all varnish in the limit and we have
	\begin{align}
		\Gamma\ast\partial_k\partial_\ell F^{\ell,k}(x)=\lim_{R\rightarrow+\infty}\int_{R^{-1}\leq|y|\leq R}\partial_k\Gamma(y)\partial_\ell F^{\ell,k}(x-y){\rm d}y~.\nonumber
	\end{align}
	Similarly, we apply integration by parts once more on the annulus and write down the boundary terms:
	\begin{align}
		\underbrace{-R^{-(d-1)}\int_{\mathbb S^d}\partial_k\Gamma(R^{-1}\omega)\omega_\ell F^{\ell,k}(x-R^{-1}\omega){\rm d}\omega}_{=:\widetilde s_R(x)}
		+\underbrace{R^{(d-1)}\int_{\mathbb S^d}\partial_k\Gamma(R\omega)\omega_\ell F^{\ell,k}(x-R\omega){\rm d}\omega}_{=:\widetilde S_R(x)}.\nonumber
	\end{align}
	Then by that $\displaystyle\partial_k\Gamma(x)=-|\mathbb S^{d-1}|^{-1}\frac{x^k}{|x|^d}$, we derive
	\begin{align}
		|S_R(x)|\lesssim \int_{\mathbb S^{d-1}}|F(x-R\omega)|{\rm d}\omega\longrightarrow0~,\quad R\rightarrow+\infty\nonumber
	\end{align}
	again by Dominated Convergence as $F$ is clearly bounded; and
	\begin{align}
		\widetilde s_R(x)=|\mathbb S^{d-1}|^{-1}\int_{\mathbb S^d}\omega_\ell\omega_k F^{\ell,k}(x-R^{-1}\omega){\rm d}\omega\longrightarrow F^{\ell,k}(x)~|\mathbb S^{d-1}|^{-1}\int_{\mathbb S^d}\omega_\ell\omega_k {\rm d}\omega~,\quad as~R\rightarrow+\infty\nonumber
	\end{align}
	again by Dominated Convergence. To this end, it remains to compute the integral $I_{\ell,k}:=\int_{\mathbb S^d}\omega_\ell\omega_k {\rm d}\omega$. Clearly $I_{\ell,k}=0$ if $\ell\neq k$ by symmetry of $\mathbb S^d$. When $\ell=k$, we again have $I_{\ell,\ell}=I_{1,1}$ for any $\ell=1,...,d$ by symmetry, so
	\begin{align}
		I_{1,1}=\frac{1}{d}\sum_{\ell=1}^d\int_{\mathbb S^d}\omega_\ell^2 {\rm d}\omega=\frac{|\mathbb S^{d-1}|}{d}.\nonumber
	\end{align}
	Putting together, we have
	\begin{align}
		\lim_{R\rightarrow+\infty}\widetilde s_R(x)=F^{\ell,k}(x)\frac{\delta_{\ell,k}}{d}=\frac{1}{d}{\bf tr}(F(x))~.\nonumber
	\end{align}
	As the integration by parts gives
	\begin{align}
		\int_{R^{-1}\leq|y|\leq R}\partial_k\Gamma(y)\partial_\ell F^{\ell,k}(x-y){\rm d}y
		&=-\int_{R^{-1}\leq|y|\leq R}\partial_k\Gamma(y)\partial_{y_{\ell}} F^{\ell,k}(x-y){\rm d}y\nonumber\\
		&=-\left[\widetilde s_R(x)+\widetilde S_R(x)\right]+\int_{R^{-1}\leq|y|\leq R}\partial_\ell\partial_k\Gamma(y) F^{\ell,k}(x-y){\rm d}y~,\nonumber
	\end{align}
	we finally arrive at
	\begin{align}
		\Gamma\ast\partial_k\partial_\ell F^{\ell,k}(x)=-\frac{1}{d}{\bf tr}(F(x))+\partial_k\partial_\ell\Gamma\ast F^{\ell,k}(x)~.\nonumber
	\end{align}
\end{proof}

\
\section{Regularity of strong solutions to (\ref{NS velocity})}
\label{Appendix: regularity of strong solutions}

\begin{proposition}\label{prop. regularity of bounded solutions}
	Let $u\in L^\infty(0,T;L^\infty)$ be the mild solution to (\ref{NS velocity}) generated by $u_0\in L^\infty$. Then the solution has gradient $\nabla u\in L^1\big(0,T;L^\infty\big)$.
\end{proposition}

\begin{proof}[Proof of Proposition \ref{prop. regularity of bounded solutions}]

%We show that at least up to some small $T>0$, the $L^p$-strong solution generated by $u_0$ satisfying the polynomial bound (\ref{polynomial bound u_0}) has gradient $\nabla u\in L^1\big(0,T;L^q\big)$ for all $\max\{d/q,d\}< q\leq\infty$. 
We prove the conclusion for some small enough $0<T'\leq T$ that only rely on the constant $K:=\max\big\{\|u_0\|_\infty, \|u\|_{\infty,T}\big\}$~. This can be easily seen by Picard iteration where one has the inductive estimates
\begin{align}
	\big\|u^{(n)}\big\|_{\infty,T'}
	%&\leq \sup_{t\in[0,T]}\big\|G_{\nu t}\ast u_0\big\|_\infty+\sup_{t\in[0,T]}\big\|{\bf B}_t\big(u^{(n-1)}\otimes u^{(n-1)}\big)\big\|_\infty\nonumber\\
	&\leq \|u_0\|_\infty+C_d\sup_{t\in[0,T']}\int_0^t\big\|u^{(n-1)}(s)\big\|_\infty^2\frac{{\rm d}s}{\sqrt{\nu(t-s)}}\nonumber\\
	&\leq \|u_0\|_\infty+C_d\sqrt{\frac{T'}{\nu}}~\big\|u^{(n-1)}\big\|_{\infty,T'}^2\nonumber\\
	\int_0^{T'}\big\|\partial_j u^{(n)}(t)\big\|_{\infty}{\rm d}t
	&\leq \int_0^{T'}\big\|G_{\nu t}\ast\partial_j u_0\big\|_\infty{\rm d}t+2C_d\sqrt{\frac{T'}{\nu}}~\big\|u^{(n-1)}\big\|_{\infty,T'}\int_0^{T'}\big\|\partial_j u^{(n-1)}(t)\big\|_\infty{\rm d}t\nonumber\\
	&\leq 2\sqrt{\frac{T'}{\nu}}~\|u_0\|_\infty+2C_d\sqrt{\frac{T'}{\nu}}~\big\|u^{(n-1)}\big\|_{\infty,T'}\int_0^{T'}\big\|\partial_j u^{(n-1)}(t)\big\|_\infty{\rm d}t\nonumber
\end{align}
where the first estimate follows by Young's inequality and (\ref{L^1-bound of kernel mathcal K}); the second by the following estimate
\begin{align}
	\int_0^{T'}\big\|{\bf B}_t(u\otimes v)\big\|_\infty{\rm d}t 
	&\lesssim_d \int_0^{T'}{\rm d}t\int_0^t\|(u\otimes v)(s)\|_\infty\frac{{\rm d}s}{\sqrt{\nu(t-s)}}\nonumber\\
	&\lesssim_d \|u\|_{\infty,T'}\int_0^{T'}{\rm d}s\int_s^{T'}\|v(s)\|_\infty\frac{{\rm d}t}{\sqrt{\nu(t-s)}}\nonumber\\
	&\lesssim_d \sqrt{\frac{T'}{\nu}}~\|u\|_{\infty,T'}\int_0^{T'}\|v(s)\|_\infty{\rm d}s~.\label{appen. non-linear esti.}
\end{align}
Now one can choose small enough $T'>0~s.t.~4C_d\sqrt{\frac{T'}{\nu}}\|u_0\|_\infty\leq \frac{1}{2}$, i.e.
\begin{align}
	T'\leq\frac{\nu}{(8C_d K)^2}~.\label{appen. constraint of T' 1}
\end{align}
so that by the classical argument, 
\begin{align}\label{appen. 1}
	\big\|u^{(n)}\big\|_{\infty,T'}\leq \frac{\|u_0\|_\infty}{1-4C_d\sqrt{\frac{T'}{\nu}}\|u_0\|_\infty}\leq \frac{4}{3}\|u_0\|_\infty
\end{align}
and so
\begin{align}
	\int_0^{T'}\big\|\partial_j u^{(n)}(t)\big\|_{\infty}{\rm d}t
	&\leq 2\sqrt{\frac{T'}{\nu}}~\|u_0\|_\infty+4C_d\sqrt{\frac{T'}{\nu}}~\|u_0\|_\infty\int_0^{T'}\big\|\partial_j u^{(n-1)}(t)\big\|_\infty{\rm d}t\nonumber\\
	&\leq ...\leq\frac{2\sqrt{\frac{T'}{\nu}}~\|u_0\|_\infty}{1-4C_d\sqrt{\frac{T'}{\nu}}\|u_0\|_\infty}\leq 4\sqrt{\frac{T'}{\nu}}~\|u_0\|_\infty ~.\label{appen. 2}
\end{align}
We define the norm
\begin{align}
	\|u\|_{X(T')}:=\|u\|_{\infty,T'}+\int_0^{T'}\|\nabla u(t)\|_{\infty}{\rm d}t~.\nonumber
\end{align}
For the difference $w_{n+1}:=u^{(n+1)}-u^{(n)}={\bf B}\big(w_n\otimes u^{(n)}\big)+{\bf B}\big(u^{(n-1)}\otimes w_n\big)$, one derive by (\ref{appen. non-linear esti.}-\ref{appen. 2}) :
\begin{align}
	\|w_{n+1}\|_{\infty,T'}
	&\lesssim_d \sqrt{\frac{T'}{\nu}}~\Big(\big\|u^{(n)}\big\|_{\infty,T'}+\big\|u^{(n-1)}\big\|_{\infty,T'}\Big)\big\|w_n\big\|_{X(T')}\nonumber\\
	%\lesssim_d \sup_{t\in[0,T']}\left[\int_0^t\big\|\big(u^{(n)}\otimes w_n\big)(s)\big\|_\infty^2\frac{{\rm d}s}{\sqrt{\nu(t-s)}}+\int_0^t\big\|\big(u^{(n-1)}\otimes w_n\big)(s)\big\|_\infty^2\frac{{\rm d}s}{\sqrt{\nu(t-s)}}\right]\nonumber\\
	&\lesssim_d\sqrt{\frac{T'}{\nu}}~\|u_0\|_\infty\big\|w_n\big\|_{X(T')}~;\nonumber\\
	\int_0^{T'}\|\nabla w_{n+1}(t)\|_{\infty}{\rm d}t
	&\lesssim_d \sqrt{\frac{T'}{\nu}}~\Big(\big\|u^{(n)}\big\|_{X(T')}+\big\|u^{(n-1)}\big\|_{X(T')}\Big)\|w_n\|_{X(T')}\nonumber\\
	&\lesssim_d \sqrt{\frac{T'}{\nu}}\left(1+\sqrt{\frac{T'}{\nu}}\right)\|u_0\|_\infty\|w_n\|_{X(T')}\nonumber
\end{align}
which implies
\begin{align}
	\|w_n\|_{X(T')}\leq \left[B_d\sqrt{\frac{T'}{\nu}}\left(1+\sqrt{\frac{T'}{\nu}}\right)\|u_0\|_\infty\right]^n\|w_0\|_{X(T')}~.\label{appen. 3}
\end{align}
Note that $\|w_0\|_{X(T')}<\infty$. Now to make $B_d\sqrt{\frac{T'}{\nu}}\left(1+\sqrt{\frac{T'}{\nu}}\right)\|u_0\|_\infty<1$, by (\ref{appen. constraint of T' 1}), it suffices to set $T'>0$ such that $B_d\sqrt{\frac{T'}{\nu}}\left(1+\frac{1}{8C_d K}\right)K=\frac{1}{2}$, i.e.
\begin{align}
	T'=\frac{\nu}{4B_d^2K^2\left(1+\frac{1}{8C_d K}\right)^2}~.\nonumber
\end{align}
Then (\ref{appen. 3}) shows that the sequence $\big\{u^{(n)}\big\}$ is Cauchy in $L^\infty(0,T';L^\infty)\cap L^1\big(0,T';\dot{W}^{1,\infty}\big)$ and so converges strongly to a limit $\bar{u}_1$ in the space. As $\bar{u}_1,u\big|_{[0,T']}\in L^\infty(0,T';L^\infty)$ are solutions to (\ref{NS velocity}) with the same initial value, $weak$-$strong$-$uniqueness$ implies that $\bar{u}_1=u\big|_{[0,T']}$ and hence $\nabla u\in L^1\big(0,T';L^\infty\big)$.

\par As the choice of $T'$ only relies the constant $K$ and we obviously have $\|u(T'/2)\|_\infty\leq K$, one can apply the above argument to initial value $u(T'/2)$ and obtain $\bar{u}_2\in L^\infty(0,T';L^\infty)\cap L^1\big(0,T';\dot{W}^{1,\infty}\big)$ such that $\bar{u}_2(t)=u(T'/2+t)$ for all $t\in[0,T']$. This then implies that $\nabla u\in L^1\big(0,3T'/2;L^\infty\big)$. To this end, one can apply the same procedure repeatedly until we cover the whole interval $[0,T]$~.
%\par By BKM blow-up criterion \cite{BKM84} (which also applies to incompressible Navier-Stokes), if $T_{max}>0$ is the maximal lifespan of the strong solution $u$, then $\nabla u\in L^1_{loc}\big([0,T_{max});L^q\big)$

\end{proof}

%\begin{thebibliography}{99}
%\bibliographystyle{alpha}
%\bibitem{}
%\end{thebibliography}
\bibliographystyle{alpha}
\bibliography{Paper2ref}
%\reftitle{References}

\end{document}